\documentclass[12pt]{amsart}

\usepackage[matrix,arrow,curve,frame]{xy}
\usepackage{amsmath,amsthm,amssymb,enumerate}
\usepackage{latexsym}
\usepackage{amscd}
\usepackage[colorlinks=false]{hyperref}
\usepackage{euscript}

\newtheorem{thm}[subsection]{Theorem}

\newtheorem{prop}[subsection]{Proposition}

\newtheorem{cor}[subsection]{Corollary}
\newtheorem{lemma}[subsection]{Lemma}
\newtheorem{conj}[subsection]{Conjecture}
\newtheorem{remark}[subsection]{Remark}

\theoremstyle{definition}

\numberwithin{equation}{section}

\def\cT{{\cal T}}

\def\cG{{\cal G}}

\def\cA{{\cal A}}

\def\ra{\rightarrow}

\def\W{{\bf W}}

\def\cA{{\mathcal A}}
\def\cB{{\mathcal B}}
\def\cC{{\mathcal C}}
\def\cD{{\mathcal D}}
\def\cE{{\mathcal E}}
\def\cF{{\mathcal F}}
\def\cG{{\mathcal G}}
\def\cH{{\mathcal H}}
\def\cI{{\mathcal I}}
\def\cJ{{\mathcal J}}

\def\cL{{\mathcal L}}

\def\cS{{\mathcal S}}
\def\cT{{\mathcal T}}

\def\cV{{\mathcal V}}
\def\cW{{\mathcal W}}

\def\gg{{\mathfrak g}}

\def\gl{{\mathfrak l}}

\def\go{{\mathfrak o}}
\def\gp{{\mathfrak p}}

\def\gs{{\mathfrak s}}

\newfont{\german}{eufm10}

\allowdisplaybreaks
\begin{document}
\pagestyle{plain}

\title
{Universal two-parameter even spin $\cW_{\infty}$-algebra}

\author{Shashank Kanade} 
\address{Department of Mathematics, University of Denver}
\email{shashank.kanade@du.edu}

\author{Andrew R. Linshaw} 
\address{Department of Mathematics, University of Denver}
\email{andrew.linshaw@du.edu}
\thanks{A. L. is supported by Simons Foundation Grant \#318755.
S. K. is supported by a startup grant provided by University of Denver. We thank T. Creutzig for the proof of Theorem \ref{thomas} and T. Arakawa for the proof of Remark \ref{arakawa}, and for many interesting discussions. We also thank T. Prochazka for pointing out to us that the even spin algebra is not invariant under the change of variables $\lambda\mapsto -\lambda$.}


{\abstract \noindent  We construct the unique two-parameter vertex algebra which is freely generated of type $\cW(2,4,6,\dots)$, and generated by the weights $2$ and $4$ fields. Subject to some mild constraints, all vertex algebras of type $\cW(2,4,\dots, 2N)$ for some $N$, can be obtained as quotients of this universal algebra. This includes the $B$ and $C$ type principal $\cW$-algebras, the $\mathbb{Z}_2$-orbifolds of the $D$ type principal $\cW$-algebras, and many others which arise as cosets of affine vertex algebras inside larger structures. As an application, we classify all coincidences among the simple quotients of the $B$ and $C$ type principal $\cW$-algebras, as well as the $\mathbb{Z}_2$-orbifolds of the $D$ type principal $\cW$-algebras. Finally, we use our classification to give new examples of principal $\cW$-algebras of $B$, $C$, and $D$ types, which are lisse and rational.}

\maketitle

\section{Introduction} \label{section:intro}

Let  $\gg$ be a simple, finite-dimensional Lie algebra over $\mathbb{C}$, and let $f \in \gg$ be a nilpotent element. The affine $\cW$-algebras $\cW^k(\gg,f)$ at level $k \in \mathbb{C}$ associated to $\gg$ and $f$ are important examples of vertex algebras in both the physics and mathematics literature. The first example other than the Virasoro algebra is the Zamolodchikov $\cW_3$-algebra \cite{Zam}, which corresponds to $\gs\gl_3$ with its principal nilpotent element $f_{\text{prin}}$. For an arbitrary $\gg$, the definition of $\cW^k(\gg,f_{\text{prin}})$ via quantum Drinfeld-Sokolov reduction was given by Feigin and Frenkel in \cite{FF1}. For an arbitrary nilpotent $f$, the definition of $\cW^k(\gg,f)$ is due to Kac, Roan, and Wakimoto \cite{KRW}, and is a generalization of the Drinfeld-Sokolov reduction.

The algebras $\cW^k(\gg,f_{\text{prin}})$ are closely related to the classical $\cW$-algebras which arose in the context of integrable hierarchies of soliton equations in works of Adler, Gelfand, Dickey, Drinfeld, and Sokolov \cite{Ad,GD,Di,DS}. The KdV hierarchy, which corresponds to the Virasoro algebra, was generalized by Drinfeld and Sokolov to an integrable hierarchy associated to any simple Lie algebra. The corresponding classical $\cW$-algebras are Poisson vertex algebras, and can be obtained as quasi-classical limits of affine $\cW$-algebras \cite{FBZ}. For a general nilpotent $f$, $\cW^k(\gg,f)$ can also be viewed as a chiralization of the finite $\cW$-algebra $\cW^{\text{fin}}(\gg,f)$ \cite{DSKII}, since it is related to $\cW^{\text{fin}}(\gg,f)$ via the Zhu functor \cite{Z}.

Let $\cW_k(\gg,f)$ denote the simple quotient of $\cW^k(\gg,f)$ by its maximal proper graded ideal. In the case $f = f_{\text{prin}}$, it was conjectured by Frenkel, Kac and Wakimoto \cite{FKW} and proven by Arakawa \cite{AIV,AV} that for a nondegenerate admissible level $k$, $\cW_k(\gg,f_{\text{prin}})$ is lisse (or $C_2$-cofinite) and rational. These are known as {\it minimal models}, and provide a large family of new rational vertex algebras that generalize the Virasoro minimal models \cite{GKO}. In fact, there are many other known lisse, rational $\cW$-algebras for other nilpotents, not all of which are at admissible levels; see for example \cite{AII,AMI,Kaw,KW,CLIV}.

\subsection{$\cW_{\infty}$-algebras}
For $n\geq 3$, $\cW^k(\gs\gl_n, f_{\text{prin}})$ is of type $\cW(2,3,\dots, n)$, meaning that it has a minimal strong generating set consisting of one field in each of these conformal weights. For different values of $n$, these structures are distinct and there are no nontrivial homomorphisms of one-parameter vertex algebras $$\cW^k(\gs\gl_n, f_{\text{prin}}) \ra \cW^{\ell}(\gs\gl_m, f_{\text{prin}}),\qquad  n \neq m.$$ However, it was conjectured in the physics literature \cite{YW,BK,B-H,BS,GG,ProI,ProII,PR} and recently proven by the second author in \cite{LI}, that there exists a unique two-parameter vertex algebra of type $\cW(2,3,\dots)$, denoted by $\cW_{\infty}[\mu]$, which interpolates between all the algebras $\cW^k(\gs\gl_n, f_{\text{prin}})$ in the following sense. The structure constants of $\cW_{\infty}[\mu]$ are continuous functions of the central charge $c$ and the parameter $\mu$, and if we set $\mu = n$, there is a truncation at weight $n+1$ that allows all fields in weights $d \geq n+1$ to be eliminated in the simple quotient of $\cW_{\infty}[\mu]$. This quotient is isomorphic to $\cW^k(\gs\gl_n, f_{\text{prin}})$ as a one-parameter vertex algebra. In the quasi-classical limit, the existence of a Poisson vertex algebra of type $\cW(2,3,\dots)$ which interpolates between the classical $\cW$-algebras of $\gs\gl_n$ for all $n$, has been known for many years; see \cite{KZ,KM,DSKV}. 

For convenience, in \cite{LI} we used a different parameter $\lambda$ which is a rational function of $\mu$ and the central charge $c$, and we denoted the universal algebra by $\cW(c,\lambda)$. It is a simple vertex algebra defined over the polynomial ring $\mathbb{C}[c,\lambda]$. However, there are certain prime ideals $I \subseteq \mathbb{C}[c,\lambda]$ such that the quotient
$$\cW^I(c,\lambda) = \cW(c,\lambda) / I \cdot \cW(c,\lambda),$$ is {\it not} simple as a vertex algebra over the ring $\mathbb{C}[c,\lambda] / I$. Here $I$ is regarded as a subset of the weight zero space $\cW(c,\lambda)[0] \cong \mathbb{C}[c,\lambda]$, and $I \cdot \cW(c,\lambda)$ denotes the vertex algebra ideal generated by $I$. Let $\cI \subseteq \cW^I(c,\lambda)$ denote the maximal proper ideal graded by conformal weight, so that $\cW^I(c,\lambda)/ \cI$ is the unique simple graded quotient. It turns out that {\it all} one-parameter vertex algebras of type $\cW(2,3,\dots, N)$ for some $N$ satisfying some mild hypotheses, can be obtained in this way for some choice of $I$. In \cite{LI}, we gave explicit formulas for the generator of $I$ in the case of $\cW^k(\gs\gl_n, f_{\text{prin}})$ as well as a few other such families. As an application, we obtained many nontrivial coincidences between the simple quotients of such vertex algebras; these correspond to the intersection points of the varieties $V(I)$ in the parameter space $\mathbb{C}^2$.

\subsection{Even spin $\cW_{\infty}$-algebra}  The $B$ and $C$ type principal $\cW$-algebras $\cW^k(\gs\go_{2n+1}, f_{\text{prin}})$ and $\cW^{\ell}(\gs\gp_{2n}, f_{\text{prin}})$ are of type $\cW(2,4,\dots, 2n)$, and are in fact isomorphic when $k$ and $\ell$ are related by the level shift 
\begin{equation} \label{levelshift} (k + 2 n-1) (\ell + n+1) = \frac{1}{2},\end{equation} by Feigin-Frenkel duality \cite{FF2,FF3}. Also, the $\mathbb{Z}_2$-orbifold of the $D$ type principal $\cW$-algebra $\cW^k(\gs\go_{2n}, f_{\text{prin}})^{\mathbb{Z}_2}$ is of type $\cW(2,4,\dots, 4n)$. A similar conjecture in the physics literature says that there should exist a unique two-parameter vertex algebra of type $\cW(2,4,\dots)$ which is strongly generated by a field in each even weight $2,4,\dots$, and generated by the weights $2$ and $4$ fields. This algebra is expected to interpolate between the $B$ and $C$ type principal $\cW$-algebras, as well as the $\mathbb{Z}_2$-orbifold of the $D$ type principal $\cW$-algebras, as above. Considerable experimental evidence for the existence and uniqueness of this algebra was obtained by Candu, Gaberdiel, Kelm and Vollenweider in \cite{CGKV}. Also, it is expected that many other one-parameter vertex algebras of type $\cW(2,4,\dots, 2N)$ for some $N$, can be obtained as quotients of this universal algebra.

\subsection{Main result} In this paper, we prove the existence and uniqueness of this algebra, which we denote by $\cW^{\mathrm{ev}}(c,\lambda)$. It is defined over the polynomial ring $\mathbb{C}[c,\lambda]$ and is generated by the Virasoro field $L$ of central charge $c$, and a weight $4$ primary field $W^4$. The remaining strong generators $W^{2i}$ of weight $2i$ are defined inductively by $$W^{2i} = W^{4}_{(1)} W^{2i-2},\qquad i \geq 3.$$ 
The procedure is similar to the construction of $\cW(c,\lambda)$ in \cite{LI}, although there are some surprising phenomena that appear in low weights which make the computations more involved. First, we show that all structure constants in the OPEs of $L(z) W^{2i}(w)$ and $W^{2j}(z) W^{2k}(w)$ for $2i \leq 12$ and $2j+2k \leq 14$, are uniquely determined by imposing appropriate Jacobi identities. This computation was carried out using the Mathematica package of Thielemans \cite{T}. Next, we show inductively that this data uniquely determines {\it all} structure constants in the OPEs $L(z) W^{2i}(w)$ and $W^{2j}(z) W^{2k}(w)$, if a certain subset of Jacobi identities are imposed. By invoking a result of De Sole and Kac \cite{DSKI} which associates to a nonlinear conformal algebra satisfying certain conditions, a vertex algebra known as its {\it universal enveloping vertex algebra}, we conclude that $\cW^{\mathrm{ev}}(c,\lambda)$ exists and is of type $\cW(2,4,\dots)$. However, it is not immediately clear that it is {\it freely generated} of this type. Finally, by considering a certain family of quotients of $\cW^{\mathrm{ev}}(c,\lambda)$ whose graded characters are known, we prove that $\cW^{\mathrm{ev}}(c,\lambda)$ is indeed freely generated.

\subsection{Quotients of $\cW^{\mathrm{ev}}(c,\lambda)$ and the classification of vertex algebras of type $\cW(2,4,\dots, 2N)$}

$\cW^{\mathrm{ev}}(c,\lambda)$ has a conformal weight grading $$\cW^{\mathrm{ev}}(c,\lambda) = \bigoplus_{n\geq 0} \cW^{\mathrm{ev}}(c,\lambda)[n],$$ where each $\cW^{\mathrm{ev}}(c,\lambda)[n]$ is a free $\mathbb{C}[c,\lambda]$-module and $\cW^{\mathrm{ev}}(c,\lambda)[0] \cong \mathbb{C}[c,\lambda]$. There is a symmetric bilinear form on $\cW^{\mathrm{ev}}(c,\lambda)[n]$ given by
$$\langle ,  \rangle_n : \cW^{\mathrm{ev}}(c,\lambda)[n] \otimes_{\mathbb{C}[c,\lambda]} \cW^{\mathrm{ev}}(c,\lambda)[n] \ra \mathbb{C}[c,\lambda],\qquad \langle \omega, \nu \rangle_n = \omega_{(2n-1)} \nu.$$ The level $n$ Shapovalov determinant $\mathrm{det}_n \in \mathbb{C}[c,\lambda]$ is just the determinant of this form. It turns out that $\mathrm{det}_n$ is nonzero for all $n$; equivalently, $\cW^{\mathrm{ev}}(c,\lambda)$ is a simple vertex algebra over $\mathbb{C}[c,\lambda]$.

Let $p$ be an irreducible factor of $\mathrm{det}_{2N+2}$ and let $I = (p) \subseteq \mathbb{C}[c,\lambda] \cong \cW^{\mathrm{ev}}(c,\lambda)[0]$ be the ideal generated by $p$. Consider the quotient
$$ \cW^{\mathrm{ev},I}(c,\lambda) = \cW^{\mathrm{ev}}(c,\lambda) / I \cdot \cW^{\mathrm{ev}}(c,\lambda),$$ where $I \cdot \cW^{\mathrm{ev}}(c,\lambda)$ is the vertex algebra ideal generated by $I$. This is a vertex algebra over the ring $\mathbb{C}[c,\lambda]/I$, which is no longer simple. It contains a singular vector $\omega$ in weight $2N+2$, that is, a nonzero element of the maximal proper ideal $\cI\subseteq \cW^{\mathrm{ev},I}(c,\lambda)$ graded by conformal weight. If $p$ does not divide $\mathrm{det}_{m}$ for any $m<2N+2$, $\omega$ will have minimal weight among elements of $\cI$. Often, there exists a localization $R$ of $\mathbb{C}[c,\lambda]/I$ such that $\omega$ has the form \begin{equation} \label{sing:intro} W^{2N+2} - P(L, W^4,\dots, W^{2N}),\end{equation} in the localization 
$$\cW^{\mathrm{ev},I}_R(c,\lambda) = R \otimes_{\mathbb{C}[c,\lambda]/I} \cW^{\mathrm{ev},I}(c,\lambda).$$
Here $P$ is a normally ordered polynomial in the fields $L,W^4,\dots, W^{2N}$, and their derivatives, with coefficients in $R$. If this is the case, there will exist relations $$W^{2m} = P_{2m}(L, W^4, \dots, W^{2N})$$ for all $m \geq N$ expressing $W^{2m}$ in terms of $L, W^4,\dots, W^{2N}$ and their derivatives. The simple quotient $\cW^{\mathrm{ev},I}_R(c,\lambda) / \cI$ will then be of type $\cW(2,4,\dots, 2N)$. Conversely, we will show that any simple one-parameter vertex algebra of type $\cW(2,4,\dots, 2N)$ satisfying some mild hypotheses, can be obtained as the simple quotient of $\cW^{\mathrm{ev},I}_R(c,\lambda)$ for some $I$ and $R$. This reduces the classification of such vertex algebras to the classification of prime ideals $I = (p) \subseteq\mathbb{C}[c,\lambda]$ such that $p$ divides $\mathrm{det}_{2N+2}$ but does not divide $\mathrm{det}_m$ for $m<2N+2$, and $\cW^{\mathrm{ev},I}(c,\lambda)$ contains a singular vector of the form \eqref{sing:intro}, possibly after localizing. 

There are many interesting one-parameter vertex algebras of type $\cW(2,4,\dots, 2N)$ for some $N$. Here is a short list of examples.

\begin{enumerate}

\item For $n\geq 2$, the $B$ and $C$ type principal $\cW$-algebras $\cW^{k}(\gs\go_{2n+1}, f_{\text{prin}})$ and $\cW^{\ell}(\gs\gp_{2n}, f_{\text{prin}})$ and are freely generated of type $\cW(2,4,\dots, 2n)$, and are isomorphic after the level shift \eqref{levelshift}.

\item For $n\geq 3$, the type $D$ principal $\cW$-algebra  $\cW^k(\gs\go_{2n}, f_{\text{prin}})$ has a $\mathbb{Z}_2$-action, and the orbifold $\cW^k(\gs\go_{2n}, f_{\text{prin}})^{\mathbb{Z}_2}$ is of type $\cW(2,4,\dots, 4n)$; see Corollary \ref{cor:typeD}.

\item For a Lie algebra $\gg$, let $V^k(\gg)$ denote the universal affine vertex algebra of $\gg$ at level $k$, and $L_k(\gg)$ its simple graded quotient. For $n\geq 1$, the coset of $V^{k}(\gs\gp_{2n})$ inside $V^{k+1/2}(\gs\gp_{2n}) \otimes L_{-1/2}(\gs\gp_{2n})$, is of type $\cW(2,4\dots, 2n^2+4n)$; see Example 7.1 of \cite{CLi}.

\item For $n\geq 2$, the coset of $V^{k+1/2}(\gs\gp_{2n-2})$ inside the minimal $\cW$-algebra $\cW^k(\gs\gp_{2n}, f_{\text{min}})$ is of type $\cW(2,4,\dots, 2n^2+2n -2)$; see Theorem 5.2 of \cite{ACKL}.

\item Let $N^k(\mathfrak{sl}_2)$ denote the parafermion algebra of $\mathfrak{sl}_2$, that is, the coset of the Heisenberg algebra inside $V^k(\mathfrak{sl}_2)$. The $\mathbb{Z}_2$-orbifold $N^k(\mathfrak{sl}_2)^{\mathbb{Z}_2}$ is of type $\cW(2,4,6,8,10)$.
\end{enumerate}

All the above families of vertex algebras arise as quotients of $\cW^{\mathrm{ev},I}_R(c,\lambda)$ for some prime ideal $I = (p) \subseteq\mathbb{C}[c,\lambda]$ and some localization $R$ of $\mathbb{C}[c,\lambda]/I$; see Corollaries \ref{cor:typecrealization} and \ref{cor:typedrealization}, and Theorems \ref{thm:otherquotients} and \ref{paraorbcurve}. Given a prime ideal $I = (p)$ such that $p$ lies in the Shapovalov spectrum, let $V(I)\subseteq \mathbb{C}^2$ denote the corresponding variety. We call $V(I)$ the {\it truncation curve} associated to the one-parameter vertex algebra arising as the simple quotient of $\cW^{\mathrm{ev},I}_R(c,\lambda)$. For Examples (1) and (2) above, the explicit generator of $I$ can be found by combining our main result with the calculations of Hornfeck appearing in \cite{H}. For Example (5), we will also give the explicit generator of $I$, and for Example (4), we will give a conjectural formula for the generator.

It is also important to consider $\cW^{\mathrm{ev},I}(c,\lambda)$ when $I\subseteq \mathbb{C}[c,\lambda]$ is a {\it maximal} ideal, which has the form $I = (c- c_0, \lambda- \lambda_0)$ for some $c_0, \lambda_0\in \mathbb{C}$. Then $\cW^{\mathrm{ev},I}(c,\lambda)$ and its quotients are vertex algebras over $\mathbb{C}$. Given two maximal ideals $I_0 = (c- c_0, \lambda- \lambda_0)$ and $I_1 = (c - c_1, \lambda - \lambda_1)$, let $\cW_0$ and $\cW_1$ be the simple quotients of $\cW^{\mathrm{ev},I_0}(c,\lambda)$ and $\cW^{\mathrm{ev},I_1}(c,\lambda)$. There is a very simple criterion for $\cW_0$ and $\cW_1$ to be isomorphic; see Theorem \ref{thm:coincidences}. We must have $c_0 = c_1$, and if this central charge is $0, 1, -24, \frac{1}{2},-\frac{22}{5}$, there is no restriction on $\lambda_0, \lambda_1$. If $c_0 = c = c_1$ is arbitrary with $c\neq 1, 25$, and $\displaystyle \lambda_0 = \pm \frac{1}{7 \sqrt{(c -25) (c -1)}} = \pm \lambda_1$ or $\displaystyle \lambda_0 = \pm \frac{\sqrt{196 - 172 c + c^2}}{21 (c-1) (22 + 5 c)} = \pm \lambda_1$,  we also have $\cW_0 \cong \cW_1$.  In all other cases, we must have $\lambda_0 = \lambda_1$. Our criterion for $\cW_0$ and $\cW_1$ to be isomorphic implies that aside from the above coincidences, all other pointwise coincidences among the simple quotients of one-parameter vertex algebras $\cW^{\mathrm{ev},I}(c,\lambda)$ and $\cW^{\mathrm{ev},J}(c,\lambda)$, correspond to intersection points of their truncation curves $V(I)$ and $V(J)$. As an application, we classify all nontrivial coincidences among the simple algebras $\cW_{\ell}(\gs\go_{2n},f_{\text{prin}})^{\mathbb{Z}_2}$, $\cW_{\ell'}(\gs\go_{2m},f_{\text{prin}})^{\mathbb{Z}_2}$, $\cW_{k}(\gs\gp_{2r},f_{\text{prin}})$ and $\cW_{k'}(\gs\gp_{2s},f_{\text{prin}})$, as well as the simple parafermion orbifold $N_t(\mathfrak{sl}_2)^{\mathbb{Z}_2}$. The coincidences between $\cW_{\ell}(\gs\go_{2n},f_{\text{prin}})^{\mathbb{Z}_2}$ and $\cW_{\ell'}(\gs\go_{2m},f_{\text{prin}})^{\mathbb{Z}_2}$ were previously observed in the physics literature \cite{CGKV}, and our approach provides a rigorous proof. Finally, as a corollary of our classification, we give new examples of principal $\cW$-algebras of $B$, $C$, and $D$ types at nonadmissible levels, which are lisse and rational.

\section{Vertex algebras} \label{section:VOAs}
Here we define vertex algebras, which have been discussed from various different points of view in the literature \cite{Bor,FLM,FHL,K,FBZ,LeLi}. We will follow the formalism developed in \cite{LZ} and partly in \cite{Li}, and our presentation closely follows \cite{LI}. Let $V=V_0\oplus V_1$ be a super vector space over $\mathbb{C}$, let $z,w$ be formal variables, and let $\text{QO}(V)$ denote the space of linear maps $$V\ra V((z))=\left\lbrace\sum_{n\in\mathbb{Z}} v(n) z^{-n-1}|
v(n)\in V,\ v(n)=0\ \text{for} \ n>\!\!>0 \right\rbrace.$$ Each element $a\in \text{QO}(V)$ can be represented as a power series
$$a=a(z)=\sum_{n\in\mathbb{Z}}a(n)z^{-n-1}\in \text{End}(V)[[z,z^{-1}]].$$ We assume that $a=a_0+a_1$ where $a_i:V_j\ra V_{i+j}((z))$ for $i,j\in\mathbb{Z}/2\mathbb{Z}$, and we write $|a_i| = i$.

For each $n \in \mathbb{Z}$, $\text{QO}(V)$ has a bilinear operation defined on homogeneous elements $a$ and $b$ by
$$ a(w)_{(n)}b(w)=\text{Res}_z a(z)b(w)\ \iota_{|z|>|w|}(z-w)^n- (-1)^{|a||b|}\text{Res}_z b(w)a(z)\ \iota_{|w|>|z|}(z-w)^n.$$
Here $\iota_{|z|>|w|}f(z,w)\in\mathbb{C}[[z,z^{-1},w,w^{-1}]]$ denotes the power series expansion of a rational function $f$ in the region $|z|>|w|$. For $a,b\in \text{QO}(V)$, we have the following identity known as the {\it operator product expansion} (OPE).
\begin{equation}\label{opeform} a(z)b(w)=\sum_{n\geq 0}a(w)_{(n)} b(w)\ (z-w)^{-n-1}+:a(z)b(w):. \end{equation}
Here $:a(z)b(w):\ =a(z)_-b(w)\ +\ (-1)^{|a||b|} b(w)a(z)_+$, where $a(z)_-=\sum_{n<0}a(n)z^{-n-1}$ and $a(z)_+=\sum_{n\geq 0}a(n)z^{-n-1}$. Often, \eqref{opeform} is written as
$$a(z)b(w)\sim\sum_{n\geq 0}a(w)_{(n)} b(w)\ (z-w)^{-n-1},$$ where $\sim$ means equal modulo the term $:a(z)b(w):$, which is regular at $z=w$. 

Note that $:a(w)b(w):$ is a well-defined element of $\text{QO}(V)$. It is called the {\it Wick product} or {\it normally ordered product} of $a$ and $b$, and it
coincides with $a_{(-1)}b$. For $n\geq 1$ we have
$$ n!\ a(z)_{(-n-1)} b(z)=\ :(\partial^n a(z))b(z):,\qquad \partial = \frac{d}{dz}.$$
For $a_1(z),\dots ,a_k(z)\in \text{QO}(V)$, the iterated Wick product is defined inductively by
\begin{equation}\label{iteratedwick} :a_1(z)a_2(z)\cdots a_k(z):\ =\ :a_1(z)b(z):,\qquad b(z)=\ :a_2(z)\cdots a_k(z):.\end{equation}
We usually omit the formal variable $z$ when no confusion can arise.

We denote the constant power series $\text{id}_V \in \text{QO}(V)$ by $\mathbf{1}$. A subspace $\cA\subseteq \text{QO}(V)$ containing $\mathbf{1}$ which is closed under all the above products is called a {\it quantum operator algebra} (QOA). We say that $a,b\in \text{QO}(V)$ are {\it local} if $$(z-w)^N [a(z),b(w)]=0$$ for some $N\geq 0$. A {\it vertex algebra} will be a QOA whose elements are pairwise local. This notion is equivalent to the notion of a vertex algebra in the sense of \cite{FLM}. 

A vertex algebra $\cA$ is {\it generated} by a subset $S=\{\alpha^i|\ i\in I\}$ if $\cA$ is spanned by words in the letters $\alpha^i$, and all products, for $i\in I$ and $n\in\mathbb{Z}$. We say that $S$ {\it strongly generates} $\cA$ if $\cA$ is spanned by words in the letters $\alpha^i$, and all products for $n<0$. Equivalently, $\cA$ is spanned by $$\{ :\partial^{k_1} \alpha^{i_1}\cdots \partial^{k_m} \alpha^{i_m}:| \ i_1,\dots,i_m \in I,\ k_1,\dots,k_m \geq 0\}.$$ Suppose that $S$ is an ordered strong generating set $\{\alpha^1, \alpha^2,\dots\}$ for $\cA$ which is at most countable. We say that $S$ {\it freely generates} $\cA$, if $\cA$ has a basis consisting of all monomials 
\begin{equation} \label{freegen} \begin{split} & :\partial^{k^1_1} \alpha^{i_1} \cdots \partial^{k^1_{r_1}}\alpha^{i_1} \partial^{k^2_1} \alpha^{i_2} \cdots \partial^{k^2_{r_2}}\alpha^{i_2}
 \cdots \partial^{k^n_1} \alpha^{i_n} \cdots \partial^{k^n_{r_n}} \alpha^{i_n}:,\qquad 
 1\leq i_1 < \dots < i_n,
 \\ & k^1_1\geq k^1_2\geq \cdots \geq k^1_{r_1},\qquad k^2_1\geq k^2_2\geq \cdots \geq k^2_{r_2},  \ \ \cdots,\ \  k^n_1\geq k^n_2\geq \cdots \geq k^n_{r_n},
 \\ &  k^{t}_1 > k^t_2 > \dots > k^t_{r_t}\ \  \text{whenever} \ \ \alpha^{i_t}\ \ \text{is odd}. 
 \end{split} \end{equation}

A {\it conformal structure} with central charge $c$ on $\cA$ is a Virasoro vector $$L(z) = \sum_{n\in \mathbb{Z}} L_n z^{-n-2}$$ in $\cA$ satisfying
\begin{equation} \label{virope} L(z) L(w) \sim \frac{c}{2}(z-w)^{-4} + 2 L(w)(z-w)^{-2} + \partial L(w)(z-w)^{-1},\end{equation} such that $L_{-1} \alpha = \partial \alpha$ for all $\alpha \in \cA$, and $L_0$ acts diagonalizably on $\cA$. We say that $\alpha$ has conformal weight $d$ if $L_0(\alpha) = d \alpha$, and we denote the conformal weight $d$ subspace by $\cA[d]$. In this paper, all our vertex algebras will have conformal structures, and will be $\mathbb{N}$-graded by conformal weight:
$$\cA = \bigoplus_{d\geq 0} \cA[d].$$ 
We say that a vertex algebra $\cA$ is of type $$\cW(d_1,d_2,\dots)$$ if it has a minimal strong generating set consisting of one even field in each conformal weight $d_1, d_2, \dots $. If $\cA$ is freely generated of type $\cW(d_1,d_2,\dots)$, it has graded character
\begin{equation} \label{gradedchar:ua} \chi(\cA, q) = \sum_{n\geq 0} \text{dim}(\cA[n]) q^n = \prod_{i\geq 1} \prod_{k\geq 0} \frac{1}{1-q^{d_i +k}}.\end{equation}

Given fields $a,b,c$ in a vertex algebra $\cA$, the following identities hold.
\begin{equation} \label{deriv} (\partial a)_{(n)} b = -na_{(n-1)}b \qquad \forall n\in \mathbb{Z},\end{equation}
\begin{equation} \label{commutator} a_{(n)} b  =  (-1)^{|a||b|} \sum_{p \in \mathbb{Z}} (-1)^{p+1} (b_{(p)} a)_{(n-p-1)} \mathbf{1},\qquad \forall n\in \mathbb{Z},\end{equation}
\begin{equation} \label{nonasswick} :(:ab:)c:\  - \ :abc:\ 
=  \sum_{n\geq 0}\frac{1}{(n+1)!}\big( :(\partial^{n+1} a)(b_{(n)} c):\ +
(-1)^{|a||b|} (\partial^{n+1} b)(a_{(n)} c):\big)\ .\end{equation}
\begin{equation} \label{ncw} a_{(n)}
(:bc:) -\ :(a_{(n)} b)c:\ - (-1)^{|a||b|}\ :b(a_{(n)} c): \ = \sum_{i=1}^n
\binom{n}{i} (a_{(n-i)}b)_{(i-1)}c, \qquad \forall n \geq 0.
\end{equation}

Given fields $a,b,c$ and integers $m,n \geq 0$, the following identities are known as {\it Jacobi identities} of type $(a,b,c)$. 
\begin{equation} \label{jacobi} a_{(r)}(b_{(s)} c) = (-1)^{|a||b|} b_{(s)} (a_{(r)}c) + \sum_{i =0}^r \binom{r}{i} (a_{(i)}b)_{(r+s - i)} c.\end{equation}

\subsection{Vertex algebras over commutative rings} \label{section:voaring}
Let $R$ be a finitely generated, unital, commutative $\mathbb{C}$-algebra. A {\it vertex algebra over $R$} is an $R$-module $\cA$ with a vertex algebra structure which is defined as above. The theory of vertex algebras over general commutative rings was developed by Mason \cite{Ma}, but the main difficulties are not present when $R$ is a $\mathbb{C}$-algebra. 

Given an $R$-module $M$, we define $\text{QO}_R(M)$ to be the set of $R$-module homomorphisms $a: M \ra M((z))$, which can be represented by power series $$a(z) = \sum_{n\in \mathbb{Z}} a(n) z^{-n-1} \in \text{End}_R(M)[[z,z^{-1}]].$$ Here $a(n) \in \text{End}_R(M)$ is an $R$-module endomorphism, and for each $v\in M$, $a(n) v = 0$ for $n>\!\!>0$. Then $\text{QO}_R(M)$ is an $R$-module, and we define $a_{(n)} b$ as before. These operations are $R$-module homomorphisms from $\text{QO}_R(M) \otimes_R \text{QO}_R(M) \ra \text{QO}_R(M)$. A QOA will be an $R$-module $\cA \subseteq \text{QO}_R(M)$ containing $\mathbf{1}$ and closed under all products. Locality is defined as before, and a vertex algebra over $R$ is a QOA $\cA\subseteq \text{QO}_R(M)$ whose elements are pairwise local.

We say that a set $S = \{\alpha^i|\ \ i\in I\}$ generates $\cA$ if $\cA$ is spanned as an $R$-module by all words in $\alpha^i$ and the above products. We say that $S$ strongly generates $\cA$ if $\cA$ is spanned as an $R$-module by all iterated Wick products of these generators and their derivatives. If $S = \{\alpha^1, \alpha^2,\dots\}$ is an ordered strong generating set for $\cA$ which is at most countable, we say that $S$ freely generates $\cA$, if $\cA$ has an $R$-basis consisting of all normally ordered monomials of the form \eqref{freegen}. In general, $\cA$ need not be a free $R$-module, but if $\cA$ is freely generated by some $S$, it is a free $R$-module.

Suppose that $\cV$ is a vertex algebra over $R$ containing a field $L$ satisfying \eqref{virope}, for some $c \in R$. We call $L$ a conformal structure on $\cV$ if $L_{-1}$ acts on $\cV$ by $\partial$ and $L_0$ acts diagonalizably, and we have an $R$-module decomposition $$\cV = \bigoplus_{d\in R} \cV[d],$$ where $\cV[d]$ is the $L_0$-eigenspace with eigenvalue $d$. If each $\cV[d]$ is in addition a free $R$-module of finite rank, we have the graded character
$$\chi(\cV ,q) = \sum_{d \in R} \text{rank}_R(\cV[d]) q^d.$$
In all our examples, the grading will be by $\mathbb{N}$, regarded as a subsemigroup of $R$, and $\cV[0] \cong R$.

Let $\cV$ be a vertex algebra over $R$ with weight grading 
$$\cV = \bigoplus_{n\geq 0} \cV[n],\qquad \cV[0] \cong R.$$ 
A vertex algebra ideal $\cI \subseteq \cV$ is called {\it graded} if $$\cI = \bigoplus_{n\geq 0} \cI[n],\qquad \cI[n] = \cI \cap \cV[n].$$ 
We say that $\cV$ is {\it simple} if there are no graded ideals $\cI$ such that $\cI[0] \neq \{0\}$. If $I\subseteq R$ is an ideal, we may regard $I$ as a subset of $\cV[0] \cong R$. Let $I \cdot \cV$ denote the set of $I$-linear combinations of elements of $\cV$, which is just the vertex algebra ideal generated by $I$. Then $$\cV^I = \cV / (I \cdot \cV)$$ is a vertex algebra over the ring $R/I$. Even if $\cV$ is simple as a vertex algebra over $R$, $\cV^I$ need not be simple as a vertex algebra over $R/I$.

\subsection{Shapovalov form}
 Let $\cV = \bigoplus_{n\geq 0}\cV[n]$ be a vertex algebra over $R$ which is $\mathbb{N}$-graded by conformal weight. Then each $\cV[n]$ has a symmetric bilinear form 
\begin{equation} \label{bilinearform} \langle,\rangle_n: \cV[n]\otimes_{R} \cV[n] \ra R,\qquad \langle u,v \rangle_n = u_{(2n-1)}v.\end{equation}
We stipulate that $\langle \cV[n],\cV[m]\rangle=0$ if $n\neq m$
and extend $\langle,\rangle$ linearly to all $\cV$.

A vector $v$ in the radical of the Shapovalov form $\langle,\rangle$ is called a singular vector. Suppose now that each weight space $\cV[n]$ is a free $R$-module of finite rank. We then define the level $n$ Shapovalov determinant $\mathrm{det}_n \in R$ to be the determinant of the matrix of $\langle,\rangle_n$. 
Now we provide more details on the Shapovalov form.

\begin{lemma}\label{lem:singularbetter}
	Let $\mathcal{V}$ be an $\mathbb{N}$-graded vertex algebra over $R$.
	Let $v$ be a singular vector of weight $n>0$.
	Then, for any homogeneous $w$ of weight $n-t$ with $0\leq t\leq n$, $w_{(2n-t-1)}v=0$.
\end{lemma}
\begin{proof}
	We induct on $t$.
	Since $v$ is a singular vector, $t=0$ case follows by definition.  
	From \eqref{deriv}, $-(2n-t)\,w_{(2n-t-1)}v = (\partial w)_{(2n-(t-1)-1)}v$. The right-hand side here is $0$ by inductive assumption and also $(2n-t)\neq 0$. 
\end{proof}

\begin{prop}\label{prop:shap} 
	Let $\mathcal{V}$ be an $\mathbb{N}$-graded vertex algebra over $R$ where $\cV[0] \cong R$ and each $\cV[n]$ is a free $R$-module of finite rank. We also assume that $L_1\cV[1]=0$. Then a homogeneous vector of weight $n>0$ is in the radical of the Shapovalov form (that is, it is a singular vector) if and only if it is contained in a proper ideal of $\mathcal{V}$.

\end{prop}
\begin{proof}
	Due to our assumptions on $\mathcal{V}$, Theorem 3.1 of \cite{Li-bilinear} applies and there exists a unique up to scaling 
	non-trivial invariant bilinear form $(-, -)$ on $\mathcal{V}$.
	This form is automatically symmetric \cite{FHL} and determined by $(\mathbf{1},\mathbf{1})=1$. We have $(\mathcal{V}[{i}],\mathcal{V}[{j}])=0$, if $i\neq j$
	\cite{Li-bilinear}. Due to invariance, the radical of $(-,-)$ is a proper ideal of $\mathcal{V}$ which we denote by $\mathcal{R}$.
	
	We first prove that if $v$ is a singular vector of weight $n$
	then $v\in\mathcal{R}$.
	It is enough to show that $(u,v)=0$ for all $u$ homogeneous of weight $n$.
	By Equation 3.1 of \cite{Li-bilinear}, 
	\begin{align}
	(u,v) &= \mathrm{Coeff}_{z^0}\, (-1)^{n}z^{-2n}\left(\mathbf{1}, (e^{zL_1}u)(z^{-1})v   \right)\\
	&= (-1)^{n}\cdot \mathrm{Coeff}_{z^{2n}}\, \left(\mathbf{1}, 
	\sum_{t\geq 0} \frac{z^t}{t!}\sum_{j\in\mathbb{Z}}({L_1^t}u)_{(j)} v\,z^{j+1}   \right)\nonumber\\
	&= (-1)^{n} \left(\mathbf{1}, 
	\sum_{t\geq 0} \frac{1}{t!}({L_1^t}u)_{(2n-t-1)} v   \right)\nonumber= (-1)^{n} \left(\mathbf{1}, 
	\sum_{0\leq t\leq n} \frac{1}{t!}({L_1^t}u)_{(2n-t-1)} v   \right)\nonumber.
	\end{align}
	Since $L_1^tu$ has weight $n-t$ we conclude that  right-hand side is $0$ by Lemma \ref{lem:singularbetter}.
	
	Conversely, if $v$ is contained in a proper ideal, then $v$ clearly has to be a singular vector, otherwise the vacuum vector is obtainable by acting appropriately on $v$. In fact, we have now proved that $\mathcal{R}$ is precisely the radical of the Shapovalov form.
\end{proof}

Under the above hypotheses, if $R$ is in addition a unique factorization ring,  each irreducible factor $a$ of $\mathrm{det}_n$ give rise to a prime ideal $(a) \subseteq R$. Clearly if $a| \mathrm{det}_n$, then $a|\mathrm{det}_m$ for all $m>n$. We define the {\it Shapovalov spectrum} $\text{Shap}(\cV)$  to be the set of distinct prime ideals of the form $(a) \subseteq R$ such that $a$ is a divisor of $\mathrm{det}_n$ for some $n$. Note that the ideals $I = (a) \in \text{Shap}(\cV)$ are precisely the prime ideals $I$ for which $\cV^I$ is not simple as a vertex algebra over $R/I$. We say that $(a) \in \text{Shap}(\cV)$ has level $n$ if $a$ divides $\mathrm{det}_n$ but does not divide $\mathrm{det}_m$ for $m<n$.

\section{Main construction}\label{section:main} 
In this section, we will construct the universal two-parameter vertex algebra $\cW^{\mathrm{ev}}(c,\lambda)$. It is defined over the polynomial ring $\mathbb{C}[c,\lambda]$ and is freely generated by a Virasoro field $L$ of central charge $c$, and a field $W^{2i}$ of conformal weight $2i$ for each $i\geq 2$. As in the case of $\cW(c,\lambda)$, $\cW^{\mathrm{ev}}(c,\lambda)$ will be the universal enveloping vertex algebra of a nonlinear Lie conformal algebra $\cL^{\mathrm{ev}}(c,\lambda)$ over $\mathbb{C}[c,\lambda]$ with generators $\{L, W^{2i}|\ i \geq 2\}$ and grading $\Delta(L) = 2$, $\Delta(W^{2i}) = 2i$, in the sense of \cite{DSKI}.

We shall work with the OPE rather than lambda-bracket formalism, so the sesquilinearity, skew-symmetry, and Jacobi identities from \cite{DSKI} are replaced with \eqref{deriv}-\eqref{jacobi}. As explained in \cite{LI}, specifying a nonlinear Lie conformal algebra in the language of OPEs means specifying fields $\{\alpha^1, \alpha^2,\dots\}$ where $\alpha^i$ has weight $d_i >0$, and pairwise OPEs $\displaystyle \alpha^i(z) \alpha^j(w) = \sum_{n\geq 0} (\alpha^i_{(n)} \alpha^j)(w)(z-w)^{-n-1}$ where each term $\alpha^i_{(n)} \alpha^j$ has weight $d_i + d_j -n-1$, and is a normally ordered polynomial in the generators and their derivatives. Additionally, for all $a,b,c \in \{\alpha^1, \alpha^2,\dots\}$, \eqref{deriv}-\eqref{ncw} hold, and all Jacobi identities \eqref{jacobi} hold as consequences of \eqref{deriv}-\eqref{ncw} alone. This data uniquely determines the universal enveloping vertex algebra which is freely generated by $\{\alpha^1, \alpha^2,\dots\}$.

In this notation, our goal will be to construct the OPE algebra of the generating fields $\{L, W^{2i}|\ i\geq 2\}$ of $\cW^{\mathrm{ev}}(c,\lambda)$, such that \eqref{deriv}-\eqref{ncw} are imposed, and the Jacobi identities \eqref{jacobi} hold as a consequence of these identities alone. 

We postulate that $\cW^{\mathrm{ev}}(c,\lambda)$ has the following features. 
\begin{enumerate} 
\item $\cW^{\mathrm{ev}}(c,\lambda)$ possesses a Virasoro field $L$ of central charge $c$ and an even weight $4$ primary field $W^4$, so that
\begin{equation} \label{ope:zeroth} \begin{split} & L(z) L(w)  \sim \frac{c}{2} (z-w)^{-4} + 2 L(w)(z-w)^{-2} + \partial L(w)(z-w)^{-1}, \\  & L(z) W^4(w) \sim 4 W^4(w)(z-w)^{-2}+\partial W^4(w)(z-w)^{-1}.\end{split}\end{equation}

\item The fields $W^{2i}$ are defined inductively by $$W^{2i} = W^4_{(1)} W^{2i-2},\qquad i\geq 3.$$ We have 
$$L(z) W^{2i}(w) \sim \cdots + 2i W^{2i} (w)(z-w)^{-2} + \partial W^{2i}(w)(z-w)^{-1},$$ and the fields $\{L, W^{2i}|\ i\geq 2\}$ close under OPE. Since $W^{2i}$ is not assumed primary for $i >2$, $(\cdots)$ denotes the higher order poles. In particular, $\cW^{\mathrm{ev}}(c,\lambda)$ is generated as a vertex algebra by $\{L,W^4\}$, and is strongly generated by $\{L, W^{2i}|\ i\geq 2\}$.

\item $W^4$ is nondegenerate in the sense that $W^4_{(7)} W^4 \neq 0$.
\end{enumerate}

For notational convenience, we sometimes denote $L$ by $W^2$. The assumption that $W^{2i}$ has conformal weight $2i$ has the following consequence. For $k = 0,1$, and $2\leq 2i\leq 2j$, $W^{2i}_{(k)} W^{2j}$ depends only on $L, W^4, \dots, W^{ 2i + 2j -2}$ and their derivatives. 
In particular, we can write
\begin{equation} \label{eq:aij} W^{2i}_{(1)} W^{2j} = a_{2i,2j}  W^{2i+2j-2} + C_{2i,2j},\end{equation} where $a_{2i,2j}$ denotes the coefficient of $W^{2i+2j-2}$ and $C_{2i,2j}$ is a normally ordered polynomial in $L, W^4, \dots, W^{2i+2j-4}$ and their derivatives. By assumption, $a_{2,2j} = 2j$, $a_{4,2j} = 1$, $C_{2,2j} =0$, and $C_{4,2j} = 0$ for all $j\geq 2$. Similarly,
\begin{equation}  \label{eq:bij} W^{2i}_{(0)} W^{2j} = b_{2i,2j} \partial W^{2i+2j-2} + D_{2i,2j},\end{equation} where $b_{2i,2j}$ denotes the coefficient of $\partial W^{2i+2j-2}$ and $D_{2i,2j}$ is a normally ordered polynomial in $L, W^4, \dots, W^{2i+2j-4}$ and their derivatives. By assumption, $b_{2,2j} = 1$ and $D_{2,2j} = 0$ for all $j\geq 2$.

Let $S_{2n,k}$ denote the set of products \begin{equation} \{W^{2i}_{(k)} W^{2j} |\ 2i+2j = 2n,\ k \geq 0\};\end{equation} note that they vanish for $k > 2n-1$. First, imposing the identities \eqref{deriv}-\eqref{commutator}, together with all Jacobi relations of type $(W^{2i}, W^{2j}, W^{2k})$ for $2i+2j+2k \leq 16$, uniquely determines $S_{2n,k}$ for all $2n\leq 14$ and $k\geq 0$. If we then assume $S_{2m,k}$ to be known for $2m\leq 2n$ and $k\geq 0$, we show that by imposing a subset of Jacobi relations of type $(W^{2i}, W^{2j}, W^{2k})$ with $2i+2j+2k = 2n+4$, this uniquely determines $S_{2n+2,k}$. 

It is important to note that we are not imposing {\it all} Jacobi relations of type $(W^{2i}, W^{2j}, W^{2k})$ for $2i+2j+2k = 2n+4$ at this stage. We leave open the possibility that some of them might not vanish, that is, they give rise to nontrivial null fields; see \cite{DSKI} as well as the expository account in \cite{LI} in the OPE formalism. This has the effect that as we proceed through the induction, the OPEs $W^{2i}(z) W^{2j}(w)$ are only determined up to null fields. For convenience, we suppress this from our notation. At the end of the induction, we obtain the existence of a (possibly degenerate) nonlinear conformal algebra $\cL^{\mathrm{ev}}(c,\lambda)$ over $R$ with generators $\{L, W^{2i}|\ i\geq 2\}$, satisfying skew-symmetry. By the De Sole-Kac correspondence \cite{DSKI}, the universal enveloping vertex algebra $\cW^{\mathrm{ev}}(c,\lambda)$ exists and is unique, but may not be freely generated. Finally, by considering a family of quotients of $\cW^{\mathrm{ev}}(c,\lambda)$ whose graded characters are known, we will show that $\cW^{\mathrm{ev}}(c,\lambda)$ is in fact freely generated; equivalently, $\cL^{\mathrm{ev}}(c,\lambda)$ is a nonlinear Lie conformal algebra.

\subsection{Step 1: $S_{2n,k}$ for $2n\leq 14$}

The most general OPE of $L$ and $W^6$ is
\begin{equation*}\begin{split} L(z) W^6(w) & \sim a_0 (z-w)^{-8} + a_1 L(w)(z-w)^{-6} + a_2 \partial L(w)(z-w)^{-5} \\ & + \bigg(a_3 W^4+a_4 :L^2:+a_5 \partial^2 L\bigg)(w)(z-w)^{-4}  \\ & + \bigg(a_6 \partial W^4+ a_7 :(\partial L) L:+ a_8 \partial^3 L\bigg)(w)(z-w)^{-3} \\ & + 6 W^6(w)(z-w)^{-2} + \partial W^6(w)(z-w)^{-1}. \end{split} \end{equation*}
Similarly, since $W^6 = W^4_{(1)} W^4$, the most general OPE of $W^4$ with itself is
\begin{equation*} \begin{split} W^4(z) W^4(w) & \sim b_0 (z-w)^{-8} + b_1 L(w)(z-w)^{-6} + b_2 \partial L(w)(z-w)^{-5}\\ & + \bigg(b_3 W^4 + b_4 :L^2: + b_5 \partial L\bigg)(w)(z-w)^{-4}  \\ & + \big(b_6 \partial W^4  + b_7 :(\partial L)L: +b_8  \partial^3 L \big)(w) (z-w)^{-3} \\ & + W^6(w)(z-w)^{-2} \\ & +\bigg(b_9 \partial W^6 + b_{10}  :(\partial L)L^2: + b_{11} :(\partial L) W^4: +b_{12}  : L \partial W^4: + b_{13} :(\partial^3 L)L: \\ & +b_{14}  :(\partial^2 L)\partial L: + b_{15} \partial^3 W^4 +b_{16} \partial^5 L\bigg)(w)(z-w)^{-1}.\end{split} \end{equation*}

Imposing the Jacobi relations of type $(L,W^4,W^4)$ and $(L,L,W^6)$ constrains the parameters $a_0,\dots, a_8$ and $b_0, \dots, b_{16}$. It is not difficult to check that in addition to the central charge $c$, there are two free parameters that are consistent with these identities. We may take these parameters to be $a_3$ and $a_7$, and express the remaining variables in terms of $c, a_3, a_7$ as follows:
\begin{equation}\label{ope:first} \begin{split} L(z) W^6(w) & \sim  \frac{a_7 c (22 + 5 c)}{42} (z-w)^{-8} + \frac{2 a_7 (22 + 5 c)}{15} L(w)(z-w)^{-6} + \frac{11a_7 (22 + 5 c)}{210}  \partial L(w)(z-w)^{-5} \\ & + \bigg(a_3 W^4+ \frac{8 a_7}{5} :L^2:+  \frac{2 a_7 (c-4)}{35}  \partial^2 L\bigg)(w)(z-w)^{-4}  \\ & + \bigg(\frac{5 a_3}{16} \partial W^4+ a_7 :(\partial L) L:  + \frac{a_7 (c-4)}{126}  \partial^3 L\bigg)(w)(z-w)^{-3} \\ & + 6 W^6(w)(z-w)^{-2} + \partial W^6(w)(z-w)^{-1}. \end{split} \end{equation}

\begin{equation}\label{ope:second}  \begin{split} W^4(z) W^4(w) & \sim \frac{a_7 c (22 + 5 c) }{840}  (z-w)^{-8} + \frac{a_7 (22 + 5 c)}{105}  L(w)(z-w)^{-6} +\frac{a_7 (22 + 5 c)}{210}  \partial L(w)(z-w)^{-5}
\\ & + \bigg( \frac{a_3}{8}  W^4 + \frac{a_7}{5}  :L^2: + \frac{a_7 (c-4)}{140}  \partial^2 L\bigg)(w)(z-w)^{-4}  
\\ & + \bigg( \frac{a_3}{16} \partial W^4  +\frac{a_7}{5} :(\partial L)L:  + \frac{a_7 (c-4)}{630} \partial^3 L \big)(w) (z-w)^{-3} 
\\ & + W^6(w)(z-w)^{-2} \\ & 
+\bigg( \frac{1}{2} \partial W^6  -\frac{a_7}{60} :(\partial^3 L)L: -\frac{a_7}{20} :(\partial^2 L)\partial L:   - \frac{a_3}{192} \partial^3 W^4 
- \frac{a_7 (c-4)}{10080} \partial^5 L\bigg)(w)(z-w)^{-1}.\end{split} \end{equation}

Note that if we fix a normalization of $W^4$, the parameter $a_7$ will be determined, so up to isomorphism there are really only two parameters at this stage, namely $c$ and $a_3$. It is convenient to leave $a_7$ undetermined for the moment.

Next, the most general OPE of $L$ and $W^8$ is
\begin{equation*} \begin{split} L(z) W^8(w) & \sim c_0 (z-w)^{-10} +  c_1 L(w) (z-w)^{-8} + c_2 \partial L(w) (z-w)^{-7}
\\ & +  \bigg(c_3 W^4 + c_4 :L^2: + c_5 \partial^2 L \bigg)(w)(z-w)^{-6}
\\ & + \bigg(c_6 \partial W^4 + c_7 :(\partial L)L:  + c_8 \partial^3 L \bigg)(w)(z-w)^{-5}
\\ & + \bigg(c_9 W^6 + c_{10} \partial^2 W^4 + c_{11}  :L W^4: + c_{12}  :L^3: \\ & + c_{13}  :(\partial^2 L)L: + 
c_{14}  :(\partial L)^2: + c_{15}  \partial^4 L \bigg)(w)(z-w)^{-4}
 \\ & + \bigg(c_{16}  \partial W^6 + c_{17}  \partial^3 W^4 + c_{18}  :(\partial L) W^4: +c_{19}  L\partial W^4: \\ & + c_{20}  :(\partial L)L^2:  +
c_{21} :(\partial^3 L)L:  + c_{22} :(\partial^2 L) \partial L: + c_{23}\partial^5 L \bigg)(w)(z-w)^{-3}
\\ & + 8 W^8 (w)(z-w)^{-2}+  \partial W^8 (w)(z-w)^{-1} .\end{split} \end{equation*}
The most general OPE of $W^4$ and $W^6$ is

\begin{equation*} \begin{split} W^4(z) W^6(w) &  \sim d_0 (z-w)^{-10} +  d_1 L(w) (z-w)^{-8} + d_2 \partial L(w) (z-w)^{-7}
\\  &  +  \bigg(d_3 W^4 + d_4 :L^2: + d_5 \partial^2 L \bigg)(w)(z-w)^{-6}
\\  & + \bigg(d_6 \partial W^4 + d_7 :(\partial L)L:  + d_8 \partial^3 L \bigg)(w)(z-w)^{-5}
\\ & + \bigg(d_9 W^6 + d_{10} \partial^2 W^4 + d_{11}  :L W^4: + d_{12}  :L^3:  + d_{13}  :(\partial^2 L)L:  \\  & +  d_{14}  :(\partial L)^2:  + d_{15}  \partial^4 L \bigg)(w)(z-w)^{-4}
 \\  &+ \bigg(d_{16}  \partial W^6 + d_{17}  \partial^3 W^4 + d_{18}  :(\partial L) W^4: +d_{19}  L\partial W^4:  + d_{20}  :(\partial L)L^2:  
 \\ & + d_{21} :(\partial^3 L)L:  + d_{22} :(\partial^2 L) \partial L: + d_{23}\partial^5 L \bigg)(w)(z-w)^{-3}  + W^8 (w)(z-w)^{-2}
\\ & + \bigg(d_{24} \partial W^8 + d_{25} \partial^3 W^6 + d_{26} \partial^5 W^4 + d_{27} :(\partial L) W^6: + 
 d_{28} :L \partial W^6:   + d_{29} :(\partial W^4) W^4: 
 \\  & + d_{30} :(\partial L) L W^4: + d_{31} :L^2 \partial W^4:  + d_{32} :(\partial^3 L) W^4:   + d_{33} :(\partial^2 L) \partial W^4:  + 
d_{34} :(\partial L) \partial^2 W^4: \\  & +d_{35} :L \partial^3 W^4:  + d_{36} :(\partial L) L^3:  
  + d_{37} :(\partial^3 L) L^2:  + d_{38} :(\partial^2 L) (\partial L) L:   \\  & + d_{39} :(\partial L)^3:  + 
 d_{40} :(\partial^5 L) L:   +  d_{41} :(\partial^4 L) \partial L: +  d_{42} :(\partial^3 L) \partial^2 L: + d_{43} \partial^7 L \bigg)(w)(z-w)^{-1} .\end{split} \end{equation*}

By imposing the Jacobi relations of type $(L,L,W^8)$ and $(L,W^4,W^6)$ and $(W^4, W^4, W^4)$, it turns out that in addition to the parameters, $c$, $a_3$, $a_7$, there is one additional free parameter which can be chosen to be $d_{28}$. There is a unique solution for the remaining parameters $c_0,\dots, c_{23}$ and $d_0, \dots, d_{27}, d_{29},\dots, d_{43}$ in terms of $c, a_3, a_7, d_{28}$ so that all these Jacobi identities hold. Since these solutions can easily be found by computer, we do not reproduce them here.

\begin{remark} This behavior is more complicated than the corresponding construction of $\cW(c,\lambda)$ in \cite{LI}. For $\cW(c,\lambda)$, once two free parameters appear in the OPE algebra, no new parameters appear at any stage in the calculation when all Jacobi identities of the appropriate weight are imposed.
\end{remark}

Next, we write the most general OPE relations for $$L(z) W^{10}(w),\qquad W^4(z) W^8(w),\qquad W^6(z) W^6(w),$$ with undetermined coefficients, and then impose all Jacobi identities of type $(W^{2i}, W^{2j}, W^{2k})$ for $2i+2j+2k = 14$. Rather surprisingly, it turns out that $d_{28}$ is {\it no longer a free parameter} if we want all these Jacobi identities to hold. Instead, it satisfies the following equation:
\begin{equation}\label{eqn:d28quad}
7 a_3^2 + 960 a_7 (1 - 2 c) + 784 d_{28} a_3 (c-1) + 10752 d_{28}^2 (c-1) (24 + c) = 0.\end{equation}
This is a quadratic equation in $d_{28}$ which has two solutions as long as the discriminant 
$$8960 (c-1) (35 a_3^2 (c-25) + 4608 a_7 (24 + c) (2c-1))$$ does not vanish. If we set this discriminant to zero and solve for $a_7$ and $d_{28}$, obtaining $$a_7 = -\frac{35 a_3^2 (c-25)}{4608 (24 + c) (2c-1)},\qquad d_{28} = -\frac{7 a_3}{192 (24 + c)},$$ our inductive procedure will lead to a freely generated vertex algebra of type $\cW(2,4,\dots)$ depending on $c$ and $a_3$, but the parameter $a_3$ is determined by the scaling of the weight $4$ field. Therefore up to isomorphism, this algebra will only depend on one free parameter $c$. So instead, we should choose the discriminant to be a nonzero perfect square. Up to isomorphism, the algebra $\cW^{\mathrm{ev}}(c,\lambda)$ will depend on two free parameters and will be independent of our choice for this discriminant. It is convenient to introduce a new parameter $\lambda$, and express $a_3, a_7, d_{28}$ in terms of $c$ and $\lambda$ as follows.

\begin{equation} \begin{split} \label{deflambda} a_3 & = 256 (c-1) (24 + c) (2c-1) \lambda, \\  a_7 & =-\frac{640}{63}  (c-1) (24 + c) (2c-1)(-1 + 49 \lambda^2 (c-25) (c-1)), \\  d_{28} & = \frac{4}{21} (2c-1)(5 - 49 \lambda (c-1)). \end{split} \end{equation} The reason for this choice is that the algebra $\cW^{\mathrm{ev}}(c,\lambda)$ will then be defined over the polynomial ring $\mathbb{C}[c,\lambda]$.

\begin{remark}\label{rem:d28twochoicesauto}
Given the choice of $a_3$ and $a_7$ in \eqref{deflambda}, there are two solutions for $d_{28}$ that satisfy \eqref{eqn:d28quad}, namely, the one given above as well as $d'_{28}  = \frac{4}{21} (2c-1)(-5 - 49 \lambda (c-1))$. Replacing $W^4$ with $-W^4$ must replace $a_3$ with $-a_3$, which corresponds to replacing $\lambda$ with $-\lambda$. However, the replacement $\lambda \mapsto -\lambda$ has the effect of sending $d_{28} \mapsto -d'_{28}$, instead of $d_{28} \mapsto -d_{28}$, and hence does not give the same OPE algebra. So without loss of generality, we may fix the choice  \eqref{deflambda} for $d_{28}$, but having made this choice, $\lambda\mapsto -\lambda$ does not give the same algebra up to isomorphism. 

Note also that rescaling $\lambda$ by {\it any} constant $\epsilon \neq 1$ does not correspond to changing the choice of weight $4$ field $W^4$. Indeed, $W^4$ is the unique Virasoro primary field in weight $4$ up to scaling, and rescaling $W^4$ also rescales both $a_3$ and $a_7$.
\end{remark}

Even though the parameter $d_{28}$ which was free at the previous stage is no longer free, it turns out that there is one additional free parameter in the OPE relations, which can be taken to be the coefficient $e$ of $:L \partial W^8:$ appearing in the first-order pole of $W^6(z) W^6(w)$. In other words, if we write down the most general OPEs among $L(z) W^{10}(w)$, $W^4(z) W^8(w)$ and $W^6(z) W^6(w)$ and impose all Jacobi identities of type $(W^{2i}, W^{2j}, W^{2k})$ for $2i+2j+2k = 14$, all structure constants can be written  uniquely in terms of $c$, $\lambda$, and $e$. 
 
Finally, we carry out this procedure at one more stage; that is, we write down the most general OPEs among $$L(z) W^{12}(w),\qquad W^4(z) W^{10}(w),\qquad W^6(z) W^8(w),$$ and impose all Jacobi identities $(W^{2i}, W^{2j}, W^{2k})$ for $2i+2j+2k = 16$. By a long computer computation, one can verify that there is a unique solution for the parameter $e$, namely, 
$$e = \frac{4}{21} (2c-1)(5 - 49 \lambda (c-1)),$$ and all other structure constants have a unique solution in terms of $c$ and $\lambda$. This shows that all terms in the OPE of $W^{2i}(z) W^{2j}(w)$ for $2\leq 2i \leq 2j$ and $2i+2j \leq 14$ are uniquely determined as $\mathbb{C}[c,\lambda]$-linear combinations of normally ordered monomials in the generators $W^{2i}$ for $2i \leq 12$. 
By imposing the identities \eqref{deriv} and \eqref{commutator}, this determines the OPEs $$\partial^a W^{2i}(z) \partial^b W^{2j}(w),\qquad a,b\geq 0,\qquad 2i>2j \geq 2,\qquad 2i+2j \leq 14.$$

\subsection{Step 2: Induction}
We make the following inductive hypothesis. 
\begin{enumerate}
\item For $2 \leq 2i\leq 2j$ and $2i+2j \leq 2n$, all terms in the OPE of $W^{2i}(z) W^{2j}(w)$ have been expressed as $\mathbb{C}[c,\lambda]$-linear combinations of normally monomials in $L, W^4,\dots, W^{2n-2}$ and their derivatives. The OPEs are weight homogeneous, i.e., all terms appearing in $W^{2i}_{(k)} W^{2j}$ have weight $2i+2j-k-1$ .

\item We impose \eqref{deriv} and \eqref{commutator}, which then determines all OPEs $$\partial^a W^{2i}(z) \partial^b W^{2j}(w), \qquad a,b \geq 0,\qquad 2i+2j \leq 2n.$$

\item $a_{2i,2j}$ and $b_{2i,2j}$ are independent of $c, \lambda$ for $2i+2j \leq 2n$. Here $a_{2i,2j}$ and $b_{2i,2j}$ are given by \eqref{eq:aij} and \eqref{eq:bij}, respectively.
\end{enumerate}

By {\it inductive data}, we mean the set of all OPE relations \begin{equation} \partial^a W^{2i}(z) \partial^b W^{2j}(w),\qquad a,b \geq 0,\qquad 2i+2j \leq 2n.\end{equation} If we impose \eqref{nonasswick} and \eqref{ncw}, it follows from our inductive hypothesis that if $\alpha$ is a normally ordered polynomial in $L, W^4, \dots, W^{2n-2i}$ and their derivatives of weight $w \leq 2n+2-2i$, the OPE $W^{2i}(z) \alpha(w)$ is determined by the inductive data.

\begin{lemma} \label{main:1} If we impose the Jacobi identity 
\begin{equation}\label{main:1first} \begin{split} L_{(2)} (W^4_{(0)} W^{2n-2}) & = (L_{(2)} W^4)_{(0)} W^{2n-2} + W^4 _{(0)}( L_{(2)} W^{2n-2}) \\ & + 2(L_{(1)} W^4)_{(1)} W^{2n-2} + (L_{(0)} W^4)_{(2)} W^{2n-2},\end{split} \end{equation} we must have \begin{equation*} b_{4,2n-2} = \frac{3}{2n}.\end{equation*} In particular, $b_{4,2n-2}$ is independent of $c$ and $\lambda$.
\end{lemma}

\begin{proof} Recall that $$W^4_{(0)} W^{2n-2} = b_{4,2n-2} \partial W^{2n} + D_{4,2n-2}$$ where $D_{4,2n-2}$ is a normally ordered polynomial in $L, W^4, \dots, W^{2n-2}$ and their derivatives. Then $$ L_{(2)} (W^4_{(0)} W^{2n-2}) =  b_{4,2n-2} L_{(2)}  \partial W^{2n} + L_{(2)} D_{4,2n-2}.$$ Also, $L_{(2)}D_{4,2n-2}$ has no terms depending on $W^{2n}$ by inductive assumption, so it does not contribute to the coefficient of $\partial W^{2n}$. We have 
$$L_{(2)} \partial W^{2n} = -(\partial L)_{(2)} W^{2n} + \partial (L_{(2)} W^{2n}).$$ Note that $L_{(2)} W^{2n}\in S_{2n+2,2}$ and is therefore not yet known, but by weight considerations it only depends on $L, W^4, \dots, W^{2n-2}$, and their derivatives. Modulo terms which depend on $L, W^4, \dots, W^{2n-2}$ and their derivatives, we have
$$L_{(2)} \partial W^{2n} \equiv - (\partial L)_{(2)} W^{2n} = 2 L_{(1)} W^{2n} = 2(2n) W^{2n}.$$  So the left hand side of \eqref{main:1first} is $4n b_{4,2n-2} W^{2n}$, up to terms which do not depend on $W^{2n}$.

Next, the term $(L_{(2)} W^4)_{(0)} W^{2n-2}$ from \eqref{main:1first} vanishes because $W^4$ is assumed primary. The term $$W^4 _{(0)}( L_{(2)} W^{2n-2})$$ from \eqref{main:1first} does not contribute to the coefficient of $W^{2n}$, since $L_{(2)} W^{2n-2}$ only depends on $L, W^4, \dots, W^{2n-4}$ and their derivatives. The term $$2(L_{(1)} W^4)_{(1)} W^{2n-2}$$ from \eqref{main:1first} contributes $8 W^4_{(1)}W^{2n-2} = 8 W^{2n}$. The term $$(L_{(0)} W^4)_{(2)} W^{2n-2}$$ from \eqref{main:1first} contributes $\partial W^4_{(2)} W^{2n-2} =  -2 W^4_{(1)} W^{2n-2} = -2 W^{2n}$.
Equating the coefficients of $W^{2n}$, we obtain $4n b_{4,2n-2} = 6$, which completes the proof.  \end{proof}

\begin{lemma} \label{main:2}All coefficients $b_{2i,2n+2-2i}$ for $4\leq 2i \leq n+1$ are independent of $c, \lambda$ and are determined by imposing Jacobi relations of type $(W^{2i}, W^{2j}, W^{2k})$ for $2i+2j+2k = 2n+4$.
\end{lemma}

\begin{proof} We first impose the Jacobi relation
\begin{equation} \label{main:2first} W^4_{(1)} (W^4_{(0)} W^{2n-4}) = (W^4_{(1)} W^4)_{(0)} W^{2n-4} + W^4_{(0)} (W^4_{(1)} W^{2n-4}) +  (W^4_{(0)} W^4)_{(1)} W^{2n-4}.\end{equation}
Since $\displaystyle W^4_{(0)} W^{2n-4} = \frac{3}{2n-2} \partial W^{2n-2} + D_{4,2n-4}$, the left hand side of \eqref{main:2first} is 

\begin{equation} \label{main:2second} \begin{split} W^4_{(1)} \bigg(\frac{3}{2n-2} \partial W^{2n-2} + D_{4,2n-4} \bigg) & = -\frac{3}{2n-2}(\partial W^4)_{(1)} W^{2n-2} +\frac{3}{2n-2} \partial  \bigg(W^4_{(1)} W^{2n-2}\bigg) + W^4_{(1)}  D_{4,2n-4}\\ & = \frac{3}{2n-2} \bigg(\frac{3}{2n} \partial W^{2n} +D_{4,2n-2}\bigg) + \frac{3}{2n-2} \partial W^{2n} +  W^4_{(1)}  D_{4,2n-4}\\ & = \frac{3 (3 + 2n)}{2n(2n-2)} \partial W^{2n} +  \frac{3}{2n-2} D_{4,2n-2} +  W^4_{(1)}  D_{4,2n-4}.\end{split}\end{equation}

Next,
\begin{equation} \label{main:2third} \begin{split}(W^4_{(1)} W^4)_{(0)} W^{2n-4} &= W^6_{(0)} W^{2n-4}   = b_{6,2n-4} \partial W^{2n} + D_{6,2n-4},
\\W^4_{(0)} (W^4_{(1)} W^{2n-4}) & = W^4_{(0)} W^{2n-2} = \frac{3}{2n} \partial W^{2n} + D_{4,2n-2},
\\(W^4_{(0)} W^4)_{(1)} W^{2n-4}  & = \bigg(\frac{1}{2} \partial W^6 +\alpha \bigg)_{(1)} W^{2n-4} 
\\ & = -\frac{1}{2} W^6_{(0)} W^{2n-4}  +  \alpha_{(1)} W^{2n-4} \\ & = -\frac{1}{2} \bigg(b_{6,2n-4} \partial W^{2n} + D_{6,2n-4}\bigg) +  \alpha_{(1)} W^{2n-4}  .\end{split}\end{equation} 
Here $\alpha$ does not depend on $W^6$, so that  $\alpha_{(1)} W^{2n-4}$ does not contribute to the coefficient of $\partial W^{2n}$. Collecting terms, we obtain \begin{equation*} b_{6,2n-4} = \frac{30}{2n(2n-2)}.\end{equation*}

Inductively, we impose the Jacobi relation
\begin{equation} \label{main:2fourth}W^4_{(1)} (W^{2i-2}_{(0)} W^{2n+2-2i}) = (W^4_{(1)} W^{2i-2})_{(0)} W^{2n+2-2i} + W^{2i-2}_{(0)} (W^4_{(1)} W^{2n+2-2i}) +  (W^4_{(0)} W^{2i-2})_{(1)} W^{2n+2-2i}.\end{equation} The left side of \eqref{main:2fourth} is 

\begin{equation} \label{main:2fifth} \begin{split} W^4_{(1)} \bigg( b_{2i-2,2n+2-2i} \partial W^{2n-2} + D_{2i-2,2n+2-2i}\bigg) = b_{2i-2,2n+2-2i} W^4_{(1)} \partial W^{2n-2} + W^4_{(1)}D_{2i-2,2n+2-2i}
\\ = b_{2i-2,2n+2-2i} \bigg(\frac{3}{2n} \partial W^{2n}+ D_{4,2n-2} + \partial W^{2n}\bigg) + W^4_{(1)}D_{2i-2,2n+2-2i}. \end{split}\end{equation}
The right side of \eqref{main:2fourth} is
\begin{equation} \label{main2:sixth} b_{2i,2n+2-2i} \partial W^{2n} + D_{2i,2n+2-2i} + b_{2i-2, 2n+4-2i} \partial W^{2n} + D_{2i-2,2n+4-2i}  -\frac{3}{2i} \bigg(b_{2i,2n+2-2i} \partial W^{2n} + D_{2i,2n+2-2i}\bigg).\end{equation}
Recall that $b_{2i-2,2n+2-2i}$ and $D_{2i-2,2n+2-2i}$ are part of our inductive data, and we are assuming inductively that $b_{2i-2,2n+2-2i}$ is independent of $c, \lambda$. Since $D_{2i-2,2n+2-2i}$ is a normally ordered polynomial of weight $2n-1$ in $L, W^4, \dots, W^{2n-4}$ and their derivatives, $W^4_{(1)} D_{2n-2,2n+2-2i}$ is determined by inductive data. This shows that $b_{2i,2n+2-2i}$ is determined from inductive data together with $b_{2i-2,2n+4-2i}$ for $6\leq 2i \leq n+1$. Finally, since $b_{4,2n-2}$ is independent of $c,\lambda$, each $b_{2i,2n+2-2i}$ is independent of $c,\lambda$ as well. \end{proof}

Similarly, we will show that $\displaystyle a_{6,2n-4} =\frac{6}{2n-2}$ and that $a_{2i,2n+2-2i}$ is determined from the Jacobi identities for all $i$; see Lemma \ref{main:4} below. Combining these observations, we obtain

\begin{lemma} \label{main:3} The products $W^{2i}_{(0)} W^{2n+2-2i}$ for $4 \leq 2i \leq n+1$, are determined from inductive data, Jacobi relations of type $(W^{2i}, W^{2j}, W^{2k})$ for $2i+2j+2k = 2n+4$, and elements of $S_{2n+2,1}$.
\end{lemma}

\begin{proof} Since $b_{2i,2n+2-2i}$ is determined from this data, it suffices to show that $D_{2i,2n+2-2i}$ is also determined for $4\leq 2i \leq n+1$. It follows from \eqref{main:2second} and \eqref{main:2third} that modulo terms which are determined from inductive data,
\begin{equation} \label{main:3first} D_{6,2n-4} \equiv -\frac{2 (5-2n)}{2n-2} D_{4,2n-2}.\end{equation}
Using \eqref{main:2fifth} and \eqref{main2:sixth} in the case $i=4$, we get
\begin{equation} \label{main:3second}\frac{30}{(2n-2)(2n-4)} D_{4,2n-2} \equiv D_{6,2n-4} +   \frac{5}{8} D_{8,2n-6},\end{equation} modulo inductive data.
Similarly, \eqref{main:2fifth} and \eqref{main2:sixth} show that there are nontrivial relations $D_{2i,2n+2-2i} \equiv p_i(n) D_{4,2n-2}$ for rational functions $p(n)$ for all $i$, modulo inductive data. So it is enough to find three linearly independent relation between $D_{4,2n-2}$,  $D_{6,2n-4}$, and $D_{8,2n-6}$. From the Jacobi relation
$$W^4_{(0)} (W^6_{(1)} W^{2n-6}) = (W^4_{(0)} W^6)_{(1)} W^{2n-6} + W^6_{(1)} (W^4_{(0)} W^{2n-6}),$$ we get 
$$ \frac{6}{2n-4} D_{4,2n-2} \equiv -\frac{3}{8} D_{8,2n-6} + \frac{4}{2n-4} D_{6,2n-4} + \frac{3}{2n-4} \partial C_{6,2n-4},$$ modulo inductive data. Since $C_{6,2n-4}$ is determined by $S_{2n+2,1}$, we get the relation
\begin{equation} \label{main:3third} \frac{6}{2n-4} D_{4,2n-2} \equiv -\frac{3}{8} D_{8,2n-6} + \frac{3}{2n-4} D_{6,2n-4}\end{equation} modulo inductive data, together with data determined by $S_{2n+2,1}$. Then \eqref{main:3first}-\eqref{main:3third} are the desired linearly independent relations.
\end{proof}

\begin{lemma} \label{main:4}
The set $S_{2n+2,1}$ of products $W^{2i} _{(1)} W^{2n+2-2i}$ is determined from inductive data, Jacobi relations of type $(W^{2i}, W^{2j}, W^{2k})$ for $2i+2j+2k = 2n+4$, and the sets $S_{2n+2,\ell}$ for $\ell\geq 2$.
\end{lemma}

\begin{proof} By assumption $L_{(1)} W^{2n} = 2n W^{2n}$ and $W^4_{(1)} W^{2n-2} = W^{2n}$, so we begin with the case $2i\geq 6$. We impose the Jacobi relation
\begin{equation} \label{main:4first} W^{2n-4} _{(1)} (W^4_{(1)} W^{4}) = W^4 _{(1)} (W^{2n-4}_{(1)} W^{4}) + (W^{2n-4} _{(1)} W^{4})_{(1)}W^4 +  (W^{2n-4} _{(0)} W^4)_{(2)} W^{4}.\end{equation}
By \eqref{commutator}, the left hand side of \eqref{main:4first} is $$W^{2n-4}_{(1)} W^6 = \sum_{i\geq 0} (-1)^{i} \frac{1}{i!} \partial^{i}(W^6_{(i+1)} W^{2n-4}).$$ Therefore modulo derivatives of elements of $S_{2n+2,\ell}$ for $\ell \geq 2$, 
$$W^{2n-4}_{(1)} W^6 \equiv W^6_{(1)} W^{2n-4} = a_{6,2n-4} W^{2n} + C_{6,2n-4}.$$

As for the right hand side, modulo terms which are either derivatives of elements of $S_{2n+2,\ell}$ for $\ell \geq 2$, or are determined by inductive data, we have 
$$W^4 _{(1)} (W^{2n-4}_{(1)} W^{4}) \equiv W^4 _{(1)} (W^{4}_{(1)} W^{2n-4}) =W^4 _{(1)} W^{2n-2} = W^{2n},$$
$$(W^{2n-4} _{(1)} W^{4})_{(1)}W^4 \equiv (W^{4} _{(1)} W^{2n-4})_{(1)}W^4 \equiv W^{2n-2}_{(1)} W^4 \equiv W^{2n},$$
$$(W^{2n-4} _{(0)} W^4)_{(2)} W^{4} =  \bigg(\sum_{i\geq 0} (-1)^{i+1} \frac{1}{i!} \partial^i (W^4_{(i)} W^{2n-4}) \bigg)_{(2)} W^4 $$
$$= \bigg(- \frac{3}{2n-2} \partial W^{2n-2} + \partial W^{2n-2}\bigg)_{(2)} W^4  = \bigg(\frac{6}{2n-2} -2 \bigg)W^{2n-2}_{(1)} W^4$$
$$\equiv  \bigg(\frac{6}{2n-2} -2 \bigg)W^{4}_{(1)} W^{2n-2} = \bigg(\frac{6}{2n-2} -2 \bigg) W^{2n}.$$
Collecting terms, we see that \begin{equation*} a_{6,2n-4} =\frac{6}{2n-2},\end{equation*} and that $C_{6,2n-4}$ is determined modulo inductive data and derivatives of elements of $S_{2n+2,\ell}$ for $\ell \geq 2$.

Next, we impose the Jacobi relation \begin{equation} \label{main:4second} 
W^4 _{(1)} (W^6_{(1)} W^{2n-6}) = W^6 _{(1)} (W^4_{(1)} W^{2n-6}) + (W^4 _{(1)} W^6)_{(1)} W^{2n-6} +  (W^4 _{(0)} W^6)_{(2)} W^{2n-6}.\end{equation} Since $\displaystyle W^6_{(1)} W^{2n-6} =  \frac{6}{2n-4} W^{2n-2} + C_{6,2n-6}$, where $C_{6,2n-6}$ depends only on $L,W^4,\dots, W^{2n-4}$ and their derivatives and is determined by inductive data, the left hand side of \eqref{main:4second} is
$$ W^4_{(1)} \bigg(\frac{6}{2n-4} W^{2n-2} + C_{6,2n-6}\bigg) \equiv   \frac{6}{2n-4} W^{2n}.$$ For the right hand side, we have
$$W^6 _{(1)} (W^4_{(1)} W^{2n-6}) = W^6 _{(1)} W^{2n-4} = \frac{6}{2n-2} W^{2n} + C_{6,2n-4},$$
$$(W^4 _{(1)} W^6)_{(1)} W^{2n-6} = W^8_{(1)} W^{2n-6} = a_{8,2n-6} W^{2n} + C_{8,2n-6},$$
$$(W^4 _{(0)} W^6)_{(2)} W^{2n-6} = \bigg(\frac{3}{8}
\partial W^8 + \beta \bigg)_{(2)} W^{2n-6}$$
$$ \equiv -\frac{6}{8} \bigg(a_{8,2n-6} W^{2n} + C_{8,2n-6}\bigg).$$ Here $\beta$ depends only on $L,W^4, W^6$, so $\beta_{(2)} W^{2n-6}$ does not depend on $W^{2n}$.
It is immediate that $$a_{8,2n-6} = \frac{48}{(2n-4) (2n-2)},$$ and that $C_{8,2n-6}$ is determined modulo inductive data and derivatives of elements of $S_{2n+2,\ell}$ for $\ell \geq 2$.
Similarly, for $2i>6$, by imposing the Jacobi relation 
\begin{equation} \label{main:4third} 
W^4 _{(1)} (W^{2i}_{(1)} W^{2n-2i}) = W^{2i} _{(1)} (W^4_{(1)} W^{2n-2i}) + (W^4 _{(1)} W^{2i})_{(1)} W^{2n-2i} +  (W^4 _{(0)} W^{2i})_{(2)} W^{2n-2i},\end{equation} the same argument shows that both $a_{2i+2,2n-2i}$ and $C_{2i+2,2n-2i}$ are determined modulo inductive data and derivatives of elements of $S_{2n+2,\ell}$ for $\ell \geq 2$. This shows that $S_{2n+2,1}$ is determined modulo this data. \end{proof}

\begin{lemma} \label{main:5}The set $S_{2n+2, 2}$ of products $W^{2i} _{(2)} W^{2n+2-2i}$ is determined from inductive data, Jacobi relations of type $(W^{2i}, W^{2j}, W^{2k})$ for $2i+2j+2k = 2n+4$, and the sets $S_{2n+2, \ell}$ for $\ell \geq 3$.\end{lemma}

\begin{proof} First, since $L_{(2)} W^{2n} = L_{(2)} (W^4_{(1)}W^{2n-2})$, we impose the Jacobi identity
\begin{equation} \label{main:5first} \begin{split} L_{(2)} (W^4_{(1)}W^{2n-2}) & = W^4_{(1)}(L_{(2)} W^{2n-2}) +(L_{(2)} W^4)_{(1)} W^{2n-2}
\\ & +2(L_{(1)} W^4)_{(2)} W^{2n-2} +(L_{(0)} W^4)_{(3)} W^{2n-2}.\end{split} \end{equation}
 
By inductive assumption, $L_{(2)} W^{2n-2}$ is known and is expressible in terms of $L, W^{2i}$ for $4\leq 2i\leq 2n-4$. Then $W^4_{(1)}(L_{(2)} W^{2n-2})$ is also determined by inductive data. The term $(L_{(2)} W^4)_{(1)} W^{2n-2}$ vanishes because $W^4$ is primary, and the remaining terms are expressible in terms of $W^4_{(2)} W^{2n-2}$ together with inductive data. So determining $L_{(2)} W^{2n}$ is equivalent to determining $W^4_{(2)} W^{2n-2}$.

To determine $W^4_{(2)} W^{2n-2}$, we impose the Jacobi relation
\begin{equation} \label{main:5second}  W^{2n-4}_{(1)} (W^4_{(2)} W^4) =  W^4_{(2)} (W^{2n-4}_{(1)} W^4) + (W^{2n-4}_{(1)} W^4)_{(2)} W^4 + (W^{2n-4}_{(0)} W^4)_{(3)} W^4.\end{equation}
Since $W^4_{(2)} W^4$ only depends on $L, W^4$ and their derivatives, the left hand side of \eqref{main:5second} is determined by inductive data. Next, by \eqref{commutator}, we have 
$$W^4_{(2)} (W^{2n-4}_{(1)} W^4) + (W^{2n-4}_{(1)} W^4)_{(2)} W^4 = \sum_{i\geq 1} (-1)^{i+1}\frac{1}{i!} \partial^i\bigg((W^{2n-4}_{(1)} W^4)_{(i+2)} W^4\bigg).$$
Since $W^{2n-4}_{(1)} W^4 = W^{2n-2}$ modulo terms which depend only on $L, W^4, \dots, W^{2n-4}$, and is known by inductive data,
the sum $W^4_{(2)} (W^{2n-4}_{(1)} W^4) + (W^{2n-4}_{(1)} W^4)_{(2)} W^4$ in  \eqref{main:5second} is determined by inductive data together with derivatives of elements of $S_{2n+2,\ell}$ for $\ell \geq 3$. Finally, the remaining term in \eqref{main:5second} is
$$(W^{2n-4}_{(0)} W^4)_{(3)} W^4 \equiv - (W^4_{(0)} W^{2n-4})_{(3)} W^4 = -\frac{3}{2n-2} \partial W^{2n-2}_{(3)} W^4 = \frac{9}{2n-2} W^{2n-2}_{(2)} W^4 $$ $$ = \frac{9}{2n-2} \sum_{i\geq 0} (-1)^{i+1}\frac{1}{i!} \partial^i (W^4_{(i+2)} W^{2n-2}),$$ modulo inductive data. Therefore $W^4_{(2)} W^{2n-2}$ is expressible in terms of inductive data together with derivatives of elements of $S_{2n+2,\ell}$ for $\ell >2$. Since $L_{(2)} W^{2n}$ can be expressed in terms of $W^4_{(2)} W^{2n-2}$, the same holds for $L_{(2)} W^{2n}$.

Next, for $2i\geq 4$ we impose the Jacobi relation 
\begin{equation} \label{main:5third} \begin{split}W^4_{(2)} (W^{2i}_{(1)} W^{2n-2i} ) & = W^{2i}_{(1)} (W^4_{(2)} W^{2n-2i}) + (W^4_{(2)} W^{2i})_{(1)} W^{2n-2i} + 2(W^4_{(1)} W^{2i})_{(2)} W^{2n-2i}  \\ & + (W^4_{(0)} W^{2i})_{(3)} W^{2n-2i}.\end{split} \end{equation}
This allows us to express $W^{2i}_{(2)} W^{2n+2-2i}$ for all $i$ in terms of inductive data together with derivatives of $S_{2n+2,\ell}$ for $\ell \geq 3$. \end{proof}

\begin{lemma} \label{main:6} For all $k>2$, the set $S_{2n+2, k}$ of products $W^{2i} _{(k)} W^{2n+2-2i}$ is determined from inductive data, Jacobi relations of type $(W^{2i}, W^{2j}, W^{2k})$ for $2i+2j+2k = 2n+4$, and the sets $S_{2n+2, \ell}$ for $\ell >k $.\end{lemma}
 
\begin{proof} The argument is the same as the proof of Lemma \ref{main:5}. Imposing the Jacobi relation
\begin{equation} \label{main:6first} L_{(k)} (W^4_{(1)}W^{2n-2}) = W^4_{(1)}(L_{(k)} W^{2n-2}) +\sum_{i\geq 0} \binom{k}{i} (L_{(i)} W^4)_{(k+1-i)} W^{2n-2} \end{equation} shows that determining $L_{(k)} W^{2n}$ is equivalent to determining $W^4_{(k)} W^{2n-2}$. Imposing the Jacobi relation
\begin{equation} \label{main:6second}  W^{2n-4}_{(1)} (W^4_{(k)} W^4) =  W^4_{(k)} (W^{2n-4}_{(1)} W^4) + (W^{2n-4}_{(1)} W^4)_{(k)} W^4 + (W^{2n-4}_{(0)} W^4)_{(k+1)} W^4 \end{equation} shows that $W^4_{(k)} W^{2n-2}$, and hence $L_{(k)} W^{2n}$, are determined from inductive data together with $S_{2n+2,\ell}$ for $\ell >k$. Finally, imposing the Jacobi relation
\begin{equation} \label{main:6third} W^4_{(k)} (W^{2i}_{(1)} W^{2n-2i} )= W^{2i}_{(1)} (W^4_{(k)} W^{2n-2i}) + \sum_{r\geq 0} \binom{k}{r} (W^4_{(r)} W^{2i})_{(k+1-r)} W^{2n-2i} \end{equation} shows that $W^{2i}_{(k)} W^{2n+2-2i}$ can be expressed in terms of inductive data together with $S_{2n+2,\ell}$ for $\ell > k$. \end{proof} 

This process terminates after finitely many steps since all elements of $S_{2n+2,k}$ vanish for $k >2n+1$. Therefore we have proven the following.

\begin{thm} There exists a nonlinear conformal algebra $\cL^{\mathrm{ev}}(c,\lambda)$ over the ring $\mathbb{C}[c,\lambda]$ satisfying skew-symmetry, whose universal enveloping vertex algebra $\cW^{\mathrm{ev}}(c,\lambda)$ has the following properties.

\begin{enumerate}
\item $\cW^{\mathrm{ev}}(c,\lambda)$ has conformal weight grading $$\cW^{\mathrm{ev}}(c,\lambda) = \bigoplus_{n\geq 0} \cW^{\mathrm{ev}}(c,\lambda)[n],$$ and $\cW^{\mathrm{ev}}(c,\lambda)[0] \cong \mathbb{C}[c,\lambda]$. 
\item $\cW^{\mathrm{ev}}(c,\lambda)$ is strongly generated by $\{L,W^{2i}|\ i\geq 2\}$.

\item The OPE relations for $$L(z) W^{2i}(w),\qquad W^{2j}(z) W^{2k}(w),\qquad 2i\leq 12,\qquad 2j+2k \leq 14,$$ which are consequences of the Jacobi identities of type $(W^{2i}, W^{2j}, W^{2k})$ for $2i+2j+2k \leq 16$, all hold.

\item The additional Jacobi identities \eqref{main:1first}-\eqref{main:2first}, \eqref{main:2fourth}, \eqref{main:4first}-\eqref{main:6third}, which appear in the above lemmas, all hold as consequences of \eqref{deriv}-\eqref{ncw} alone.
\end{enumerate}
Moreover, $\cW^{\mathrm{ev}}(c,\lambda)$ is the unique initial object in the category of vertex algebras with the above properties. 
\end{thm}

It is not yet apparent that all Jacobi identities of the form $(W^{2i}, W^{2j}, W^{2k})$ hold in $\cW^{\mathrm{ev}}(c,\lambda)$ as consequences of \eqref{deriv}-\eqref{ncw} alone, or equivalently, that $\cL^{\mathrm{ev}}(c,\lambda)$ is a nonlinear Lie conformal algebra and $\cW^{\mathrm{ev}}(c,\lambda)$ is freely generated.

\subsection{Step 3: Free generation of $\cW^{\mathrm{ev}}(c,\lambda)$} In order to prove that $\cW^{\mathrm{ev}}(c,\lambda)$ is freely generated, we need to consider certain simple quotients of $\cW^{\mathrm{ev}}(c,\lambda)$. First, let $$I \subseteq \mathbb{C}[c,\lambda] \cong \cW^{\mathrm{ev}}(c,\lambda)[0]$$ be an ideal, and let $I\cdot \cW^{\mathrm{ev}}(c,\lambda)$ denote the vertex algebra ideal generated by $I$. The quotient
\begin{equation}\label{eq:quotientbyi} \cW^{\mathrm{ev},I}(c,\lambda) = \cW^{\mathrm{ev},I}(c,\lambda) / I\cdot \cW^{\mathrm{ev}}(c,\lambda)\end{equation} 
has strong generators $\{L, W^{2i}|\ i\geq 2\}$ satisfying the same OPE relations as the corresponding generators of $\cW^{\mathrm{ev}}(c,\lambda)$ where all structure constants in $\mathbb{C}[c,\lambda]$ are replaced by their images in $\mathbb{C}[c,\lambda] / I$.

We now consider localizations of $\cW^{\mathrm{ev},I}(c,\lambda)$. Let $D\subseteq \mathbb{C}[c,\lambda]/I$ be a multiplicatively closed subset, and let $R = D^{-1}(\mathbb{C}[c,\lambda]/I)$ denote the localization of $\mathbb{C}[c,\lambda]/I$ along $D$. Then we have the localization of $\mathbb{C}[c,\lambda]/I$-modules
$$\cW_R^{\mathrm{ev},I}(c,\lambda) = R \otimes_{\mathbb{C}[c,\lambda]/I} \cW^{\mathrm{ev},I}(c,\lambda),$$ which is a vertex algebra over $R$.

\begin{thm} \label{thm:simplequotient} Let $I$, $D$, and $R$ be as above, and let $\cW$ be a simple vertex algebra over $R$ with the following properties.
\begin{enumerate}
\item $\cW$ is generated by a Virasoro field $\tilde{L}$ of central charge $c$ and a weight $4$ primary field $\tilde{W}^4$.
\item Setting $\tilde{W}^{2i}  = \tilde{W}^4_{(1)} \tilde{W}^{2i-2}$ for $2i\geq 6$, the OPE relations for $\tilde{L}(z) \tilde{W}^{2i}(w)$ and $\tilde{W}^{2j}(z) \tilde{W}^{2k}(w)$ for $2i\leq 12$ and $2j+2k \leq 14$ are the same as in $\cW^{\mathrm{ev}}(c,\lambda)$ if the structure constants are replaced with their images in $R$. 
\end{enumerate}
Then $\cW$ is the simple quotient of $\cW^{\mathrm{ev},I}_R(c,\lambda)$ by its maximal graded ideal $\cI$.
\end{thm}

\begin{proof} The assumption that $\{\tilde{L}, \tilde{W}^{2i}|\ i\geq 2\}$ satisfy the above OPE relations is equivalent to the statement that the Jacobi relations of type $(\tilde{W}^{2i}, \tilde{W}^{2j}, \tilde{W}^{2k})$ for $2i+2j+2k \leq 16$ hold in the corresponding nonlinear conformal algebra, which is possibly degenerate. Then all OPE relations among the generators $\{L, W^{2i}|\ i\geq 2\}$ of $\cW^{\mathrm{ev},I}_R(c,\lambda)$ must also hold among the fields $\{\tilde{L}, \tilde{W}^{2i}|\ i \geq 2\}$, since they are formal consequences of these OPE relations together with Jacobi identities, which hold in $\cW$. It follows that $\{\tilde{L},\tilde{W}^{2i}|\ i\geq 2\}$ closes under OPE and strongly generate a vertex subalgebra $\cW' \subseteq \cW$, which must coincide with $\cW$ since $\cW$ is assumed to be generated by $\{\tilde{L},\tilde{W}^4\}$. So $\cW$ has the same strong generating set and OPE algebra as $\cW^{\mathrm{ev},I}_R(c, \lambda)$. Since $\cW$ is simple and the category of vertex algebras over $R$ with this strong generating set and OPE algebra has a unique simple graded object, $\cW$ must be the simple quotient of $\cW^{\mathrm{ev},I}_R(c, \lambda)$ by its maximal graded ideal. \end{proof}

\subsection{Some quotients of $\cW^{\mathrm{ev}}(c,\lambda)$}
We now consider a family of vertex algebras that arise as quotients of $\cW^{\mathrm{ev}}(c,\lambda)$ in this way. Let 
\begin{equation} \label{eq:ckn} \cC^k(n) = \text{Com}(V^k(\gs\gp_{2n}), V^k(\go\gs\gp(1|2n)),\end{equation} that is, the coset of the affine vertex algebra of $\gs\gp_{2n}$ inside the affine vertex superalgebra of $\go\gs\gp(1|2n)$. We shall regard $\cC^k(n)$ as a one-parameter vertex algebra, defined over a localization of the ring $\mathbb{C}[k]$. It has central charge 
$$c(k) = -\frac{k n (3 + 4 k + 2 n)}{(1 + k + n) (1 + 2 k + 2 n)},$$ 
and has a well-defined limit 
$$\lim_{k\ra \infty} \cC^k(n) \cong \cA(n)^{\text{Sp}_{2n}},$$ where $\cA(n)$ denotes the rank $n$ symplectic fermion algebra and $\cA(n)^{\text{Sp}_{2n}}$ denotes the orbifold under the full automorphism group $\text{Sp}_{2n}$; see Example 7.3 of \cite{CLi}.  This orbifold has the following properties.
\begin{enumerate}
\item $\cA(n)^{\text{Sp}_{2n}}$ is generated by the weights $2$ and $4$ fields. This follows from Theorem 3.2 of \cite{CLII}. Note that the weight $4$ field in \cite{CLII} is not primary with respect to $L$, but this result still holds if we replace it by the unique primary field. 

\item $\cA(n)^{\text{Sp}_{2n}}$ is freely generated of type $\cW(2,4,\dots, 2n)$ and has graded character 
\begin{equation} \chi(\cA(n)^{\text{Sp}_{2n}},q) = \sum_{n\geq 0} \text{dim}( \cA(n)^{\text{Sp}_{2n}}[n]) q^n =  \prod_{i=1}^n \prod_{k\geq 0} \frac{1}{1-q^{2i +k}}.\end{equation} This appears as Corollary 5.7 of \cite{CLII}.
\item $\cA(n)^{\text{Sp}_{2n}}$ is a simple vertex algebra; see \cite{KR}, Theorem 1.1 and Remark 1.1, or \cite{DLM}.
 \end{enumerate}
 
\begin{lemma} \label{lem:ckn} As a one-parameter vertex algebra, $\cC^k(n)$ inherits these properties; it is generated by the weights $2$ and $4$ fields, freely generated of type $\cW(2,4,\dots, 2n)$, and simple. Equivalently, $\cC^k(n)$ inherits these properties for generic values of $k \in \mathbb{C}$.
\end{lemma}

\begin{proof}
The argument is the same to the proof of Corollary 8.6 of \cite{CLI}.
\end{proof}

As above, set $W^{2i} = W^4_{(1)} W^{2i-2}$ for $i\geq 3$. For $n\geq 7$, it follows from Weyl's second fundamental theorem of invariant theory for the standard representation of $\gs\gp_{2n}$ \cite{We}, that among the generators $\{L,W^{2i}|\ 2 \leq  i\leq 7\}$ there can be no normally ordered polynomial relations. Therefore for $n\geq 7$, and $k$ generic, all Jacobi identities of type $(W^{2r}, W^{2s}, W^{2t})$ among the generators $\{L, W^{2i}|\ i\geq 2\}$ of $\cC^k(n)$ must hold as consequences of \eqref{deriv}-\eqref{ncw} alone, for $2r+2s+2t \leq 16$. For $2\leq n \leq 6$, it is straightforward to verify by computer that among the generators $\{L, W^{2i}|\ i\geq 2\}$ of $\cA(n)^{\text{Sp}_{2n}}$, all Jacobi relations of type $(W^{2r}, W^{2s}, W^{2t})$ hold as consequences of \eqref{deriv}-\eqref{ncw} alone, for $2r+2s+2t \leq 16$. This property is then inherited by $\cC^k(n)$ for $2\leq n \leq 6$ and $k$ generic, since after suitably rescaling, a nonvanishing Jacobi relation among the corresponding fields in $\cC^k(n)$ would remain nontrivial in the limit as $k\ra \infty$.

\begin{thm} \label{wprinquot} For all $n\geq 2$, the one-parameter vertex algebra $\cC^k(n)$ is the simple quotient of $\cW^{\mathrm{ev},I}_R(c,\lambda)$ for some prime ideal $I \subseteq \mathbb{C}[c,\lambda]$ and some localization $R$ of $\mathbb{C}[c,\lambda] / I$. The maximal proper graded ideal $\cI \subseteq \cW^{\mathrm{ev},I}_R(c,\lambda)$ is generated by a singular vector of weight $2n+2$ of the form
\begin{equation} \label{sing:genfirstproof} W^{2n+2} - P(L, W^4, \dots, W^{2n}),\end{equation} where $P$ is a normally ordered polynomial in $L, W^4,\dots, W^{2n}$, and their derivatives.
\end{thm}
 
\begin{proof} Since $\cC^k(n)$ is generated by $L$ and $W^{4}$, strongly generated by $\{L, W^{2i}|\ i\geq 2\}$, all Jacobi identities of type $(W^{2r}, W^{2s}, W^{2t})$ hold for $2r+2s+2t \leq 16$, and all structure constants are rational functions of $k$, it follows from Theorem \ref{thm:simplequotient} that there exists some rational function $$\lambda(k) = \frac{f(k)}{g(k)}$$ such that the OPE relations of $L(z) W^{2r}(w)$ and $W^{2s}(z) W^{2t}(w)$ for $2r\leq 12$ and $2s+2t \leq 14$ are satisfied if $c$ and $\lambda$ are replaced by $c(k)$ and $\lambda(k)$, respectively. 

Letting $J \subseteq \mathbb{C}[c,\lambda,k]$ be the ideal generated by $$c(1 + k + n) (1 + 2 k + 2 n) + k n (3 + 4 k + 2 n),\qquad g(k) \lambda - f(k),$$ we can eliminate $k$ to obtain a prime ideal $I \subseteq \mathbb{C}[c,\lambda]$ such that some localization $R$ of $\mathbb{C}[c,\lambda] /I$ is isomorphic to some localization $D^{-1} \mathbb{C}[k]$. Here $D$ is a multiplicatively closed subset of $\mathbb{C}[k]$ containing  $(1 + k + n)$, $(1 + 2 k + 2 n)$ and all roots $g(k)$.

Since $\cC^k(n)$ is of type $\cW(2,4,\dots, 2n)$, there must be a singular vector in $\cW^{\mathrm{ev},I}_R(c,\lambda)$ of weight $2n+2$ such that the coefficient of $W^{2n+2}$ is nonzero. If this coefficient is not invertible in $R$, we may localize $R$ further (without changing notation) so it becomes invertible, and the singular vector has the form \eqref{sing:genfirstproof}. Since $\cC^k(n)$ is simple as a vertex algebra over $R$, it must be the simple quotient $\cW^{\mathrm{ev},I}_{R}(c,\lambda)/ \cI$, by Theorem \ref{thm:simplequotient}. Here $\cI$ is the maximal proper graded ideal of $\cW^{\mathrm{ev},I}_{R}(c,\lambda)$. Finally, we need to show that \eqref{sing:genfirstproof} generates $\cI$.

Let $\cI' \subseteq \cI$ be the ideal in $\cW^{\mathrm{ev},I}_R(c,\lambda)$ generated by \eqref{sing:genfirstproof}. Since $\cC^k(n) \cong\cW^{\mathrm{ev},I}_R(c,\lambda)/ \cI$, $\cC^k(n)$ is also a quotient of $\cW^{\mathrm{ev},I}_R(c,\lambda) / \cI'$, and $\cW^{\mathrm{ev},I}_R(c,\lambda)/ \cI'$ is of type $\cW(2,4,\dots, 2n)$; see Theorem \ref{decoup} below. Since $\cC^k(n)$ is freely generated, there can be no more relations in $\cW^{\mathrm{ev},I}_R(c,\lambda)/ \cI$ than in $\cW^{\mathrm{ev},I}_R(c,\lambda) / \cI'$, so $\cI'  = \cI$. \end{proof}

\begin{cor} \label{cor:freegeneration} All Jacobi identities of type $(W^{2i}, W^{2j}, W^{2k})$ hold as consequences of \eqref{deriv}-\eqref{ncw} alone in $\cL^{\mathrm{ev}}(c,\lambda)$, so $\cL^{\mathrm{ev}}(c,\lambda)$ is a nonlinear Lie conformal algebra with generators $\{L, W^{2i}|\ i\geq 2\}$. Equivalently, $\cW^{\mathrm{ev}}(c,\lambda)$ is freely generated by $\{L, W^{2i}|\ i\geq 2\}$ and has graded character
\begin{equation} \label{grchar} \chi(\cW^{\mathrm{ev}}(c,\lambda), q) = \sum_{n\geq 0} \text{rank}_{\mathbb{C}[c,\lambda]}( \cW^{\mathrm{ev}}(c,\lambda)[n]) q^n =  \prod_{i\geq 1} \prod_{k\geq 0} \frac{1}{1-q^{2i +k}}.\end{equation} For any prime ideal $I \subseteq \mathbb{C}[c,\lambda]$, $\cW^{\mathrm{ev},I}(c,\lambda)$ is freely generated by $\{L, W^{2i}|\ i\geq 2\}$ as a vertex algebra over $\mathbb{C}[c,\lambda] / I$, and
\begin{equation} \label{grcharq} \chi(\cW^{\mathrm{ev},I}(c,\lambda), q) = \sum_{n\geq 0} \text{rank}_{\mathbb{C}[c,\lambda]/I}( \cW^{\mathrm{ev},I}(c,\lambda)[n]) q^n = \prod_{i\geq 1} \prod_{k\geq 0} \frac{1}{1-q^{2i +k}}.\end{equation} 
For any localization $R = D^{-1}(\mathbb{C}[c,\lambda]/I)$ along a multiplicatively closed set $D \subseteq \mathbb{C}[c,\lambda]/I$, $\cW^{\mathrm{ev},I}_R(c,\lambda)$ is freely generated by $\{L, W^{2i} |\ i\geq 2\}$ and \begin{equation} \label{grcharql} \chi(\cW^{\mathrm{ev},I}_R(c,\lambda), q) = \sum_{n\geq 0} \text{rank}_{R}( \cW^{\mathrm{ev},I}_R(c,\lambda)[n]) q^n =\prod_{i\geq 1} \prod_{k\geq 0} \frac{1}{1-q^{2i +k}}.\end{equation} 
\end{cor}

\begin{proof} If some Jacobi identity of type $(W^{2i}, W^{2j}, W^{2k})$ does not hold as a consequence of \eqref{deriv}-\eqref{ncw}, there would be a null vector of weight $2N$ in $\cW^{\mathrm{ev}}(c,\lambda)$ for some $N$. Then $\text{rank}_{\mathbb{C}[c,\lambda]}(\cW^{\mathrm{ev}}(c,\lambda)[2N])$ would be smaller than that given by \eqref{grchar}, and the same would hold in any quotient of $\cW^{\mathrm{ev}}(c,\lambda)[2N]$, as well as any localization of such a quotient. But since $\cC^k(N)$ is a localization of such a quotient and is freely generated of type $\cW(2,4,\dots, 2N)$, this is impossible. 
\end{proof}

\begin{cor} \label{cor:simplicity} $\cW^{\mathrm{ev}}(c,\lambda)$ is a simple vertex algebra.
\end{cor}

\begin{proof} If $\cW^{\mathrm{ev}}(c,\lambda)$ is not simple, it would have a singular vector $\omega$ in weight $2N$ for some $N$. Let $p \in \mathbb{C}[c,\lambda]$ be an irreducible polynomial and let $I = (p) \subseteq \mathbb{C}[c,\lambda]$. By rescaling if necessary, we can assume without loss of generality that $\omega$ is not divisible by $p$, and hence descends to a nontrivial singular vector in $\cW^{\mathrm{ev},I}(c,\lambda)$. Then for any localization $R$ of $\mathbb{C}[c,\lambda]/ I$, the simple quotient of $\cW^{\mathrm{ev},I}_R(c,\lambda)$ would have a smaller weight $2N$ submodule than $\cW^{\mathrm{ev},I}(c,\lambda)$ for all such $I$. This contradicts the fact that $\cC^k(N)$ is such a quotient.
\end{proof}

\begin{cor}\label{cor:trivialAut}
The automorphism group of $\cW^{\mathrm{ev}}(c,\lambda)$ is trivial.
\end{cor}
\begin{proof}
By definition, an automorphism $g$ preserves the Virasoro generator,
$L$. Therefore, under $g$, Virasoro primaries must map to Virasoro primaries. It is easy to check that in weight $4$, there is only one such primary vector up to scaling, namely $W^4$, therefore $g(W^4)=tW^4$ for some $t\in \mathbb{C}[c,\lambda]\backslash \{0\}$.
Considering the OPE of $g(W^4)$ with $g(W^4)$ forces $t=1$.
Since $\cW^{\mathrm{ev}}(c,\lambda)$ is generated by $L$ and  $W^4$, $g$ must now be trivial.
\end{proof}

\section{Classification of vertex algebras of type $\cW(2,4,\dots, 2N)$}
Theorem \ref{thm:simplequotient} and Corollary \ref{cor:simplicity} reduce the classification of vertex algebras of type $\cW(2,4,\dots)$ with the above properties to the classification of ideals $I \subseteq \mathbb{C}[c,\lambda]$ such that $\cW^{\mathrm{ev},I}(c,\lambda)$ is not simple.  In this section, we restrict to the case where $I = (p)$ for some irreducible $p \in \mathbb{C}[c,\lambda]$, so that $\cW^{\mathrm{ev},I}(c,\lambda)$ is a one-parameter vertex algebra in the sense that $R$ has Krull dimension one. Later, in Section \ref{section:coincidences} we will consider the case where $I$ is a maximal ideal.

It follows from Corollary \ref{cor:freegeneration} that $\cW^{\mathrm{ev}}(c,\lambda)[n]$ is a free $\mathbb{C}[c,\lambda]$-module whose rank is given by \eqref{grchar}. Recall that it has a symmetric bilinear form $$\langle,\rangle_n: \cW^{\mathrm{ev}}(c,\lambda)[n]\otimes_{\mathbb{C}[c,\lambda]} \cW^{\mathrm{ev}}(c,\lambda)[n] \ra \mathbb{C}[c,\lambda],\qquad \langle \omega,\nu \rangle_n = \omega_{(2n-1)}\nu,$$ 
and that the Shapovalov determinant $\mathrm{det}_n  \in \mathbb{C}[c,\lambda]$ is the determinant of the matrix of this form. Also, recall that an irreducible element $p \in \mathbb{C}[c,\lambda]$ lies in the level $n$ Shapovalov spectrum if $p$ divides $\mathrm{det}_n$ but does not divide $\mathrm{det}_m$ for any $m<n$. 

Let $p$ be an irreducible factor of $\mathrm{det}_{2N+2}$ of level $2N+2$. Letting $I = (p) \subseteq \mathbb{C}[c,\lambda]$, $\cW^{\mathrm{ev},I}(c,\lambda)$ will then have a singular vector in weight $2N+2$. Often, the coefficient of $W^{2N+2}$ in this singular vector is nonzero. By inverting this coefficient, we obtain a localization $R$ of $\mathbb{C}[c,\lambda] / I$ such that this singular vector has the form 
\begin{equation} \label{elim:first} W^{2N+2} - P_{2N+2}(L, W^4,\dots, W^{2N})\end{equation} in $\cW^{\mathrm{ev},I}_R(c,\lambda)$. Here $P_{2N+2}$ is a normally ordered polynomial in the fields $L, W^4, \dots, W^{2N}$ and their derivatives, with coefficients in $R$. This implies that $W^{2N+2}$ decouples in the quotient of $\cW^{\mathrm{ev},I}_R(c,\lambda) / \cJ$, where $\cJ$ denotes the vertex algebra ideal generated by \eqref{elim:first}. In other words, we have the relation $$W^{2N+2} = P_{2N+2}(L, W^4,\dots, W^{2N})$$ in $\cW^{\mathrm{ev},I}_R(c,\lambda) / \cJ$. Applying the operator $W^4_{(1)}$ to this relation and using the fact that $W^4_{(1)} W^{2N+2} = W^{2N+4}$ and $W^4_{(1)} W^{2N} = W^{2N+2}$, we obtain a relation $$W^{2N+4} = Q_{2N+4}(L, W^4, \dots, W^{2N+2}).$$ If the terms $\partial^2 W^{2N+2}$ or $:L W^{2N+2}:$ appear in $Q_{2N+4}$, they can be eliminated using \eqref{elim:first} to obtain a relation 
$$W^{2N+4} = P_{2N+4}(L, W^4, \dots, W^{2N}).$$
Inductively, by applying $W^4_{(1)}$ repeatedly and using \eqref{elim:first} to eliminate $W^{2N+2}$ or $:LW^{2N+2}:$ if necessary, we obtain relations 
$$W^{2m} = P_{2m}(L, W^4, \dots, W^{2N})$$ in $\cW^{\mathrm{ev},I}_R(c,\lambda) / \cJ$, for all $2m > 2N+2$. This implies 

\begin{thm} \label{decoup}
Let $p$ be an irreducible factor of $\mathrm{det}_{2N+2}$ of level $2N+2$, and let $I = (p)$. Suppose that there exists a localization $R$ of $\mathbb{C}[c,\lambda] / I$ such that $\cW^{\mathrm{ev},I}_R(c,\lambda)$ has a singular vector of the form \begin{equation} \label{eq:decoup}W^{2N+2} - P_{2N+2}(L, W^4,\dots, W^{2N}).\end{equation} Let $\cJ \subseteq \cW^{\mathrm{ev},I}_R(c,\lambda)$ be the vertex algebra ideal generated by \eqref{eq:decoup}. Then the quotient $\cW^{\mathrm{ev},I}_R(c,\lambda) / \cJ$ has a minimal strong generating set $$\{L, W^{2i}|\ 2 \leq i \leq 2N\},$$ and in particular is of type $\cW(2,4,\dots, 2N)$. 
\end{thm}

The ideal $\cJ$ is sometimes (but not necessarily) the maximal graded ideal $\cI \subseteq \cW^{\mathrm{ev},I}_R(c,\lambda)$. However, the assumption that $p$ does not divide $\mathrm{det}_m$ for $m<2N+2$ implies that there are no singular vectors in weight $m<2N+2$, so there can be no decoupling relations of the form $W^{2m} = P_m(L,W^4,\dots W^{2m})$ for $2m<2N+2$. Therefore the simple quotient $\cW^{\mathrm{ev},I}_R(c,\lambda) / \cI$ is also of type $\cW(2,4\dots,2N)$.

\begin{thm} For all $N\geq 2$, there are finitely many isomorphism classes of simple one-parameter vertex algebras of type $\cW(2,4,\dots, 2N)$, which satisfy the hypotheses of Theorem \ref{thm:simplequotient}. \end{thm}

\begin{proof} 

Any such vertex algebra must be the simple quotient of $\cW^{\mathrm{ev},I}_R(c,\lambda)$ for some ideal $I  = (p) \subseteq \mathbb{C}[c,\lambda]$ and some localization $R$ of $\mathbb{C}[c,\lambda] / I$, such that $\cW^{\mathrm{ev},I}_R(c,\lambda)$ has a singular vector in weight $ 2m \leq 2N+2$. But there are only finitely many divisors of $\mathrm{det}_{m}$ for $m\leq 2N+2$.\end{proof}

Later, we will see that for $N\geq 3$, the correspondence between these ideals and the isomorphism classes of such vertex algebras, is a bijection; see Corollary \ref{cor:uniqueness}.

\section{Principal $\cW$-algebras of type $B$ and $C$}
The $B$ and $C$ type principal $\cW$-algebras are isomorphic after the level shift \eqref{levelshift} by Feigin-Frenkel duality, so we shall only consider the type $C$ algebra $\cW^{\ell}(\gs\gp_{2n}, f_{\text{prin}})$. It has central charge
\begin{equation} \label{ccentralcharge} c =  -\frac{n (\ell + 2 n + 2 \ell n + 2 n^2) (-3 - 2 \ell + 4 \ell n + 4 n^2)}{\ell+n+1},\end{equation}
and is well known to be freely generated of type $\cW(2,4,\dots, 2n)$. Although often assumed in the physics literature, it is not obvious that $\cW^{\ell}(\gs\gp_{2n}, f_{\text{prin}})$ is generated by the weight $2$ and $4$ fields. This must be established in order to conclude that it arises as a quotient of $\cW^{\mathrm{ev}}(c,\lambda)$. 

The following theorem is analogous to the celebrated resulted of Frenkel, Kac, Radul, and Wang \cite{FKRW} that the $\text{GL}_n$-orbifold of the rank $n$ $bc$-system is isomorphic to $\cH \otimes \cW^{k}(\gs\gl_n, f_{\text{prin}})$ at level $k = 1 - h^{\vee} =   1 - n$. The proof is due to T. Creutzig, and we thank him for explaining it to us.

\begin{thm} \label{thomas} Let $\cA(n)$ denote the rank $n$ symplectic fermion algebra as above. Then $\cA(n)^{\text{Sp}_{2n}}$ is isomorphic to $\cW^{\ell}(\gs\gp_{2n}, f_{\text{prin}})$ at level $\ell = 1/2 - h^{\vee} = -1/2 -n$.
\end{thm}

\begin{proof} For all $\ell \in \mathbb{C}$, there is an injective vertex algebra homomorphism $$\gamma: \cW^{\ell}(\gs\gp_{2n}, f_{\text{prin}}) \ra \pi_0$$ known as the {\it Miura map}, where $\pi_0$ denotes the rank $n$ Heisenberg vertex algebra; see Theorem 5.17 and Remark 5.18 of \cite{AIII}. The image of $\gamma$ lies in the joint kernel of the screening operators $$\int V_{-\alpha_i / \nu}\ dz\ : \pi_0 \ra \pi_{-\alpha_i / \nu}, \qquad \nu = \sqrt{\ell + h^{\vee}}.$$ Here $\int V_{-\alpha_i / \nu}\ dz $ maps $\pi_0$ to the Heisenberg module $\pi_{-\alpha_i / \nu}$, and $\alpha_i$ runs over a set of simple roots for $\gs\gp_{2n}$. By Theorem 14.4.12 of \cite{FBZ}, for generic values of $\ell$, $$\cW^{\ell}(\gs\gp_{2n}, f_{\text{prin}})= \bigcap_{i} \text{Ker} \bigg(\int V_{-\alpha_i / \nu}\ dz\bigg).$$ In the case $\ell = -1/2 -n$, $\cW^{-1/2-n}(\gs\gp_{2n}, f_{\text{prin}})$ has central charge $-2n$, which is the same as the central charge of $\cA(n)^{\text{Sp}_{2n}}$, and $\int V_{-\alpha_i/ \nu}\ dz = \int V_{-\sqrt{2}\alpha_i}\ dz$.

A recent paper of Flandoli and Lentner \cite{FL} describes a related algebra, which (in the above notation) is the joint kernel of the screening operators $\int V_{- \alpha_i / \sqrt{2}} \ dz$ inside the lattice vertex algebra of the coroot lattice of type $B_n$ rescaled by $\sqrt{2}$, which is just the root lattice of $\gs\gp_{2n}$ rescaled by $\sqrt{2}$. By Corollary 7.8 of \cite{FL}, the joint kernel of $\int V_{- \alpha_i / \sqrt{2}} \ dz$ in this lattice vertex algebra is isomorphic to the even subalgebra $\cA(n)^{\text{even}} \subseteq \cA(n)$. It follows that $\cW^{-1/2-n}(\gs\gp_{2n}, f_{\text{prin}})$, which lies in $\bigcap_i \text{Ker} (\int V_{- \sqrt{2} \alpha_i} \ dz) \subseteq \pi_0$, is a subalgebra of $\cA(n)^{\text{even}}$.

Recall that $\cA(n)$, and hence $\cA(n)^{\text{even}}$, has automorphism group $\text{Sp}_{2n}$. Moreover, the operators $\int V_{-\sqrt{2}\alpha_i}\ dz$ where $\alpha_i$ ranges over the $n-1$ simple short roots, generate the upper nilpotent part of $\gs\gl_n \subseteq \gs\gp_{2n}$. The condition that an element $\omega \in \pi_0$ lies in the joint kernel of these $n-1$ screening operators, forces $\omega$ to lie in $(\cA(n)^{\text{even}})^{\text{GL}_n} =  \cA(n)^{\text{GL}_n}$. Therefore $\cW^{-1/2-n}(\gs\gp_{2n}, f_{\text{prin}})$ is a subalgebra of $\cA(n)^{\text{GL}_n}$.

By Theorem 4.3 of \cite{CLII}, $\cA(n)^{\text{GL}_n}$ is of type $\cW(2,3,\dots, 2n+1)$, and the explicit generators were written down. Using this description, it is straightforward to check that $\cA(n)^{\text{GL}_n}$ has a unique up to scalar primary field of weight $4$, which in fact lies in $\cA(n)^{\text{Sp}_{2n}}$. Therefore the image of the weight $4$ field of $\cW^{-1/2-n}(\gs\gp_{2n}, f_{\text{prin}})$ in $\cA(n)^{\text{GL}_n}$ must lie in $\cA(n)^{\text{Sp}_{2n}}$. Since $\cA(n)^{\text{Sp}_{2n}}$ is generated by the weights $2$ and $4$ fields, and the graded characters of $\cW^{-1/2-n}(\gs\gp_{2n}, f_{\text{prin}})$ and $\cA(n)^{\text{Sp}_{2n}}$ coincide, we conclude that $\cW^{-1/2-n}(\gs\gp_{2n}, f_{\text{prin}}) \cong \cA(n)^{\text{Sp}_{2n}}$. \end{proof}

\begin{cor} \label{cor:genwt4} As a one-parameter vertex algebra, $\cW^{\ell}(\gs\gp_{2n}, f_{\text{prin}})$ is generated by the weights $2$ and $4$ fields.
\end{cor}

\begin{proof} We regard $\cW^{\ell}(\gs\gp_{2n}, f_{\text{prin}})$ as a vertex algebra over a localization $R$ of the ring $\mathbb{C}[\ell]$. Let $\cV^{\ell} \subseteq\cW^{\ell}(\gs\gp_{2n}, f_{\text{prin}})$ denote the subalgebra generated by the weights $2$ and $4$ fields. If $\cV^{\ell}$ were a proper subalgebra of $\cW^{\ell}(\gs\gp_{2n}, f_{\text{prin}})$ as vertex algebras over $R$, then  $\cV^{-1/2-n}$ would be a proper subalgebra of $\cW^{-1/2 - n}(\gs\gp_{2n}, f_{\text{prin}})$, where both are obtained by replacing $\ell$ with $-1/2 - n$. But this contradicts the fact that $\cA(n)^{\text{Sp}_{2n}}$ is generated by the weights $2$ and $4$ fields. \end{proof}

\begin{remark} \label{arakawa} By the same argument as Proposition A.4 of \cite{ALY}, $\cW^{\ell}(\gs\gp_{2n}, f_{\text{prin}})$ is generated by the weights $2$ and $4$ fields for all noncritical values of $\ell$. We thank T. Arakawa for pointing this out to us.
\end{remark}

\begin{cor} \label{cor:typecrealization} For all $n\geq 2$, $\cW^{\ell}(\gs\gp_{2n}, f_{\text{prin}})$ can be obtained as a quotient of $\cW^{\mathrm{ev}}(c,\lambda)$ of the form $\cW^{\mathrm{ev},I_{n}}_{R_{n}}(c,\lambda) / \cI_{n}$ for some prime ideal $I_{n} \subseteq \mathbb{C}[c,\lambda]$ and some localization $R_{n}$ of $\mathbb{C}[c,\lambda]/I$. Here $\cI_{n}$ denotes the maximal proper graded ideal in $\cW^{\mathrm{ev},I_{n}}_{R_{n}}(c,\lambda)$.
\end{cor}

\begin{proof} Recall that by Lemma \ref{lem:ckn}, $\cA(n)^{\text{Sp}_{2n}}$ is generated by $L, W^4$, strongly generated by $\{L, W^{2i}|\ i\geq 2\}$, and all Jacobi identities of type $(W^{2i}, W^{2j}, W^{2k})$ hold as consequences of \eqref{deriv}-\eqref{ncw} alone, for $2i+2j+2k \leq 16$. Since $\cA(n)^{\text{Sp}_{2n}} \cong \lim_{\ell \ra -1/2 -n} \cW^{\ell}(\gs\gp_{2n}, f_{\text{prin}})$, these properties are inherited by $\cW^{\ell}(\gs\gp_{2n}, f_{\text{prin}})$ for generic $\ell$. The result then follows from Theorem \ref{thm:simplequotient}. \end{proof}

Note that the OPE algebra is determined by the coefficient of $W^4$ in the fourth order pole of $W^4$ with itself. After rescaling $W^4$ so that its eighth order pole with itself is $\frac{c}{4}$, this coefficient is denoted by $\sqrt{\cC}$ in \cite{H}, and is now only determined up to sign. The explicit formula for $\cC$ appears in Appendix A of \cite{H}. Using this equation and \eqref{ope:second} and \eqref{deflambda}, we can find the explicit generator $p_{n}$ of $I_{n}$ as follows. First, we must rescale the field $W^4$ in $\cW^{\mathrm{ev}}(c,\lambda)$ by the factor 
$$\mu = \frac{21 \sqrt{3}}{8 \sqrt{(c-1) (24 + c) (2c-1) (22 + 5 c) (1 -49\lambda^2 (c-25) (c-1) )}},$$
so that the eighth order pole of $\tilde{W}^4 = \mu W^4$ with itself is $\frac{c}{4}$. Using \eqref{ope:second} and \eqref{deflambda}, we obtain
\begin{equation} \label{form:structconst} \cC = -\frac{21168 \lambda^2 (c-1) (24 + c) (2c-1)}{(22 + 5 c) (-1 + 49 \lambda^2 (c-25)(c-1))}.\end{equation}
Equating these two expressions for $\cC$, after some simplifications we obtain the explicit formula for $p_{n}$, which appears in Appendix A.

\subsection{Zhu functor}
The Zhu functor is a basic tool in the representation theory of vertex algebras \cite{Z}. Given a vertex algebra $\cV$ over $\mathbb{C}$ with weight grading $\cV = \bigoplus_{n\geq 0} \cV[n]$, the Zhu algebra $\text{Zhu}(\cV)$ is a certain vector space quotient of $\cV$ with quotient map  $\pi_{\text{Zhu}}:\cV\ra \text{Zhu}(\cV)$, with the structure of a unital, associative algebra.

If $\cV$ is strongly generated by homogeneous elements $\{\alpha^1, \alpha^2,\dots\}$, $\text{Zhu}(\cV)$ is generated by $\{ a^i = \pi_{\text{Zhu}}(\alpha^i)\}$. An $\mathbb{N}$-graded $\cV$-module $M = \bigoplus_{n\geq 0} M[n]$ is called a {\it positive energy module} if for every $a\in\cV[m]$, $a(n) M_k \subseteq M[m+k -n-1]$, for all $n$ and $k$. For a field $a\in\cV[m]$, $a(m-1)$ acts on each $M[k]$. The subspace $M[0]$ is a $\text{Zhu}(\cV)$-module with action $\pi_{\text{Zhu}}(a)\mapsto a(m-1) \in \text{End}(M[0])$. In fact, $M\mapsto M[0]$ is a bijection between irreducible, positive energy $\cV$-modules and irreducible $\text{Zhu}(\cV)$-modules. If $\text{Zhu}(\cV)$ is commutative, all its irreducible modules are one-dimensional. The corresponding irreducible $\cV$-modules $M = \bigoplus_{n\geq 0} M[n]$ are then cyclic, and will be called {\it highest-weight modules}.

It is known that the Zhu algebra of $\cW^{k}(\gs\gp_{2n}, f_{\text{prin}})$ is independent of $k$ and is isomorphic to the center of $U(\gs\gp_{2n})$, which is the polynomial algebra on the fundamental invariants of $\gs\gp_{2n}$; see \cite{AI}, Theorem 4.16.3.

The Zhu algebra of vertex algebras over a ring $R$ can be defined in a similar way.

\begin{thm} \label{thm:zhu} The Zhu algebra $\text{Zhu}(\cW^{\mathrm{ev}}(c,\lambda))$ is isomorphic to the polynomial algebra
\begin{equation} \label{zhu} \mathbb{C}[c,\lambda] \otimes_{\mathbb{C}} \mathbb{C}[\ell, w^{2i} |\ i\geq 2],\end{equation} where the generators are the images of $\{L, W^{2i}|\ i\geq 2\}$ under the Zhu map. For any ideal $I\subseteq \mathbb{C}[c,\lambda]$, any localization $R$ of $\mathbb{C}[c,\lambda] / I$, and any quotient $\cW^{\mathrm{ev},I}_R(c,\lambda) /\cI$, the Zhu algebra $\text{Zhu}(\cW^{\mathrm{ev},I}_R(c,\lambda) /\cI)$ is a quotient of a localization of \eqref{zhu}, and hence is abelian.\end{thm}

\begin{proof} Since $\cW^{\mathrm{ev}}(c,\lambda)$ is strongly generated by $\{L, W^{2i}|\ i\geq 2\}$, $\text{Zhu}(\cW^{\mathrm{ev}}(c,\lambda))$ is generated by $\{\ell, w^{2i}|\ i\geq 2\}$. Also, $\ell$ is central. The commutator $[w^{2i}, w^{2j}]$ in the Zhu algebra is expressed in terms of the OPE algebra and hence is a polynomial in $\{\ell, w^{2i}|\ i\geq 2\}$ with structure constants in $\mathbb{C}[c,\lambda]$. Since the Zhu algebra of $\cW^k(\gs\gp_{2n}, f_{\text{prin}})$ is abelian, each structure constant is divisible by the generator $p_{n}$ of the ideal $I_{n}$ in Corollary \ref{cor:typecrealization}. Since $\cW^k(\gs\gp_{2n}, f_{\text{prin}})$ is generated by the weight $4$ field for all $n\geq 2$, the polynomials $p_{n}$ must all be distinct, so all of the above structure constants must vanish. The remaining statements follow from the fact that for any vertex algebra $\cV = \bigoplus_{n\geq 0} \cV[n]$ as above, and any graded ideal $\cI \subseteq \cV$, we have
$\text{Zhu}(\cV/\cI)\cong \text{Zhu}(\cV)/ I$ where $I = \pi_{\text{Zhu}}(\cI)$.
\end{proof}

\begin{cor} For any vertex algebra $\cW$ over $\mathbb{C}$ which arises as a quotient of $\cW^{\mathrm{ev}}(c,\lambda)$, all irreducible, positive energy modules are highest-weight modules, and are parametrized by the variety $\text{Specm}(\text{Zhu}(\cW))$. \end{cor}

\subsection{Poisson structure}
For any vertex algebra $\cV$, we have Li's canonical decreasing filtration
$$F^0(\cV) \supseteq F^1(\cV) \supseteq \cdots.$$ Here $F^p(\cV)$ is spanned by elements of the form
$$:(\partial^{n_1} \alpha^1) (\partial^{n_2} \alpha^2) \cdots (\partial^{n_r} \alpha^r):,$$ 
where $\alpha^1,\dots, \alpha^r \in \cV$, $n_i \geq 0$, and $n_1 + \cdots + n_r \geq p$ \cite{Li-poisson}. Clearly $\cV = F^0(\cV)$ and $\partial F^i(\cV) \subseteq F^{i+1}(\cV)$. Set $$\text{gr}^F(\cV) = \bigoplus_{p\geq 0} F^p(\cV) / F^{p+1}(\cV),$$ and for $p\geq 0$ let 
$$\sigma_p: F^p(\cV) \ra F^p(\cV) / F^{p+1}(\cV) \subseteq \text{gr}^F(\cV)$$ be the projection. Then $\text{gr}^F(\cV)$ is a graded commutative algebra with product
$$\sigma_p(\alpha) \sigma_q(\beta) = \sigma_{p+q}(\alpha_{(-1)} \beta),$$ for $\alpha \in F^p(\cV)$ and $\beta \in F^q(\cV)$. There is a differential $\partial$ on $\text{gr}^F(\cV)$ defined by $$\partial( \sigma_p(\alpha) ) = \sigma_{p+1} (\partial \alpha),\qquad \alpha \in F^p(\cV).$$ Finally, $\text{gr}^F(\cV)$ has a Poisson vertex algebra structure \cite{Li-poisson}; for $n\geq 0$, $\alpha \in F^p(\cV)$, and $\beta\in F^q(\cV)$, define
$$\sigma_p(\alpha)_{(n)} \sigma_q(\beta) = \sigma_{p+q-n} \alpha_{(n)} \beta.$$
 
The subalgebra $F^0(\cV) / F^1(\cV)$ is isomorphic to Zhu's commutative algebra $C(\cV)$ \cite{Z}, and is known to generate $\text{gr}^F(\cV)$ as a differential graded algebra \cite{Li-poisson}. We change our notation slightly and denote by $\bar{\alpha}$ the image of $\alpha$ in $C(\cV)$. It is a Poisson algebra with product $\bar{\alpha} \cdot \bar{\beta} = \overline{\alpha_{(-1)} \beta}$ and Poisson bracket $\{ \bar{\alpha}, \bar{\beta}\} = \overline{\alpha_{(0)} \beta}$. Also, if the Poisson bracket on $C(\cV)$ is trivial, the Poisson vertex algebra structure on $\text{gr}^F(\cV)$ is also trivial in the sense that for all $a \in F^p(\cV)$, $b\in F^q(\cV)$ and $n\geq 0$, $\sigma_p(a)_{(n)} \sigma_q(b) = 0$; see Remark 4.32 of \cite{AMII}.

\begin{thm} 
For any vertex algebra $\cW =  \cW^{\mathrm{ev},I}_R(c,\lambda) / \cI$ for some $I$ and $R$ as above, the Poisson structure on $C(\cW)$ and the Poisson vertex algebra structure on $\text{gr}^F(\cW)$ are trivial.
 \end{thm}

\begin{proof} Since $C(\cW)$ is generated by $\{\bar{L}, \bar{W}^{2i}|\ i \geq 2\}$ and $\{\bar{L},-\}$ acts trivially on $C(\cW)$, it suffices to show that$\{\bar{W}^{2j}, \bar{W}^{2k}\} = 0$ for all $j,k \geq 2$. But $W^{2j}_{(0)} W^{2k}$ is a sum of normally ordered monomials in $\{L,W^{2i}|\ 2\leq i \leq j+k -1\}$ of odd weight $2j+2k-1$. Therefore each monomial must lie in $F^1(\cW)$, so that $\overline{W^{2j}_{(0)} W^{2k}} = 0$. \end{proof}

\section{Principal $\cW$-algebras of type $D$}
Let $\cW^{\ell}(\gs\go_{2n}, f_{\text{prin}})$ denote the principal $\cW$-algebra at level $\ell$ associated to $\gs\go_{2n}$ for $n\geq 3$. It has central charge
\begin{equation} \label{dcentralcharge} c = -\frac{n (5 - 10 n + 4 n^2 - 2 \ell + 2 n \ell) (4 - 8 n + 4 n^2 - \ell + 2 n \ell)}{\ell + 2 n -2 },\end{equation} and has a strong generator in each weight $2,4,\dots, 2n-2$ and $n$. 

Consider the coset
\begin{equation} \label{eq:dkn} \cD^k(n) =  \text{Com}(V^{k+1}(\gs\go_{2n}), V^{k}(\gs\go_{2n}) \otimes L_1(\gs\go_{2n})).\end{equation}
By Theorem 8.7 of \cite{ACL}, we have
\begin{equation} \cW^{\ell}(\gs\go_{2n}, f_{\text{prin}}) = \cD^k(n), \qquad \ell = -(2n-2) + \frac{k+ 2n-2}{k+2n-1}.\end{equation}

\begin{cor}\label{cor:typeD} $\cW^{\ell}(\gs\go_{2n}, f_{\text{prin}})$ has an action of $\mathbb{Z}_2$, and the orbifold $\cW^{\ell}(\gs\go_{2n}, f_{\text{prin}})^{\mathbb{Z}_2}$ is simple, of type $\cW(2,4,\dots, 4n)$, and generated by the weight $4$ field.\end{cor}

\begin{proof} It is apparent that the action of $\text{SO}_{2n}$ on $V^{k}(\gs\go_{2n}) \otimes L_1(\gs\go_{2n})$ which is infinitesimally generated by the zero modes of $V^{k+1}(\gs\go_{2n})$, extends to an $\text{O}_{2n}$-action on $V^{k}(\gs\go_{2n}) \otimes L_1(\gs\go_{2n})$. This yields the $\mathbb{Z}_2$ action on $\cD^k(n)$. We have
$$\text{lim}_{k \ra \infty} \cD^k(n) \cong L_1(\gs\go_{2n})^{\text{O}_{2n}} \cong F(2n)^{\text{O}_{2n}},$$ where $F(2n)$ denotes the rank $2n$ free fermion algebra. By Theorems 6 and 9 of \cite{LII}, $F(2n)^{\text{O}_{2n}}$ is of type $\cW(2,4,\dots, 4n)$, and it is easily seen to be generated by the weight $4$ field. It is simple by \cite{KR,DLM}. Finally, these properties are inherited by $\cD^k(n)$ for generic values of $k$.\end{proof}

\begin{cor} \label{cor:typedrealization} For all $n\geq 3$, $\cW^{\ell}(\gs\go_{2n}, f_{\text{prin}})^{\mathbb{Z}_2}$ can be obtained as a quotient 
$\cW^{\mathrm{ev},J_{n}}_{R_{n}}(c,\lambda)/ \cJ_{n}$ for some ideal $J_{n} \subseteq \mathbb{C}[c,\lambda]$ and some localization $R_{n}$ of $\mathbb{C}[c,\lambda] / J_{n}$. Here $\cJ_{n}$ denotes the maximal proper graded ideal of $\cW^{\mathrm{ev},J_{n}}_{R_{n}}(c,\lambda)$. \end{cor}

\begin{proof} This is similar to the proof of Corollary \ref{cor:typecrealization}.
\end{proof}

As in the case of type $C$ principal $\cW$-algebras, the explicit generator of $J_{n}$ can be deduced from calculations of Hornfeck \cite{H}.
\begin{thm} \label{thm:prinw} 
For $n\geq 3$, let
\begin{equation} \label{ideal:wslna} \begin{split} q_{n} & =  7 \lambda (c-1) (22 + 5 c) (-3 c + 10 n - 2 c n - 19 n^2 + 2 c n^2 + 12 n^3) \\ & -(-6 c - 9 c^2 - 28 n + 120 c n - 2 c^2 n + 10 n^2 - 147 c n^2 + 2 c^2 n^2 + 24 n^3 + 36 c n^3). \end{split} \end{equation}
Then
\begin{enumerate}
\item $q_{n}$ is an irreducible factor of $\mathrm{det}_{4n+2}$ of level $4n+2$, and generates the ideal $J_{n} \subseteq \mathbb{C}[c,\lambda]$.

\item There exists a localization $R_{n}$ of $\mathbb{C}[c,\lambda] / J_{n}$ in which $(c-1)$,  $(22 + 5 c)$, and  $(-3 c + 10 n - 2 c n - 19 n^2 + 2 c n^2 + 12 n^3)$ are invertible, so that  
\begin{equation} \label{ideal:solvelambda} \lambda =  \frac{-6 c - 9 c^2 - 28 n + 120 c n - 2 c^2 n + 10 n^2 - 147 c n^2 + 2 c^2 n^2 + 24 n^3 + 36 c n^3}{7 (c-1) (22 + 5 c) (-3 c + 10 n - 2 c n - 19 n^2 + 2 c n^2 + 12 n^3)}\end{equation}
 in $R_{n}$, and there is a unique singular vector of weight $4n+2$
 \begin{equation} \label{sing:gen} W^{4n+2} - P_{4n+2}(L, W^4,\dots, W^{4n}).\end{equation} 
 \item We have an isomorphism $$\cW^{\mathrm{ev},J_{n}}_{R_{n}}(c,\lambda)/ \cJ_{n} \cong \cW^{\ell}(\gs\go_{2n}, f_{\text{prin}})^{\mathbb{Z}_2},$$ where $\cJ_{n}$ denotes the maximal proper graded ideal, and $c$ and $\lambda$ are related to $\ell$ as above.
\end{enumerate}
\end{thm}

\begin{proof}
If the generator $W^4$ of $\cW^{\ell}(\gs\go_{2n}, f_{\text{prin}})^{\mathbb{Z}_2}$ is normalized so its eighth order pole with itself is $\frac{c}{4}$, the coefficient of $W^4$ appearing in the fourth order pole of $W^4$ with itself, denoted by $\sqrt{\cC}$ in \cite{H}, is determined by
\begin{equation} \label{form:structconstD} \begin{split}& \cC = \\ & \frac{36 (-6 c - 9 c^2 - 28 n + 120 c n - 2 c^2 n + 10 n^2 - 147 c n^2 + 2 c^2 n^2 + 24 n^3 + 36 c n^3)^2}{(22 + 5 c) (2 + c - 7 n + c n + 3 n^2) (c - 10 n + 2 c n + 16 n^2) (6 c - 5 n - 7 c n + 4 n^2 + 2 c n^2)};\end{split} \end{equation} see Equation 4.4 of \cite{H}. As before, $\cC$ is also given by \eqref{form:structconst}. Equating \eqref{form:structconst} and \eqref{form:structconstD} and solving for $\lambda$ yields \eqref{ideal:solvelambda} up to sign. Finally, by Remark \ref{rem:d28twochoicesauto}, only one choice of sign, namely, the one given by \eqref{ideal:solvelambda}, is consistent with the OPE relations in Section \ref{section:main}. \end{proof}

\section{Some other one-parameter quotients of $\cW^{\mathrm{ev}}(c,\lambda)$}

In the introduction, we mentioned a few other one-parameter vertex algebras of type $\cW(2,4,\dots,2N)$ arising as cosets of affine vertex algebras inside other structures. Here we show that they can all be obtained as quotients of $\cW^{\mathrm{ev}}(c,\lambda)$ as above.

\begin{thm} \label{thm:otherquotients} Consider the following coset vertex algebras
\begin{equation} \begin{split} 
& \cE^k(n) = \text{Com}(V^{k}(\gs\gp_{2n}), V^{k+1/2}(\gs\gp_{2n}) \otimes L_{-1/2}(\gs\gp_{2n})),\qquad n\geq 1,
\\ & \cF^k(n) = \text{Com}(V^{k+1/2}(\gs\gp_{2n-2}), \cW^k(\gs\gp_{2n}, f_{\text{min}})),\qquad n\geq 2,
\end{split} \end{equation}
These vertex algebras can all be obtained in the form $\cW^{\mathrm{ev},I}_R(c,\lambda)$, for some prime ideal $I \subseteq \mathbb{C}[c,\lambda]$, and some localization $R$ of $\mathbb{C}[c,\lambda]/I$.
\end{thm}

\begin{proof} Recall that $\cE^k(n)$ is of type $\cW(2,4\dots, 2n^2+4n)$ for generic values of $k$, and that
$$\lim_{k\ra \infty} \cE^k(n) \cong L_{-1/2}(\gs\gp_{2n})^{\text{Sp}_{2n}} \cong \cS(n)^{\text{Sp}_{2n}},$$ where $\cS(n)$ denotes the rank $n$ $\beta\gamma$-system. By Lemma 3 and Theorem 1 of \cite{LII}, $\cS(n)^{\text{Sp}_{2n}}$ is generated by the weight $4$ field, and by Weyl's fundamental theorems of invariant theory for the standard representation of $\text{Sp}_{2n}$ \cite{We}, there are no normally ordered polynomial relations among the generators $\{L, W^{2i}|\ i\geq 2\}$ of $\cS(n)^{\text{Sp}_{2n}}$ of weight less than $2n^2+4n+2$. These properties are inherited by $\cE^k(n)$ for generic values of $k$, so for $n\geq 2$ the claim holds by Theorem \ref{thm:simplequotient}. In the case $n = 1$, $\cS(1)^{\text{Sp}_{2}}$ is of type $\cW(2,4,6)$ \cite{BFH}, and it can be verified by computer that it is a quotient of $\cW^{\mathrm{ev}}(c,\lambda)$ as above.

Next, by Theorem 5.2 of \cite{ACKL},
$$\lim_{k\ra \infty} \cF^k(n) \cong \cT \otimes \cG_{\mathrm{ev}}(n-1)^{\text{Sp}_{2n-2}},$$ where $\cT$ is a generalized free field algebra that is strongly generated by a field in weight $2$, and $\cG_{\mathrm{ev}}(n-1)$ is another generalized free field algebra such that $\cG_{\mathrm{ev}}(n-1)^{\text{Sp}_{2n-2}}$ is of type $\cW(4,6,\dots, 2n^2+2n-2)$. Note that in \cite{ACKL}, $\cG_{\mathrm{ev}}(n-2)$ is a typo and should be replaced by $\cG_{\mathrm{ev}}(n-1)$. It is not difficult to check that  $\cG_{\mathrm{ev}}(n-1)^{\text{Sp}_{2n-2}}$ is generated by the weight $4$ field. It follows that $\cF^k(n)$ is generated by the weights $2$ and $4$ fields, and is of type $\cW(2,4,\dots, 2n^2+2n-2)$. There are no normally ordered relations of weight less than $2n^2+2n$ among the generators $\{L, W^{2i}|\ i\geq 2\}$ of $\cF^k(n)$, so for $n\geq 3$, the claim follows from Theorem \ref{thm:simplequotient}. It is easy to verify by computer that it also holds for $n = 2$. 
\end{proof}

\begin{cor} The vertex algebras $\cE^k(n)$ and $\cF^k(n)$ all have abelian Zhu algebras, so their irreducible, positive energy modules are all highest-weight modules.
\end{cor}

\subsection{Explicit truncation curve for $\cF^k(n)$}  

In \cite{ACKL}, it was shown that for levels $k\in \frac{1}{2} + \mathbb{N}$, there is an embedding
$$L_{k+1/2}(\gs\gp_{2n-2}) \hookrightarrow \cW_k(\gs\gp_{2n}, f_{\text{min}}),$$ and that the simple quotient $\cF_k(n)$ of $\cF^k(n)$ coincides with
$$\text{Com}(L_{k+1/2}(\gs\gp_{2n-2}), \cW_k(\gs\gp_{2n}, f_{\text{min}})).$$ Based on a remarkable coincidence of central charges, the following conjecture was made in \cite{ACKL}.
\begin{conj} \label{conj:ackl} For all $n\geq 2$ and $k \in \frac{1}{2} + \mathbb{N}$, 
$$\cF_k(n) \cong \cW_s(\gs\gp_{2m},f_{prin}),\qquad m = k + \frac{1}{2},\qquad s = -(m+1) + \frac{n+k+1/2}{2n+2k+2}.$$ These are nondegenerate admissible levels for $\gs\gp_{2m}$, so $\cW_s(\gs\gp_{2m},f_{prin})$ is lisse and rational. This conjecture implies that $\cW_k(\gs\gp_{2n}, f_{\text{min}})$ is also lisse and rational for $k \in \frac{1}{2} + \mathbb{N}$.
\end{conj}

Conjecture \ref{conj:ackl} turns out to be equivalent to an explicit formula for the truncation curve in $\mathbb{C}^2$ that realizes $\cF^k(n)$ as a quotient of $\cW^{\mathrm{ev}}(c,\lambda)$.

\begin{conj} \label{conj:reform} The truncation curve for $\cF^k(n)$ has the following rational parametrization.
\begin{equation} \label{conj:reformeq} \begin{split}
c_n(k) = & -\frac{(1 + 2 k) (2 + 2 k + n) (2 + 3 k + 2 n)}{(1 + k + n) (1 + 2 k + 2 n)},
\\ \lambda_n(k) = & -\frac{(1 + k + n) (1 + 2 k + 2 n) p_n(k)}{7 (1 + k) (1 + 2 k + n) (5 + 6 k + 4 n) q_n(k) r_n(k) },
\\  p_n(k) & =  -60 - 306 k - 408 k^2 + 198 k^3 + 720 k^4 + 360 k^5 - 294 n - 812 k n + 177 k^2 n 
\\ & + 1916 k^3 n + 1236 k^4 n - 360 n^2 - 153 k n^2 + 1606 k^2 n^2 + 1504 k^3 n^2 - 102 n^3 + 464 k n^3 
\\ & + 776 k^2 n^3 + 24 n^4 + 144 k n^4,
\\  q_n(k) & = 6 + 9 k + 6 k^2 + 15 n + 20 k n + 12 n^2,
\\  r_n(k) & = -2 + 24 k + 86 k^2 + 60 k^3 - 36 n + 7 k n + 70 k^2 n - 34 n^2 + 20 k n^2.
\end{split}
\end{equation}
In particular, if we let $K_n \subseteq \mathbb{C}[c,\lambda]$ be the ideal corresponding to this curve, then $$\cW^{\mathrm{ev},K_n}(c,\lambda) = \cW^{\mathrm{ev}}(c,\lambda)/ K_n \cdot \cW^{\mathrm{ev}}(c,\lambda),$$ which is a vertex algebra over the ring $\mathbb{C}[c,\lambda] / K_n$, has a singular vector of weight $2n^2+2n$, and the simple quotient is isomorphic to $\cF^k(n)$ after a suitable localization.
\end{conj} 

The proof that Conjectures \ref{conj:ackl} and \ref{conj:reform} are equivalent is analogous to the proof of Theorem 8.3 of \cite{LI}. The isomorphisms in Conjectures \ref{conj:ackl}, together with the explicit generator for the ideal $I_{n}$ appearing in Appendix A, are enough to determine $K_n$ uniquely. 

\begin{thm} \label{thm:minimalsp4} Conjecture \ref{conj:reform} holds in the first nontrivial case $n=2$. Therefore Conjecture \ref{conj:ackl} holds for $n=2$ and all $k \in \frac{1}{2} + \mathbb{N}$, and $\cW_k(\gs\gp_{4}, f_{\text{min}})$ is lisse and rational for all $k \in \frac{1}{2} + \mathbb{N}$.
\end{thm}

\begin{proof} This can be verified by computer by writing down the explicit generators of $\cF^k(2)$ and computing the fourth order pole of the field $W^4$ with itself.
\end{proof}

\section{Quotients of $\cW^{\mathrm{ev}}(c,\lambda)$ by maximal ideals and coincidences between algebras of type $\cW(2,4,\dots, 2N)$} \label{section:coincidences}

So far, we have considered quotients of the form $\cW^{\mathrm{ev},I}_R(c,\lambda)$ which are {\it one-parameter families} of vertex algebras in the sense that $R$ has Krull dimension $1$. Here, we consider simple quotients of $\cW^{\mathrm{ev},I}(c,\lambda)$ where $I\subseteq \mathbb{C}[c,\lambda]$ is a {\it maximal} ideal. Such an ideal always has the form $I = (c- c_0, \lambda- \lambda_0)$ for $c_0, \lambda_0\in \mathbb{C}$, and $\cW^{\mathrm{ev},I}(c,\lambda)$ is an ordinary vertex algebra over $\mathbb{C}$. We first need a criterion for when the simple quotients of two such vertex algebras are isomorphic. 

\begin{thm} \label{thm:coincidences} Let $c_0,c_1, \lambda_0, \lambda_1$ be complex numbers and let $$I_0 = (c- c_0, \lambda- \lambda_0),\qquad I_1 = (c - c_1, \lambda - \lambda_1)$$ be the corresponding maximal ideals in $\mathbb{C}[c,\lambda]$. Let $\cW_0$ and $\cW_1$ be the simple quotients of $\cW^{\mathrm{ev},I_0}(c,\lambda)$ and $\cW^{\mathrm{ev},I_1}(c,\lambda)$. Then $\cW_0 \cong \cW_1$ are isomorphic only in the following cases.
\begin{enumerate}
\item $c_0 = c_1$ and $\lambda_0 =  \lambda_1$,
\item $c_0 = 0 = c_1$ and no restriction on $\lambda_0, \lambda_1$,
\item $c_0 = 1, -24, -\frac{22}{5}, \frac{1}{2} = c_1$, and no restriction on $\lambda_0, \lambda_1$,
\item $c_0 = c = c_1$, where $c$ is arbitrary except for $c \neq 1,  25$, and $$\lambda_0 = \pm \frac{1}{7 \sqrt{(c-25) (c-1)}} = \pm \lambda_1,$$
\item $c_0 = c = c_1$, where $c$ is arbitrary except for $c \neq 1,  25$, and 
$$\lambda_0 = \pm \frac{\sqrt{196 - 172 c + c^2}}{21 (c -1) (22 + 5 c)} = \pm \lambda_1.$$

\end{enumerate} 
\end{thm}

\begin{proof} Clearly the isomorphism $\cW_0 \cong \cW_1$ holds in case (1). It holds in case (2) as well because the simple quotient is $\mathbb{C}$ for all $\lambda$. Next, we compute
$$\mathrm{det}_4 = -\frac{32}{63} (c -1) c^3 (24 + c) (2 c -1) (22 + 5 c)^2 (-1 + 49 \lambda^2 (25 - 26 c + c^2)).$$
In cases (3) and (4), $W^4$ is a singular vector in $\cW^{\mathrm{ev},I_0}(c,\lambda)$ and $\cW^{\mathrm{ev},I_1}(c,\lambda)$, so the simple quotients $\cW_0$ and $\cW_1$ are just the Virasoro vertex algebra for all values of $\lambda$. Therefore the isomorphism holds in cases (3) and (4).

Next, let $I$ be the ideal generated by 
$$p = -196 + 172 c - c^2 + 213444 \lambda^2 - 329868 c \lambda^2 + 30429 c^2 \lambda^2 + 74970 c^3 \lambda^2 + 11025 c^4 \lambda^2.$$ It is not difficult to check that $p$ is an irreducible factor of $\text{det}_6$. Solving for $\lambda$ yields $$\lambda = \pm \frac{\sqrt{196 - 172 c + c^2}}{21 (c-1) (22 + 5 c)},$$ and in $\cW^{\mathrm{ev},I}(c,\lambda)$ there is a unique up to scalar singular vector in weight $6$ of the form
\begin{equation} \begin{split}  W^6  & - \frac{4096 (24 + c) (2 c-1)^2 (13 + 72 c)}{3969 (22 + 5 c)^2} :LLL:   \mp \frac{128 (2 c -1) \sqrt{196 - 172 c + c^2}}{9 (22 + 5 c)} :L W^4:
\\ & - \frac{512 (24 + c) (2 c -1)^2 (-2528 + 117 c + 176 c^2)}{11907 (22 + 5 c)^2} :(\partial^2 L) L: 
\\ & - \frac{256 (24 + c) (2 c -1)^2 (2048 + 2592 c + 295 c^2)}{11907 (22 + 5 c)^2} :(\partial L) \partial L: 
 \\ & \mp \frac{8 (2 c -1) (64 + 5 c) \sqrt{196 - 172 c + c^2}}{189 (22 + 5 c)} \partial^2 W^4 
 \\ & - \frac{256 (24 + c) (2 c -1)^2 (-116 - 764 c + 5 c^2 + 5 c^3)}{35721 (22 + 5 c)^2} \partial^4 L.
\end{split} \end{equation}
In particular, $W^6$ decouples in the simple quotient of $\cW^{\mathrm{ev},I}(c,\lambda)$, and this quotient is the unique one-parameter vertex algebra of type $\cW(2,4)$, with parameter $c$. It follows that $\cW_0 \cong \cW_1$ in this case as well.

Conversely, suppose there is another case where $\cW_0 \cong \cW_1$. Necessarily $c_0 = c = c_1$ and $\lambda_0 \neq \lambda_1$. Since we are not in the above cases, neither $W^4$ nor $W^6$ decouples in the simple vertex algebras $\cW_0$ and $\cW_1$. By \eqref{ope:second} and \eqref{deflambda} the coefficient of $W^4$ in $W^4_{(3)} W^4$ in $\cW_0$ is a nontrivial multiple of $\lambda_0$, and the coefficient of $W^4$ in $W^4_{(3)} W^4$ in $\cW_1$ is a nontrivial multiple of $\lambda_1$. So we must have $\lambda_0 = \pm \lambda_1$. Finally, by Remark \ref{rem:d28twochoicesauto}, since $W^6$ does not decouple, our choice of $d_{28}$ from \eqref{deflambda} forces $\lambda_0 =  \lambda_1$, so we are in case (1). \end{proof}

\begin{cor} \label{cor:intersections} Let $I = (p)$, $J = (q)$ be prime ideals in $\mathbb{C}[c,\lambda]$ which lie in the Shapovalov spectrum of $\cW^{\mathrm{ev}}(c,\lambda)$. Then aside from the coincidences given by Theorem \ref{thm:coincidences}, any additional pointwise coincidences between the simple quotients of $\cW^{\mathrm{ev},I}(c,\lambda)$ and $\cW^{\mathrm{ev},J}(c,\lambda)$ must correspond to intersection points of the truncation curves $V(I) \cap V(J)$. \end{cor}

\begin{cor} \label{cor:uniqueness} Suppose that $\cA$ is a simple, one-parameter vertex algebra which is isomorphic to the simple quotient of $\cW^{\mathrm{ev},I}(c,\lambda)$ for some prime ideal $I = (p)\subseteq \mathbb{C}[c,\lambda]$, possibly after localization. Then if $\cA$ is the quotient of $\cW^{\mathrm{ev},J}(c,\lambda)$ for some prime ideal $J$, possibly localized, we must have $I = J$. \end{cor}

\begin{proof} This is immediate from Theorem \ref{thm:coincidences} and Corollary \ref{cor:intersections}, since if $I$ and $J$ are distinct prime ideals in $\mathbb{C}[c,\lambda]$, their truncation curves $V(I)$ and $V(J)$ can intersect in at most finitely many points. Then the simple quotients of $\cW^{\mathrm{ev},I}(c,\lambda)$ and $\cW^{\mathrm{ev},J}(c,\lambda)$ cannot coincide as one-parameter families. \end{proof}

\section{Coincidences} As an application of Theorem \ref{thm:coincidences}, we classify all nontrivial coincidences among the simple vertex algebras $\cW_{\ell}(\gs\go_{2n},f_{\text{prin}})^{\mathbb{Z}_2}$, $\cW_{\ell'}(\gs\go_{2m},f_{\text{prin}})^{\mathbb{Z}_2}$, $\cW_{k}(\gs\gp_{2r},f_{\text{prin}})$ and $\cW_{k'}(\gs\gp_{2s},f_{\text{prin}})$. Let $I$ be an ideal in the Shapovalov spectrum and let $\cW^{\mathrm{ev},I}(c,\lambda)$ be the corresponding one-parameter vertex algebra. Let $\cA^k$ be a simple one-parameter vertex algebra which is isomorphic to the simple quotient $\cW^{\mathrm{ev},I}(c,\lambda)/ \cI$, via 
\begin{equation} \label{isoviapara} L \mapsto \tilde{L}, \qquad W^4 \mapsto \tilde{W}^4, \qquad c\mapsto c(k), \qquad \lambda\mapsto \lambda(k). \end{equation} Here $\{L,W\}$ and $\{\tilde{L}, \tilde{W}^4\}$ are the standard generators for $\cW^{\mathrm{ev},I}(c,\lambda)$ and $\cA^k$, respectively, and $k \mapsto (c(k), \lambda(k))$ is a rational parametrization of the curve $V(I)$. Then all structure constants in the OPEs $\tilde{W}^{2i}(z) \tilde{W}^{2j}(w)$ are polynomials in $c(k)$ and $\lambda(k)$, where $\tilde{W}^{2i} = \tilde{W}^4_{(1)}\tilde{W}^{2i-2}$ for $i \geq 3$. It follows that $\cA^k$ is well-defined and the isomorphism \eqref{isoviapara} holds for all values of $k$ except for the poles of $c(k)$ and $\lambda(k)$. This should be contrasted with previous statements that required us to work over a localization $R$ of $\mathbb{C}[c,\lambda] / I$. The reason we localized was to obtain a singular vector of the form $W^{2N+2} - P(L, W^4, \dots, W^{2N})$ in $\cW^{\mathrm{ev},I}_R(c,\lambda)$, so that the simple quotient $\cW^{\mathrm{ev},I}_R(c,\lambda)/ \cI$ truncates to an algebra of type $\cW(2,4,\dots 2N)$. Here we prefer {\it not} to localize; in this case, $\cW^{\mathrm{ev},I}(c,\lambda)/ \cI$ need not truncate to $\cW(2,4,\dots, 2N)$ for any $N$, but we do not need to exclude any values of $k$ except for the poles of $c(k)$ and $\lambda(k)$.

For generic $k$, the examples $\cA^k$ in this paper are either $\cW^k(\gs\gp_{2n}, f_{\text{prin}})$, $\cW^k(\gs\go_{2n}, f_{\text{prin}})^{\mathbb{Z}_2}$, or a coset of the form $\text{Com}(V^k(\gg), \tilde{\cA}^k)$, for some vertex algebra $\tilde{\cA}^k$. In order to apply Corollary \ref{cor:intersections} to find coincidences between one-parameter quotients of $\cW^{\mathrm{ev}}(c,\lambda)$, there are a few subtleties that we need to mention.

\begin{enumerate}

\item Even though $\cA^k$ is well-defined for all $k$ away from the poles of $c(k)$ and $\lambda(k)$, the specialization of $\cA^{k}$ at a value $k = k_0$, can fail to coincide with the algebra of interest. This subtlety does not occur in the case of the $\cW$-algebras or their orbifolds, but it does occur for the examples arising as cosets. The structure of cosets of the form $\cA^k = \text{Com}(V^k(\gg), \tilde{\cA}^k)$ for simple $\gg$ was studied in \cite{CLi} in a general setting. The specialization of $\cA^k$ at $k = k_0$ can be a proper subalgebra of $\text{Com}(V^{k_0}(\gg), \tilde{\cA}^{k_0})$, but by Corollary 6.7 of \cite{CLi}, under mild hypotheses that are satisfied in our examples, this can only occur for rational numbers $k_0 \leq -h^{\vee}$, where $h^{\vee}$ is the dual Coxeter number of $\gg$. By a slight abuse of notation, if $\cA^k$ is a one-parameter vertex algebra that generically coincides with $\text{Com}(V^k(\gg), \tilde{\cA}^k)$, by $\cA^{k_0}$ we mean the specialization of $\cA^k$ at the value $k = k_0$, even if it is smaller than the actual coset. If $\cA^k$ is a one-parameter quotient of $\cW^{\mathrm{ev},I}(c,\lambda)$, $\cA^{k_0}$ will then be a quotient of $\cW^{\mathrm{ev},I}(c,\lambda)$, even if the actual coset at this point is not such a quotient.

\smallskip

\item Even if $k_0$ is a pole of $c(k)$ or $\lambda(k)$, the algebra $\cA^k$ might still be well-defined at $k = k_0$. For example, $\cA^k = \cW^k(\gs\gp_{2n}, f_{\text{prin}})$ is defined for all $k_0 \in \mathbb{C}$. The corresponding truncation curve $V(I_{n})$ has the rational parametrization given by \eqref{app:typec}. Then $c(k)$ has a pole at the critical level $k_0 = -(n+1)$, and $\lambda(k)$ has poles at $\displaystyle k_0 = -\frac{2 n^2-1}{2 n-1}$, $\displaystyle k_0 = -\frac{(n+1) (2n-1)}{2 n}$, as well as the roots of the polynomials $g_n(k)$ and $h_n(k)$ which are quadratic in $k$. At these points, $\cW^{k_0}(\gs\gp_{2n}, f_{\text{prin}})$ cannot be obtained as a quotient of $\cW^{\mathrm{ev}}(c,\lambda)$, even though it is a well-defined vertex algebra.

\end{enumerate}

Suppose that $\cA^k, \cB^{\ell}$ are one-parameter vertex algebras that arise as the simple quotients of $\cW^{\mathrm{ev},I}(c,\lambda), \cW^{\mathrm{ev},J}(c,\lambda)$, respectively, and we wish to classify the coincidences between their simple quotients $\cA_{k_0}$ and $\cB_{\ell_0}$ at points $k_0, \ell_0$. By Corollary \ref{cor:intersections}, aside from the degenerate cases given by Theorem \ref{thm:coincidences}, the coincidences for which $k_0$ and $\ell_0$ are not poles of $c(k), \lambda(k)$, and $c(\ell), \lambda(\ell)$, respectively, correspond to the intersection points in $V(I) \cap V(J)$. On the other hand, if $k_0$ is a pole of $c(k)$ or $\lambda(k)$, and $\ell_0$ is a pole of $c(\ell)$ or $\lambda(\ell)$, but $\cA^{k_0}, \cB^{\ell_0}$ are still defined, Corollary \ref{cor:intersections} does not apply, and we need to address the question of whether or not $\cA_{k_0}$ and $\cB_{\ell_0}$ are isomorphic on a case-by-case basis.

\begin{thm} \label{thm:Dcoincidences}  For $3 \leq n <  m$, aside from the critical levels $\ell = -(2n-2)$ and $\ell' = -(2m-2)$, and the degenerate cases given by Theorem \ref{thm:coincidences}, all isomorphisms $\mathcal{W}_{\ell}(\mathfrak{s}\mathfrak{o}_{2n}, f_{\text{prin}})^{\mathbb{Z}_2} \cong \mathcal{W}_{\ell'}(\mathfrak{s}\mathfrak{o}_{2m}, f_{\text{prin}})^{\mathbb{Z}_2}$ appear on the following list.

\begin{equation} \begin{split} & \ell = -(2n-2) + \frac{2 n -1 }{2 (m + n-1)},\qquad \ell = -(2n-2) + \frac{2 (m + n -1)}{2 n -1},
\\ & \ell' =  -(2 m - 2) + \frac{2 m -1}{2 (m + n -1)},\qquad  \ell' =  -(2 m - 2) + \frac{2 (m + n -1)}{2 m -1}.\end{split} \end{equation} 
\end{thm} 

\begin{proof} First, we exclude the values of $\ell$ and $\ell'$ which are poles of the functions $\lambda_n(\ell)$ and $\lambda_m(\ell')$ given by \eqref{app:typedorb}, since at these values, $\mathcal{W}_{\ell}(\mathfrak{s}\mathfrak{o}_{2n}, f_{\text{prin}})^{\mathbb{Z}_2}$ and $\mathcal{W}_{\ell'}(\mathfrak{s}\mathfrak{o}_{2m}, f_{\text{prin}})^{\mathbb{Z}_2}$ are not quotients of $\cW^{\mathrm{ev}}(c,\lambda)$. For all other noncritical values of $\ell$ and $\ell'$, $\cW_{\ell}(\gs\go_{2n},f_{\text{prin}})^{\mathbb{Z}_2}$ and $\cW_{\ell'}(\gs\go_{2m},f_{\text{prin}})^{\mathbb{Z}_2}$ are obtained as quotients of $\cW^{\mathrm{ev},J_{n}}(c,\lambda)$ and $\cW^{\mathrm{ev}, J_{m}}(c,\lambda)$, respectively. 

By Corollary \ref{cor:intersections}, aside from the degenerate cases given by Theorem \ref{thm:coincidences}, all other coincidences $\mathcal{W}_{\ell}(\mathfrak{s}\mathfrak{o}_{2n}, f_{\text{prin}})^{\mathbb{Z}_2} \cong \mathcal{W}_{\ell'}(\mathfrak{s}\mathfrak{o}_{2m}, f_{\text{prin}})^{\mathbb{Z}_2}$ correspond to intersection points on the truncation curves $V(J_{n})$ and $V(J_{m})$. These ideals are described explicitly by \eqref{ideal:wslna}, and a calculation shows that 
$V(J_{n}) \cap V(J_{m})$ consists of exactly three points $(c,\lambda)$, namely, $(-24, -\frac{1}{245})$, $(\frac{1}{2}, -\frac{2}{49})$, and $(c_{n,m},\lambda_{n,m})$, where
\begin{equation} \begin{split}  c_{n,m} & =-\frac{m n (3 - 4 m - 4 n + 4 m n)}{m + n -1},
\\ \lambda_{n,m} & = - \frac{(m + n -1) f(n,m)}{7 (m -1 ) (n -1) (4 m n -1) g(n,m) h(n,m)},
\\ g(n,m) & = -10 + 19 m - 12 m^2 + 19 n - 6 m n - 8 m^2 n - 12 n^2 - 8 m n^2 + 8 m^2 n^2,
\\ h(n,m) & = 22 - 22 m - 22 n + 15 m n - 20 m^2 n - 20 m n^2 + 20 m^2 n^2,
\\ f(n,m) &  = 28 - 38 m - 14 m^2 + 24 m^3 - 38 n - 364 m n + 969 m^2 n - 696 m^3 n + 144 m^4 n 
\\ &- 14 n^2 + 969 m n^2 - 
1746 m^2 n^2 + 828 m^3 n^2 - 112 m^4 n^2 + 24 n^3 - 696 m n^3 + 828 m^2 n^3 
\\& - 32 m^3 n^3 - 64 m^4 n^3 + 144 m n^4 - 112 m^2 n^4 - 64 m^3 n^4 + 32 m^4 n^4.
\end{split}. \end{equation} 
The first two intersection points occur at degenerate values of $c$. By replacing the parameter $c$ with the levels $\ell, \ell'$ using \eqref{dcentralcharge}, we see that the third intersection point yields the nontrivial isomorphisms in Theorem \ref{thm:Dcoincidences}. Moreover, by Corollary \ref{cor:intersections}, these are the only such isomorphisms except possibly at the values of $\ell, \ell'$ excluded above. 

Finally, suppose that $\ell$ is a pole of the function $\lambda_n(\ell)$ given by \eqref{app:typedorb}. It is not difficult to check that the corresponding values of $\ell'$ for which $c_n(\ell) = c_m(\ell')$, are not poles of $\lambda_m(\ell')$. It follows that there are no additional coincidences at the excluded points. \end{proof}

\begin{remark} This family of coincidences was previously stated for the coset model of these orbifolds in \cite{CGKV}, Equations 4.11 and 4.12.
\end{remark}

\begin{thm} \label{thm:Ccoincidences} For $2 \leq n < m$, aside from the critical levels $\ell = -(n+1)$ and $\ell' = -(m+1)$, and the degenerate cases given by Theorem \ref{thm:coincidences}, all isomorphisms \begin{equation} \cW_{\ell}(\mathfrak{sp}_{2n}, f_{\text{prin}}) \cong \cW_{\ell'}(\mathfrak{sp}_{2m},f_{\text{prin}}),\end{equation} appear on the following list.

\begin{enumerate} 

\item $\displaystyle \ell = -(n + 1) + \frac{n}{2 (m + n)},\qquad \ell' = -(m + 1) + \frac{m}{2 (m + n)}$,

\smallskip

\item $\displaystyle \ell = -(n + 1) + \frac{1 + m + n}{1 + 2 n},\qquad \ell' = -(m + 1) + \frac{1 + m + n}{1 + 2 m}$,

\smallskip

\item $\displaystyle \ell = -(n + 1) + \frac{1 - 2 m + 2 n}{2 (-1 - 2 m + 2 n)},\qquad \ell' = -(m + 1) + \frac{1 - 2 n + 2 m}{2 (-1 - 2 n + 2 m)}$,

\smallskip

\item  $\displaystyle \ell = -(n + 1) + \frac{1 + 2 m + 2 n}{2 (2 n -1)},\qquad \ell' = -(m + 1) + \frac{1 + 2 m + 2 n}{2 (2 m -1)}$,

\smallskip

\item $\displaystyle \ell = -(n + 1) + \frac{1 + n}{1 + 2 m + 2 n},\qquad \ell' = -(m + 1) + \frac{1 + m}{1 + 2 m + 2 n}$.
\end{enumerate}

\end{thm}

\begin{proof} First, we exclude the values of $\ell$ and $\ell'$ which are poles of the functions $\lambda_n(\ell)$ and $\lambda_m(\ell')$ given by \eqref{app:typec}. For all other noncritical values of $\ell$ and $\ell'$, $\cW_{\ell}(\gs\gp_{2n}, f_{\text{prin}})$ and $\cW_{\ell'}(\gs\gp_{2m}, f_{\text{prin}})$ are obtained as quotients of $\cW^{\mathrm{ev},I_{n}}(c,\lambda)$ and $\cW^{\mathrm{ev}, I_{m}}(c,\lambda)$, respectively.

By Corollary \ref{cor:intersections}, aside from the degenerate cases given by Theorem \ref{thm:coincidences}, all other coincidences $\cW_{\ell}(\gs\gp_{2n}, f_{\text{prin}})\cong \cW_{\ell'}(\gs\gp_{2m}, f_{\text{prin}})$ correspond to intersection points on the truncation curves $V(I_{n})$ and $V(I_{m})$. These ideals are described explicitly by \eqref{app:idealtypec}, and a calculation shows that 
$V(I_{n}) \cap V(I_{m})$ consists of exactly eight points: $(-24, -\frac{1}{245})$, $(\frac{1}{2},\pm \frac{2}{49})$, which do not depend on $n$ and $m$, and five nontrivial points $(c^i_{n,m},\lambda^i_{n,m})$, $i=1,\dots, 5$, which appear explicitly in Appendix C. The first three intersection points occur at degenerate values of $c$. Replacing $c$ with $\ell, \ell'$ in the remaining intersection points using \eqref{ccentralcharge}, it follows from Corollary \ref{cor:intersections} that the above isomorphisms all hold, and that these are all such isomorphisms except possibly at the points excluded above. Finally, one checks that if $\ell$ is a pole of the function $\lambda_n(\ell)$ given by \eqref{app:typec}, the corresponding values of $\ell'$ for which $c_n(\ell) = c_m(\ell')$, are not poles of $\lambda_m(\ell')$. It follows that there are no additional coincidences at the excluded points. \end{proof}

\begin{thm} \label{thm:CDcoincidences} For $n\geq 2$ and $m\geq 3$, aside from the critical levels $\ell = -(n+1)$ and $\ell' = -(2m-2)$, and the degenerate cases given by Theorem \ref{thm:coincidences}, all isomorphisms 
$$\cW_{\ell}(\mathfrak{sp}_{2n}, f_{\text{prin}}) \cong \cW_{\ell'}(\mathfrak{so}_{2m}, f_{\text{prin}})^{\mathbb{Z}_2},$$ appear on the following list.

\begin{enumerate}

\item  $\displaystyle \ell = -(n+1) + \frac{m + n}{2 n},\quad \ell'  = -(2m-2) +\frac{m}{m + n}, \ \ell'  = -(2m-2) +\frac{m+n}{m}$,

\smallskip

\item  $\displaystyle \ell = -(n+1) + \frac{1 + 2 n}{2 ( 2 m + 2 n-1)},\quad \ell'  = -(2m-2) + \frac{2 (m-1)}{2 m + 2 n-1},\quad  \ell'= -(2m-2) + \frac{2 m + 2 n-1}{2 (m-1)}$,

\smallskip

\item  $ \displaystyle \ell = -(n + 1) + \frac{1 - m + n}{1 - 2 m + 2 n},\quad \ell' =  -(2m-2) + \frac{2 m - 2 n-1}{2 (m - n-1)},\quad \ell' =  -(2m-2) + \frac{2 (m - n-1)}{2 m - 2 n-1}.$
\end{enumerate}

For the third family, we must have $m \neq n+1$.
\end{thm}

\begin{proof} As above, we first exclude the values of $\ell$ and $\ell'$ which are poles of the functions $\lambda_n(\ell)$ and $\lambda_m(\ell')$ given by \eqref{app:typec} and \eqref{app:typedorb}, respectively. Aside from the degenerate cases given by Theorem \ref{thm:coincidences}, all other coincidences $\cW_{\ell}(\mathfrak{sp}_{2n}, f_{\text{prin}}) \cong \cW_{\ell'}(\mathfrak{so}_{2m}, f_{\text{prin}})^{\mathbb{Z}_2}$ correspond to intersection points on the truncation curves $V(I_{n})$ and $V(J_{m})$. We find that these curves have exactly five intersection points $(-24, -\frac{1}{245})$, $(\frac{1}{2},- \frac{2}{49})$, $(c^i_{n,m}, \lambda^i_{n,m})$, $i=1,2,3$, which appear explicitly in Appendix C. Replacing $c$ by $\ell, \ell'$ using \eqref{ccentralcharge} and \eqref{dcentralcharge}, shows that the above isomorphisms all hold, and that these are all such isomorphisms except possibly at the excluded values of $\ell, \ell'$. Finally, one checks as in the previous two theorems that there are no additional isomorphisms at the excluded values.  \end{proof}

\begin{remark} \label{newrationalW} In the third family above, suppose that $m=2n+2 + r$ for some integer $r$ satisfying $0\leq r \leq m-6$. Then we get $$\cW_{\ell}(\mathfrak{sp}_{2n}, f_{\text{prin}}) \cong \cW_{\ell'}(\mathfrak{so}_{2m}, f_{\text{prin}})^{\mathbb{Z}_2},\quad \ell = -(n+1) + \frac{1+n+r}{3+2n+2r},\quad \ell' = -(2m-2) + \frac{1+m+r}{m+r}.$$ Since $\ell$ is a nondegenerate admissible level for $\gs\gp_{2n}$, $\cW_{\ell}(\mathfrak{sp}_{2n}, f_{\text{prin}})$ is lisse and rational \cite{AIV,AV}. Since $\cW_{\ell'}(\mathfrak{so}_{2m}, f_{\text{prin}})$ is a simple current extension of $\cW_{\ell'}(\mathfrak{so}_{2m}, f_{\text{prin}})^{\mathbb{Z}_2}$, it is also lisse and rational despite the fact that $\ell'$ is not an admissible level for $\gs\go_{2m}$.

Similarly, in the third family above, suppose that $n = 2m -1 +r$ for some integer $r$ satisfying $0\leq r \leq n-5$. Then 
$$\cW_{\ell}(\mathfrak{sp}_{2n}, f_{\text{prin}}) \cong \cW_{\ell'}(\mathfrak{so}_{2m}, f_{\text{prin}})^{\mathbb{Z}_2},\quad \ell = -(n+1) + \frac{1+n+r}{2n+2r},\quad \ell' = -(2m-2) + \frac{2m+2r-1}{2m+2r}.$$ Since $\ell'$ is nondegenerate admissible for $\gs\go_{2m}$, both $\cW_{\ell'}(\mathfrak{so}_{2m}, f_{\text{prin}})$ and its orbifold $\cW_{\ell'}(\mathfrak{so}_{2m}, f_{\text{prin}})^{\mathbb{Z}_2}$ are lisse and rational. Therefore $\cW_{\ell}(\mathfrak{sp}_{2n}, f_{\text{prin}})$ is also lisse and rational, even though $\ell$ is not an admissible level for $\gs\gp_{2n}$.
\end{remark}

Recall the one-parameter vertex algebra $\cF^k(n) = \text{Com}(V^{k+1/2}(\gs\gp_{2n-2}), \cW^k(\gs\gp_{2n}, f_{\text{min}}))$ for $n\geq 2$ appearing in Theorem \ref{thm:otherquotients}, which by Conjecture \ref{conj:reform} is obtained as the simple quotient of $\cW^{\mathrm{ev},K_n}(c,\lambda)$ via the parametrization \eqref{conj:reformeq} of the curve $V(K_n)$. Also, recall that the specialization $\cF^{k_0}(n)$ of the one-parameter algebra $\cF^{k}(n)$ at $k = k_0$, can be a proper subalgebra of $\text{Com}(V^{k_0+1/2}(\gs\gp_{2n-2}), \cW^{k_0}(\gs\gp_{2n}, f_{\text{min}}))$, but this can only occur for rational numbers $k_0 \leq -n-\frac{1}{2}$. By abuse of notation, we shall use the same notation $\cF^{k}(n)$ if $k$ is regarded as a complex number rather than a formal parameter, so that $\cF^k(n)$ always denotes the specialization of the one-parameter algebra at $k\in \mathbb{C}$, even if it is a proper subalgebra of the coset. For all $k\in \mathbb{C}$, we denote by $\cF_{k}(n)$ the simple quotient of $\cF^k(n)$.

Assuming Conjecture \ref{conj:reform}, aside from the critical values and the degenerate cases, we can find all nontrivial isomorphisms $\cF_k(n) \cong \cW_{\ell}(\mathfrak{sp}_{2m}, f_{\text{prin}})$ and $\cF_k(n) \cong \cW_{\ell}(\mathfrak{so}_{2m}, f_{\text{prin}})^{\mathbb{Z}_2}$ by finding the intersection points $V(K_n) \cap V(I_{m})$ and $V(K_n) \cap V(J_{m})$.

\begin{conj} For $n\geq 2$ and $m\geq 2$, aside from the critical values $k = -(n+1)$, $k = -n -1/2$, and $\ell = -(m+1)$, and the degenerate cases given by Theorem \ref{thm:coincidences}, all isomorphisms $\cF_k(n) \cong \cW_{\ell}(\gs\gp_{2m}, f_{\text{prin}})$ appear in the following list.

\begin{enumerate}

\item $\displaystyle k = \frac{1}{2} (2 m-1),\qquad \ell = -(m+1) + \frac{m + n}{1 + 2 m + 2 n}$. This is the family of rational vertex algebras for all $m,n$ appearing in Conjecture \ref{conj:ackl}.

\smallskip

\item $\displaystyle k = \frac{1}{3} (m - 2 n -2),\qquad \ell =  -(m+1) +\frac{1 + m + n}{2 m + 2 n -1}$. These vertex algebras are rational for all $m,n$. 

\smallskip

\item $\displaystyle k = -\frac{2 (1 + m) (1 + n)}{3 + 2 m},\qquad \ell = -(m+1) + \frac{3 + 2 m}{2 (1 + 2 m - 2 n)}$,

\smallskip

\item $\displaystyle k = -\frac{2 m + 4 m n -1}{2 (2 m -1)},\qquad \ell = -(m+1) + \frac{2 m-1}{2 (2 m - 2 n -1)}$,

\smallskip

\item $\displaystyle k = -\frac{2 + 2 m + n + 2 m n}{2 (1 + m)}, \qquad \ell = -(m+1) + \frac{1 + m - n}{2 (1 + m)}$,

\smallskip

\item $\displaystyle k = -\frac{m - n + 2 m n -2}{2 (m -1)}, \qquad \ell = -(m+1) +\frac{m - n}{2 (m-1)}$. 
\end{enumerate}

\end{conj}

\begin{conj} For $n\geq 2$ and $m\geq 3$, aside from the critical values $k = -(n+1)$, $k = -n -1/2$, and $\ell = -(2m-2)$, and the degenerate cases given by Theorem \ref{thm:coincidences}, all isomorphisms $\cF_k(n) \cong \cW_{\ell}(\gs\go_{2m}, f_{\text{prin}})^{\mathbb{Z}_2}$ appear in the following list.

\begin{enumerate}

\item $\displaystyle k = \frac{1}{2} (m - n -2),  \qquad \ell = -(2m-2) + \frac{m + n -1}{m + n},  \qquad \ell = -(2m-2) + \frac{m + n}{m + n -1}$. These vertex algebras are rational whenever $n\geq m-1$.

\smallskip

\item $\displaystyle k = -\frac{1 + 4 m + 4 m n}{2 (1 + 2 m)}, \qquad \ell = -(2m-2)  +\frac{2 (m - n)}{1 + 2 m},   \qquad \ell = -(2m-2)  +\frac{1 + 2 m}{2 (m - n)}$,

\smallskip

\item $\displaystyle k =-\frac{m - 2 n + 2 m n -2}{2 m -3},\qquad \ell =  -(2m-2)  +\frac{2 (m - n -1)}{2 m -3},  \qquad \ell =  -(2m-2)  +\frac{2 m -3}{2 (m - n -1)}$.

\end{enumerate}

For the third family, we must have $m \neq n+1$.
\end{conj}

\begin{remark} In view of Theorem \ref{thm:minimalsp4}, these conjectures both hold in the case $n=2$, for all $m$. \end{remark}

\section{The relationship between $\cW(c,\lambda)$ and $\cW^{\mathrm{ev}}(c,\lambda)$}
In the notation of \cite{LI}, the universal two-parameter vertex algebra $\cW(c,\lambda)$ is freely generated of type $\cW(2,3,\dots)$, and has strong generators $\{L, W^i|\ i\geq 3\}$. Here $W^3$ is primary of weight $3$, $W^3_{(5)} W^3 = \frac{c}{3}$, and $W^i = W^3_{(1)} W^{i-1}$ for $i\geq 4$. It has a $\mathbb{Z}_2$-action determined by $L \mapsto L$ and $W^3 \mapsto -W^3$. It is natural to ask whether there is a homomorphism $\cW^{\mathrm{ev}}(c,\lambda) \ra \cW(c,\lambda)^{\mathbb{Z}_2}$. However, it is easy to see that this is false. First, $\cW(c,\lambda)$ has a unique up to scalar primary weight $4$ field 
$$\tilde{W}^4 = W^4 - \frac{32}{22 + 5 c} :LL:  - \frac{3 (c-2)}{2 (22 + 5 c)} \partial^2 L.$$ 
Then $\tilde{W}^4$ lies in $\cW(c,\lambda)^{\mathbb{Z}_2}$, and we consider the subalgebra of $\cW(c, \lambda)^{\mathbb{Z}_2}$ generated by $L$ and $\tilde{W}^4$. It is straightforward to verify that this subalgebra has $8$-dimensional weight $6$ subspace, with basis
$$:LLL:, \quad :W^3 W^3:, \quad :L W^4:,\quad  :(\partial^2 L)L:,\quad :(\partial L)^2:,\quad W^6, \quad \partial^2 W^4,\quad \partial^4 L.$$ One the other hand, the weight $6$ subspace of $\cW^{\mathrm{ev}}(c,\lambda)$ is only $7$-dimensional, with basis
$$:LLL:, \quad :L W^4:,\quad  :(\partial^2 L)L:,\quad :(\partial L)^2:,\quad W^6, \quad \partial^2 W^4,\quad \partial^4 L.$$

Instead, one can ask the following more refined question. Given a prime ideal $I \subseteq \mathbb{C}[c,\lambda]$ and a localization $R$ of $\mathbb{C}[c,\lambda] / I$, consider the quotient $\cW^{\mathrm{ev},I}_R(c,\lambda) / \cI$ by the maximal graded idea $\cI$, which is a one-parameter vertex algebra. Can we find another ideal $I' \subseteq \mathbb{C}[c,\lambda]$ and a localization $R'$ of $\mathbb{C}[c,\lambda] / I'$, such that $\cW^{\mathrm{ev},I}_R(c,\lambda) / \cI$ is isomorphic to the $\mathbb{Z}_2$-orbifold of $\cW^{I'}_{R'}(c,\lambda) / \cI'$, as a one-parameter vertex algebra? For this to happen, we must have a singular vector in weight $6$ in $\cW(c,\lambda)^{\mathbb{Z}_2}$, when $c$ is a free parameter. One can check by computer that there are exactly four truncation curves where this occurs, which correspond to the algebras $\cW^k(\gs\gl_n, f_{\text{prin}})$ for $n = 3,4,5$, and the parafermion algebra $N^k(\gs\gl_2) = \text{Com}(\cH, V^k(\gs\gl_2))$. For all $n\geq 3$, it can be verified using the explicit realization of $\cW^k(\mathfrak{sl}_n, f_{\text{prin}})$ as a quotient of $\cW(c,\lambda)$ that $\cW^k(\mathfrak{sl}_n, f_{\text{prin}})^{\mathbb{Z}_2}$ does not arise as a quotient of $\cW^{\mathrm{ev}}(c,\lambda)$, as a one-parameter vertex algebra. Therefore the only possible example of such a one-parameter vertex algebra is $N^k(\gs\gl_2)^{\mathbb{Z}_2}$. We now show that it is indeed such an example.

Recall that $N^k(\gs\gl_2)$ has central charge $\displaystyle c =  \frac{2 (k-1)}{k+2}$ and is of type $\cW(2,3,4,5)$ for all $k\neq 0$ \cite{DLY}. By Theorem 7.1 of \cite{LI}, $N^k(\gs\gl_2)$ can be obtained from $\cW(c,\lambda)$ as follows. 
Let $I \subseteq\mathbb{C}[c,\lambda]$ be the ideal generated by $$4\lambda (c+7) (2c-1) + (c-2) (c+4),$$ let $D$ be the multiplicative set generated by $(c+7)$ and $(2c -1)$, and let $$R = D^{-1} \mathbb{C}[c,\lambda] / I \cong D^{-1} \mathbb{C}[c].$$ Then $N^k(\gs\gl_2) \cong \cW^I_R(c,\lambda) / \cI$, where $\cI$ is the maximal proper graded ideal of $\cW^I_R(c,\lambda)$. In particular, $N^k(\gs\gl_2)$ is obtained from $\cW(c,\lambda)$ by setting
\begin{equation} \label{ratpara:para} c =  \frac{2 (k-1)}{k+2},\qquad \lambda = \frac{k+1}{(k-2) (3k+4)},\end{equation} and then taking the simple quotient. Note that the specialization $N^{k_0}(\gs\gl_2)$ of the one-parameter vertex algebra $N^k(\gs\gl_2)$ at a value $k = k_0$ coincides with $\text{Com}(\cH, V^{k_0}(\gs\gl_2))$ for all $k_0 \neq 0$. At $k_0 = 0$, $\text{Com}(\cH, V^{0}(\gs\gl_2))$ has an extra field in weight $1$, so the specialization $N^0(\gs\gl_2)$ is a quotient of $\cW^I(c,\lambda)$ even though the actual coset $\text{Com}(\cH, V^{0}(\gs\gl_2))$ is not.

The action of $\mathbb{Z}_2$ on $N^k(\gs\gl_2)$ is inherited from the above action of $\mathbb{Z}_2$ on $\cW(c,\lambda)$. The orbifold $N^k(\gs\gl_2)^{\mathbb{Z}_2}$ has recently been studied by Jiang and Wang \cite{JW}, and they have classified the irreducible, positive-energy modules for the simple quotient $N_k(\gs\gl_2)^{\mathbb{Z}_2}$, for $k$ a positive integer.

It can be shown by a computer calculation that as a one-parameter vertex algebra, $N^k(\gs\gl_2)^{\mathbb{Z}_2}$ is of type $\cW(2,4,6,8,10)$ and is generated by the weight $4$ field. In addition, there are no normally ordered relations among these generators in weight below $16$, so it follows from Theorem \ref{thm:simplequotient} that $N^k(\gs\gl_2)^{\mathbb{Z}_2}$ is the simple quotient of $\cW^{\mathrm{ev},I'}_{R'}(c,\lambda)$ for some $I'$ and $R'$. The following is straightforward to verify by computer.

\begin{thm} \label{paraorbcurve} The ideal $I'$ is generated by 
\begin{equation} \label{eq:paraorbcurve} 7\lambda (c-1) (22 + 5 c) (16 - c + c^2) + 64 + 6 c - 45 c^2 - 5 c^3.\end{equation}
In particular, $N^k(\gs\gl_2)^{\mathbb{Z}_2}$ can be obtained from $\cW^{\mathrm{ev}}(c,\lambda)$ by setting 
\begin{equation} \label{ratpara:para} c =  \frac{2 (k-1)}{k+2},\qquad \lambda = \frac{(2 + k) (-4 - 33 k - 15 k^2 + 4 k^3)}{7 (k-4) (17 + 16 k) (4 + 3 k + k^2)},\end{equation} and then taking the simple quotient. 
\end{thm}

By combining Theorems \ref{thm:prinw} and \ref{paraorbcurve}, we can find the coincidences among the simple quotients $N_k(\gs\gl_2)^{\mathbb{Z}_2}$ and $\cW_{\ell}(\mathfrak{so}_{2n}, f_{\text{prin}})^{\mathbb{Z}_2}$ for $n\geq 3$ by finding the intersection points of their truncation curves $V(I')$ and $V(J_n)$. The complete list of intersection points is the following. 

\begin{enumerate}
\item $ \displaystyle \big(-24, -\frac{1}{245}\big)$, 

\smallskip

\item  $\displaystyle \big( \frac{4n-1}{2n+1},\   \frac{(1 + 2 n) (-1 - 33 n - 60 n^2 + 64 n^3)}{14 (n-1) (17 + 64 n) (1 + 3 n + 4 n^2)}\big)$,

\smallskip

\item $\displaystyle \big(2-3n,\ -\frac{-16 + 78 n - 75 n^2 + 15 n^3}{7 (3 n-1) (15 n - 32) (2 - n + n^2)}\big)$,

\smallskip

\item  $\displaystyle \big(-\frac{2 n}{2n-3},\  \frac{(2 n-3) (-48 + 93 n - 45 n^2 + 4 n^3)}{7 (4 n-3) (17 n-33) (8 - 11 n + 4 n^2)}\big)$.

\end{enumerate}

\begin{cor} \label{cor:orbifoldcoincidences} For $n\geq 3$, aside from the critical levels $k = -2$ and $\ell = -(2n-2)$, and the degenerate cases given by Theorem \ref{thm:coincidences}, all isomorphisms $N_k(\gs\gl_2)^{\mathbb{Z}_2}\cong \cW_{\ell}(\mathfrak{so}_{2n}, f_{\text{prin}})^{\mathbb{Z}_2}$, appear on the following list.

\begin{enumerate}

\item  $\displaystyle k =4n,\qquad \ell = -(2n-2) +  \frac{2n}{2n+1},\qquad  \ell = -(2n-2) +  \frac{2n+1}{2n}$,
\smallskip

\item $\displaystyle k = -\frac{2 (n-1)}{n},\qquad  \ell = -(2n-2) + \frac{n-1}{n}, \qquad \ell = -(2n-2) + \frac{n}{n-1}$,

\smallskip

\item $\displaystyle k = \frac{1}{1 - n},\qquad \ell = -(2n-2) +  \frac{2n-3}{2n-2}, \qquad \ell = -(2n-2) +  \frac{2n-2}{2n-3}$.

\end{enumerate}
\end{cor}

The proof that this list is complete is similar to the proofs of Theorems \ref{thm:Dcoincidences} and \ref{thm:Ccoincidences}, and is omitted. Similarly, we can find the coincidences between $N_k(\gs\gl_2)^{\mathbb{Z}_2}$ and $\cW_{\ell}(\mathfrak{sp}_{2n}, f_{\text{prin}})$ for $n\geq 2$ by finding the intersection of their truncation curves $V(I')$ and $V(I_n)$. The complete list of intersection points is the following.

\begin{enumerate}
\item $\displaystyle \big(-24, -\frac{1}{245}\big),\ \big(\frac{1}{2}, \frac{2}{49}\big)$,

\smallskip

\item $\displaystyle \big(-1 - 6 n, \ -\frac{1 - 27 n - 60 n^2 + 60 n^3}{14 (1 + 3 n) (30 n -17) (1 + n + 2 n^2)}\big)$,

\smallskip

\item $\displaystyle \big(-\frac{2 + n}{n-1}, \  \frac{(n-1) (16 + 30 n - 33 n^2 + 2 n^3)}{7 (1 + 2 n) (17 n - 32) (2 - 3 n + 2 n^2)}\big)$,

\smallskip

\item $\displaystyle \big(-2 n, \ -\frac{16 - 3 n - 45 n^2 + 10 n^3}{ 7 (1 + 2 n) (5 n-11) (8 + n + 2 n^2)}\big)$,

\smallskip

\item $\displaystyle \big(\frac{4 n}{3 + 2 n},\ \frac{(3 + 2 n) (-24 - 51 n - 6 n^2 + 16 n^3)}{7 (2 n-3) (33 + 32 n) (4 + 5 n + 2 n^2)}\big)$.
\end{enumerate}

\begin{cor} \label{cor:orbifoldcoincidencestypec} For $n\geq 2$, aside from the critical levels $k = -2$ and $\ell = -(n+1)$, and the degenerate cases given by Theorem \ref{thm:coincidences}, all isomorphisms $N_k(\gs\gl_2)^{\mathbb{Z}_2} \cong \cW_{\ell}(\mathfrak{sp}_{2n}, f_{\text{prin}})$, appear on the following list.

\begin{enumerate} 

\item $\displaystyle k = - \frac{4 n}{1 + 2 n}, \qquad \ell = -(n+1) + \frac{n}{1 + 2 n}$,

\smallskip

\item $\displaystyle k = -\frac{2}{n}, \qquad \ell  = -(n+1) + \frac{n}{2 (n-1)}$,

\smallskip

\item $\displaystyle k = -\frac{2 n -1}{1 + n}, \qquad \ell  = -(n+1) +\frac{1 + n}{2 n-1}$,

\smallskip

\item $\displaystyle k = \frac{4 n}{3 + 2 n}, \qquad \ell  = -(n+1) +\frac{3 + 2 n}{2 (1 + 2 n)}$. 

\end{enumerate}

\end{cor}

\subsection{Type $A$ orbifolds and even spin algebras}
Even though $\cW^k(\mathfrak{sl}_n, f_{\text{prin}})^{\mathbb{Z}_2}$ does not arise as a quotient of $\cW^{\mathrm{ev}}(c,\lambda)$ as a one-parameter vertex algebra, there are certain special values of $k$ where the {\it simple} orbifold $\cW_k(\mathfrak{sl}_n, f_{\text{prin}})^{\mathbb{Z}_2}$ does indeed arise as a quotient of $\cW^{\mathrm{ev}}(c,\lambda)$. For example, the first family in Corollary \ref{cor:orbifoldcoincidences} has this property since $$N_{4n}(\gs\gl_2) \cong \cW_k(\mathfrak{sl}_{4n}, f_{\text{prin}}),\qquad k = -4 n + \frac{1 + 4 n}{2 + 4 n},$$ by Theorem 6.1 of \cite{ALY}. Similarly, the second family in Corollary \ref{cor:orbifoldcoincidences} also has this property for $n\geq 4$, since
$$N_{-2(n-1)/n}(\gs\gl_2) \cong \cW_k(\mathfrak{sl}_{n-1}, f_{\text{prin}}),\qquad k = -(n-1) + \frac{n-2}{n},$$ by Theorem 10.6 of \cite{LI}. Finally, the family in Corollary \ref{cor:orbifoldcoincidencestypec} has this property since
$$N_{-4 n/(1 + 2 n)}(\gs\gl_2) \cong \cW_k(\mathfrak{sl}_{2n}, f_{\text{prin}}),\qquad k = -2n + \frac{2n-1}{2n+1},$$ by Theorem 10.6 of \cite{LI}. 

\begin{cor} For $n\geq 3$, we have isomorphisms of simple vertex algebras
\begin{equation*} \begin{split} & \cW_k(\mathfrak{sl}_{4n}, f_{\text{prin}})^{\mathbb{Z}_2} \cong  N_{4n}(\gs\gl_2)^{\mathbb{Z}_2} \cong \cW_{\ell}(\mathfrak{so}_{2n}, f_{\text{prin}})^{\mathbb{Z}_2},
\\ & k = -4 n + \frac{1 + 4 n}{2 + 4 n},\quad \ell = -(2n-2) +  \frac{2n}{2n+1}.
\end{split} \end{equation*}

For $n\geq 4$, we have isomorphisms of simple vertex algebras
\begin{equation*} \begin{split} & \cW_k(\mathfrak{sl}_{n-1}, f_{\text{prin}})^{\mathbb{Z}_2} \cong  N_{-2(n-1)/n}(\gs\gl_2)^{\mathbb{Z}_2} \cong \cW_{\ell}(\mathfrak{so}_{2n}, f_{\text{prin}})^{\mathbb{Z}_2},
\\ & k = -(n-1) + \frac{n-2}{n},\quad \ell = -(2n-2) + \frac{n-1}{n}.
\end{split} \end{equation*}

For $n\geq 2$, we have isomorphisms of simple vertex algebras
\begin{equation*} \begin{split} & \cW_k(\mathfrak{sl}_{2n}, f_{\text{prin}})^{\mathbb{Z}_2} \cong  N_{-4 n/(1 + 2 n)}(\gs\gl_2)^{\mathbb{Z}_2} \cong \cW_{\ell}(\mathfrak{sp}_{2n}, f_{\text{prin}}),
\\ &k = -2n + \frac{2n-1}{2n+1},\quad \ell = -(n+1) + \frac{n}{1 + 2 n}.
\end{split} \end{equation*}
\end{cor}

\appendix

\section{The generator of the ideal $I_{n}$ and rational parametrization of $V(I_{n})$}

The explicit generator $p_{n}$ for the ideal $I_{n}\subseteq \mathbb{C}[c,\lambda]$ which gives rise to the type $B$ and $C$ principal $\cW$-algebras as quotients of $\cW^{\mathrm{ev}}(c,\lambda)$, is given by

\begin{equation} \label{app:idealtypec} \begin{split}
p_{n} & = f(c,n) + \lambda g(c,n) + \lambda^2 h(c,n),
\\  f(c,n) & = -204 c^2 - 192 c^3 + 171 c^4 + 952 c n - 4612 c^2 n + 2348 c^3 n - 38 c^4 n + 1568 n^2 - 7708 c n^2
\\ & + 1788 c^2 n^2 + 2401 c^3 n^2 - 74 c^4 n^2 + 560 n^3 - 18936 c n^3 + 22280 c^2 n^3 - 2112 c^3 n^3 +  8 c^4 n^3 
\\ & - 16304 n^4 + 18640 c n^4 + 3420 c^2 n^4 - 364 c^3 n^4 + 8 c^4 n^4 - 17408 n^5 + 27680 c n^5 - 10576 c^2 n^5 
\\ & + 304 c^3 n^5 - 3264 n^6 - 3072 c n^6 + 2736 c^2 n^6,
\\ g(c,n) & = - 14 (c-1) (2c-1) (22 + 5 c) (n-2) (n-1) (3 c + 10 n + 2 c n + 12 n^2) 
\\ & (5 c + 28 n + 2 c n + 40 n^2) ,
\\ h(c,n) & =  49 (c-1)^2 (22 + 5 c)^2 (21 c^2 + 70 c n - 14 c^2 n + 200 n^2 - 135 c n^2 - 26 c^2 n^2 + 380 n^3 
\\ & - 176 c n^3 + 8 c^2 n^3 + 436 n^4 + 132 c n^4 + 8 c^2 n^4 + 448 n^5 + 112 c n^5 +  336 n^6).
 \end{split} \end{equation}

The variety $V(I_{n}) \subseteq \mathbb{C}^2$ is a rational curve that admits the following two rational parametrizations with parameter $k$.
 
\begin{equation} \label{app:typec} \begin{split}  c_{n}(k) &  = -\frac{n (k + 2 n + 2 k n + 2 n^2) (-3 - 2 k + 4 k n + 4 n^2)}{1 + k + n},
\\  \lambda_{n}(k) & = -\frac{(1 + k + n) f_n(k)}{7 (1 + 2 n) (-1 - k + 2 k n + 2 n^2) (-1 + n + 2 k n + 2 n^2) g_n(k) h_n(k)},
\\  g_n(k) & = -10 - 21 k - 14 k^2 - 23 n - 32 k n - 16 n^2 + 8 k n^2 + 8 k^2 n^2 + 8 n^3 + 16 k n^3 + 8 n^4,
\\  h_n(k) & = -22 - 22 k - 22 n - 15 k n - 10 k^2 n - 30 n^2 - 50 k n^2 - 30 n^3 + 40 k n^3 + 40 k^2 n^3 
\\  & + 40 n^4 + 80 k n^4 + 40 n^5,
\\  f_n(k) & = -28 - 130 k - 170 k^2 - 68 k^3 - 222 n - 532 k n - 161 k^2 n + 228 k^3 n + 76 k^4 n - 186 n^2 
\\ & + 687 k n^2 + 1892 k^2 n^2 + 908 k^3 n^2 + 730 n^3 + 2780 k n^3 + 1508 k^2 n^3 - 800 k^3 n^3 - 320 k^4 n^3 
\\ & + 1106 n^4 + 536 k n^4 - 2520 k^2 n^4 - 1360 k^3 n^4 - 136 n^5 - 2608 k n^5 - 2064 k^2 n^5 + 128 k^3 n^5 
\\ &+ 64 k^4 n^5 - 888 n^6 - 1328 k n^6 + 384 k^2 n^6 + 256 k^3 n^6 - 304 n^7 + 384 k n^7 + 384 k^2 n^7 
\\ & + 128 n^8 + 256 k n^8 + 64 n^9.
\end{split} \end{equation}

\begin{equation} \label{app:typeb} \begin{split}  c_{n}(k) &  = -\frac{n (-3 + 2 k + 2 k n + 4 n^2) (k - 2 n + 2 k n + 4 n^2)}{k + 2 n -1},
\\  \lambda_{n}(k) & = -\frac{(k + 2 n -1)  f_n(k)}{7 (1 + 2 n) (-1 + k + k n + 2 n^2) (1 - 4 n + 2 k n + 4 n^2)  g_n(k) h_n(k)},
\\  g_n(k) & -20 + 19 k - 6 k^2 + 46 n - 36 k n + 4 k^2 n - 32 n^2 + 4 k^2 n^2 - 16 n^3 + 16 k n^3 + 16 n^4,
\\  h_n(k) & = 22 - 22 k - 44 n - 15 k n + 10 k^2 n + 30 n^2 - 50 k n^2 + 30 k^2 n^2 - 60 n^3 + 40 k n^3 + 20 k^2 n^3 
\\ & - 40 n^4 + 80 k n^4 + 80 n^5,
\\  f_n(k) & = 56 - 66 k - 2 k^2 + 12 k^3 - 500 n + 1080 k n - 941 k^2 n + 348 k^3 n - 36 k^4 n + 760 n^2
\\ & - 792 k n^2 - 322 k^2 n^2 +  514 k^3 n^2 - 100 k^4 n^2 + 1476 n^3 - 4024 k n^3 + 2948 k^2 n^3 - 480 k^3 n^3 
\\ & - 40 k^4 n^3 - 3656 n^4 + 3672 k n^4 + 240 k^2 n^4 - 600 k^3 n^4 + 40 k^4 n^4 + 496 n^5 + 3008 k n^5 
\\ & - 2256 k^2 n^5 + 192 k^3 n^5 + 16 k^4 n^5 + 2816 n^6 - 3104 k n^6 + 192 k^2 n^6 + 128 k^3 n^6 - 1344 n^7 
 \\ & - 256 k n^7 + 384 k^2 n^7 - 384 n^8 + 512 k n^8 + 256 n^9.
 \end{split} \end{equation}
The parametrization \eqref{app:typec} gives rise to the type $C$ algebra $\cW^k(\gs\gp_{2n}, f_{\text{prin}})$ at level $k$, and the parametrization \eqref{app:typeb} gives rise to the type $B$ algebra $\cW^k(\gs\go_{2n+1}, f_{\text{prin}})$ at level $k$.

\section{Rational parametrization of $V(J_{n})$}
The explicit generator $q_{n}$ for the ideal $J_{n}\subseteq \mathbb{C}[c,\lambda]$ which gives rise $\cW^{\ell}(\gs\go_{2n}, f_{\text{prin}})^{\mathbb{Z}_2}$ as a quotient of $\cW^{\mathrm{ev}}(c,\lambda)$, is given by \eqref{ideal:wslna}. The variety $V(J_{n})$ admits the following rational parametrization with parameter $\ell$:

\begin{equation} \label{app:typedorb} \begin{split}  c_n(\ell) &  =  -\frac{n (5 - 10 n + 4 n^2 - 2 \ell + 2 n \ell) (4 - 8 n + 4 n^2 - \ell + 2 n \ell)}{\ell + 2 n -2 },
\\  \lambda_{n}(\ell ) & = -\frac{(2 n + \ell -2) f_n(\ell)}{7 (n-1) (1 - 6 n + 4 n^2 + 2 n \ell) (2 - 8 n + 4 n^2 - \ell + 2 n \ell) g_n(\ell) h_n(\ell)},
\\  g_n(\ell) & = -40 + 22 n + 60 n^2 - 64 n^3 + 16 n^4 + 29 \ell - 40 n^2 \ell + 16 n^3 \ell - 6 \ell^2 - 4 n \ell^2 + 4 n^2 \ell^2,
\\  h_n(\ell) & = 44 + 56 n - 400 n^2 + 580 n^3 - 360 n^4 + 80 n^5 - 22 \ell - 65 n \ell + 220 n^2 \ell - 240 n^3 \ell 
\\ & + 80 n^4 \ell + 10 n \ell^2 - 30 n^2 \ell^2 + 20 n^3 \ell^2,
\\  f_n(\ell) & = -352 + 1608 n + 1200 n^2 - 11792 n^3 + 15800 n^4 - 3904 n^5 - 7296 n^6 + 6656 n^7 
\\ & - 2176 n^8 + 256 n^9 + 388 \ell - 992 n \ell - 
  5486 n^2 \ell + 17312 n^3 \ell - 13688 n^4 \ell - 2368 n^5 \ell \\ & + 7936 n^6 \ell - 
  3584 n^7 \ell + 512 n^8 \ell - 130 \ell^2 - 221 n \ell^2 + 4352 n^2 \ell^2 - 
  7772 n^3 \ell^2 + 2640 n^4 \ell^2 \\ & + 2784 n^5 \ell^2 - 2112 n^6 \ell^2 + 
  384 n^7 \ell^2 + 12 \ell^3 + 228 n \ell^3 - 1186 n^2 \ell^3 + 1200 n^3 \ell^3 + 
  160 n^4 \ell^3 \\ & - 512 n^5 \ell^3 + 128 n^6 \ell^3 - 36 n \ell^4 + 100 n^2 \ell^4 - 
  40 n^3 \ell^4 - 40 n^4 \ell^4 + 16 n^5 \ell^4.
\end{split} \end{equation}

\section{Intersection points}

Recall that the truncation curves  $V(I_{n})$ and $V(I_{m})$ intersect at exactly eight points. Three of these points $(-24, -\frac{1}{245})$ and $(\frac{1}{2}, \pm \frac{2}{49})$ do not depend on $n$ and $m$, and the remaining five points $(c^i_{n,m}, \lambda^i_{n,m})$, $i=1,\dots, 5$, are as follows.

\begin{equation} \begin{split} c^1_{n,m} & = -\frac{(2 m + n + 2 m n) (m + 2 n + 2 m n)}{m + n},
\\  \lambda^1_{n,m} & = - \frac{(m + n) f(n,m)}{7 (1 + 2 m) (1 + 2 n) (m + n + m n) g(n,m)h(n,m)},
\\ g(n,m) & = -6 m^2 - 5 m n + 4 m^2 n - 6 n^2 + 4 m n^2 + 4 m^2 n^2,
\\ h(n,m) & = -22 m + 10 m^2 - 22 n + 25 m n + 30 m^2 n + 10 n^2 + 30 m n^2 + 20 m^2 n^2,
\\ f(n,m) & = 12 m^3 - 36 m^4 + 70 m^2 n + 60 m^3 n - 100 m^4 n + 70 m n^2 + 211 m^2 n^2 + 2 m^3 n^2  - 40 m^4 n^2 
\\ & + 12 n^3 + 60 m n^3 + 2 m^2 n^3 + 40 m^4 n^3 - 36 n^4 - 100 m n^4 - 40 m^2 n^4 + 40 m^3 n^4 + 16 m^4 n^4.
\end{split} \end{equation}

\begin{equation} \begin{split}  c^2_{n,m} & = -\frac{m n (4 m n  - 2 m - 2 n -3)}{1 + m + n},
\\  \lambda^2_{n,m} & = -\frac{(1 + m + n) f(n,m)}{7 (2 m n - m -1) (2 m n -n -1) g(n,m) h(n,m)},
\\  g(n,m) & = -10 - 21 m - 14 m^2 - 21 n - 18 m n - 14 n^2 + 8 m^2 n^2,
\\   h(n,m) & = -22 - 22 m - 22 n - 15 m n - 10 m^2 n - 10 m n^2 + 20 m^2 n^2,
\\  f(n,m) & = -28 - 130 m - 170 m^2 - 68 m^3 - 130 n - 372 m n - 93 m^2 n + 228 m^3 n + 76 m^4 n 
\\ & - 170 n^2 - 93 m n^2 + 610 m^2 n^2 + 300 m^3 n^2 - 152 m^4 n^2 - 68 n^3 + 228 m n^3 + 300 m^2 n^3
\\ & - 192 m^3 n^3 - 16 m^4 n^3 + 76 m n^4 - 152 m^2 n^4 - 16 m^3 n^4 + 32 m^4 n^4.
\end{split} \end{equation}

\begin{equation} \begin{split}  c^3_{n,m} & = -\frac{4 m n (3 + 2 m + 2 n)}{(2 m - 2 n -1) (1 + 2 m - 2 n)},
\\  \lambda^3_{n,m} & = -\frac{(2 m - 2 n -1) (1 + 2 m - 2 n) f(n,m)}{7 (1 + 2 m) (1 + 2 n) (2 m + 2 n -1)g(n,m) h(n,m)},
\\  g(n,m) & = (5 - m + 6 m^2 - n - 16 m n + 6 n^2),
\\  h(n,m) & = (11 - 44 m^2 + 118 m n + 20 m^2 n - 44 n^2 + 20 m n^2),
\\  f(n,m) & = -7 + 23 m + 34 m^2 - 92 m^3 - 24 m^4 + 23 n + 108 m n + 142 m^2 n - 168 m^3 n + 72 m^4 n 
\\ & + 34 n^2 + 142 m n^2 + 344 m^2 n^2 - 88 m^3 n^2 - 92 n^3 - 168 m n^3 - 88 m^2 n^3 - 24 n^4 + 72 m n^4.
\end{split} \end{equation}

\begin{equation} \begin{split}  c^4_{n,m} & = -\frac{2 m n (3 + 2 m + 2 n + 4 m n)}{1 + 2 m + 2 n},
\\  \lambda^4_{n,m} & = -\frac{(1 + 2 m + 2 n) f(n,m)}{7 (1 + 2 m) (1 + 2 n) (1 + 2 m n) g(n,m) h(n,m)},
\\  g(n,m) & = (-10 - 21 m - 14 m^2 + 2 n - 10 m n - 12 n^2 + 8 m n^2 + 8 m^2 n^2),
\\  h(n,m) & = (-11 - 22 m - 22 n + 15 m n + 10 m^2 n + 10 m n^2 + 20 m^2 n^2),
\\  f(n,m) & = (-14 - 79 m - 136 m^2 - 68 m^3 + 18 n - 36 m n - 179 m^2 n - 228 m^3 n - 76 m^4 n 
\\ & + 104 n^2 + 274 m n^2 + 228 m^2 n^2 - 4 m^3 n^2 - 152 m^4 n^2 + 24 n^3 + 264 m n^3 + 260 m^2 n^3 
\\ & - 96 m^3 n^3 + 16 m^4 n^3 - 72 m n^4 - 128 m^2 n^4 + 48 m^3 n^4 + 32 m^4 n^4).
\end{split} \end{equation}

\begin{equation} \begin{split}  c^5_{n,m} & = -\frac{2 m n (3 + 2 m + 2 n + 4 m n)}{1 + 2 m + 2 n},
\\  \lambda^5_{n,m} & = -\frac{(1 + 2 m + 2 n) f(n,m)}{7 (1 + 2 m) (1 + 2 n) (1 + 2 m n)  g(n,m) h(n,m)},
\\  g(n,m) & = (-10 + 2 m - 12 m^2 - 21 n - 10 m n + 8 m^2 n - 14 n^2 +  8 m^2 n^2),
\\  h(n,m) & = (-11 - 22 m - 22 n + 15 m n + 10 m^2 n + 10 m n^2 +  20 m^2 n^2),
\\  f(n,m) & = (-14 + 18 m + 104 m^2 + 24 m^3 - 79 n - 36 m n + 274 m^2 n + 264 m^3 n  - 72 m^4 n 
\\ & - 136 n^2 - 179 m n^2 +  228 m^2 n^2 + 260 m^3 n^2 - 128 m^4 n^2 - 68 n^3  - 228 m n^3 
\\ & - 4 m^2 n^3 - 96 m^3 n^3 + 48 m^4 n^3 - 76 m n^4 - 152 m^2 n^4 + 16 m^3 n^4 + 32 m^4 n^4).
\end{split} \end{equation}

Finally, recall that the truncation curves $V(I_{n})$ and $V(J_{m})$ intersect at two points $(-24, -\frac{1}{245})$ and $(\frac{1}{2}, - \frac{2}{49})$ that do not depend on $n$ and $m$, and three additional points $(c^i_{n,m}, \lambda^i_{n,m})$, $i=1,2, 3$, which are as follows.

\begin{equation} \begin{split}  c^1_{n,m} & = -\frac{(2 m n -m - 2 n ) (m - n + 2 m n)}{m + n},
\\  \lambda^1_{n,m} & = -\frac{(m + n) f(n,m)}{7 (m-1) (1 + 2 n) (2 m n -m - n )  g(n,m) h(n,m)},
\\  g(n,m) & = (-7 m^2 - 7 m n - 6 n^2 - 4 m n^2 + 4 m^2 n^2),
\\  h(n,m) & = (-22 m - 5 m^2 - 22 n - 5 m n + 10 n^2 - 30 m n^2 + 20 m^2 n^2),
\\  f(n,m) & = (-34 m^3 + 19 m^4 - 68 m^2 n + 38 m^3 n - 22 m n^2 - 185 m^2 n^2 + 302 m^3 n^2 - 80 m^4 n^2 
\\ & + 12 n^3 - 204 m n^3 + 302 m^2 n^3 - 80 m^3 n^3 - 36 n^4 + 100 m n^4 - 40 m^2 n^4 - 40 m^3 n^4 + 16 m^4 n^4).
\end{split} \end{equation}

\begin{equation} \begin{split}  c^2_{n,m} & = -\frac{2 m n (4 m - 2 n + 4 m n -3)}{2 m + 2 n -1},
\\  \lambda^2_{n,m} & = -\frac{(2 m + 2 n -1) f(n,m)}{7 (2 m + 2 m n -1) (1 - 2 n + 4 m n) g(n,m) h(n,m)},
\\  g(n,m) & = -10 + 19 m - 12 m^2 + 2 n - 22 m n + 8 m^2 n - 12 n^2 - 8 m n^2 + 8 m^2 n^2,
\\  h(n,m) & = 11 - 22 m - 22 n - 15 m n + 20 m^2 n - 10 m n^2 + 20 m^2 n^2,
\\  f(n,m) & = 14 - 33 m - 2 m^2 + 24 m^3 - 74 n + 404 m n - 873 m^2 n + 696 m^3 n - 144 m^4 n + 80 n^2 
\\ & - 178 m n^2 - 260 m^2 n^2 + 452 m^3 n^2 - 112 m^4 n^2 + 24 n^3 - 264 m n^3 + 348 m^2 n^3 - 256 m^3 n^3 
\\ & + 64 m^4 n^3 + 72 m n^4 - 128 m^2 n^4 - 48 m^3 n^4 + 32 m^4 n^4.
\end{split} \end{equation}

\begin{equation} \begin{split}  c^3_{n,m} & = -\frac{m n(4 m - 2 n -3)}{(2 m - 2 n -1) (m - n -1)},
\\  \lambda^3_{n,m} & = - \frac{(2 m - 2 n -1) (m - n -1) f(n,m)}{7 (m-1) (2 m - n -1) (1 + 2 n) g(n,m) h(n,m)},
\\  g(n,m) & = 10 - 19 m + 12 m^2 + 21 n - 28 m n + 14 n^2,
\\  h(n,m) & = -22 + 66 m - 44 m^2 - 66 n + 73 m n + 20 m^2 n - 44 n^2 - 10 m n^2,
\\  f(n,m) & = 28 - 94 m + 62 m^2 + 52 m^3 - 48 m^4 + 186 n - 668 m n + 857 m^2 n - 504 m^3 n 
\\ & + 144 m^4 n + 430 n^2 - 1267 m n^2 + 1198 m^2 n^2 - 376 m^3 n^2 + 408 n^3 - 772 m n^3 + 304 m^2 n^3 
\\ & + 136 n^4 - 76 m n^4.
\end{split} \end{equation}

\end{document}